\documentclass[11pt]{amsart}
\DeclareRobustCommand{\VAN}[3]{#2}
\usepackage{tikz-cd}
\usepackage[numeric]{amsrefs}

\usepackage[margin=1in]{geometry} 
\usepackage{hyperref}
\usepackage{amsmath}
\usepackage{mathtools} 
\usepackage{amsfonts}
\usepackage{amssymb} 
\usepackage{amsthm}
\usepackage{footmisc}
\usepackage{mathrsfs}
\usepackage{calligra}
\usepackage{bbm}
\usepackage{enumitem}
\usepackage{parskip}
\usepackage{hyperref}
\mathtoolsset{showonlyrefs}
\usepackage{stmaryrd}
\usepackage{mathtools}
\newtheorem{Thm}[subsubsection]{Theorem}
\newtheorem{Lem}[subsubsection]{Lemma}
\newtheorem{Prop}[subsubsection]{Proposition}
\newtheorem{Cor}[subsubsection]{Corollary}
\newtheorem{MainConj}{Conjecture}
\newtheorem{Conj}[subsubsection]{Conjecture}
\newtheorem*{Thm*}{Theorem}
\newtheorem{mainThm}{Theorem}

\theoremstyle{definition}
\newtheorem{Def}[subsubsection]{Definition}
\newtheorem{Cons}[subsubsection]{Construction}

\newtheorem{Hyp}[subsubsection]{Hypothesis}
\newtheorem{Not}[subsubsection]{Notation}

\newtheorem{Question}[subsubsection]{Question}

\newtheorem{Vb}[subsubsection]{Example}
\newtheorem{Rem}[subsubsection]{Remark}

\numberwithin{equation}{subsection}

\DeclareMathOperator{\Hom}{\mathscr{H}\text{\kern -3pt {\calligra\large om}}\,}

\newcommand{\GL}{\operatorname{GL}}

\newcommand{\zz}{\mathbb{Z}}

\newcommand{\qqq}{\mathbb{Q}}
\newcommand{\gx}{(\mathsf{G}, \mathsf{X})}
\newcommand{\g}{\mathsf{G}}
\newcommand{\gad}{\mathsf{G}^{\mathrm{ad}}}
\newcommand{\gvx}{(\mathsf{G}_{V},\mathsf{H}_{V})}
\newcommand{\afp}{\mathbb{A}_{f}^{p}}
\newcommand{\gafp}{\mathsf{G}(\afp)}

\newcommand{\shg}{\operatorname{Sh}_{\g,U}}
\newcommand{\shginf}{\operatorname{Sh}_{\g,U_p}}
\newcommand{\shgv}{\operatorname{Sh}_{\mathsf{G}_V,\mathcal{U}}}

\newcommand{\shgu}{\operatorname{Sh}_{\g,U}}
\newcommand{\shgb}{\operatorname{Sh}_{\g,U,[b]}}

\newcommand{\ovfp}{\overline{\mathbb{F}}_p}
\newcommand{\fpbar}{\overline{\mathbb{F}}_{p}}
\newcommand{\gsc}{\g^{\mathrm{sc}}}
\newcommand{\gder}{\g^{\mathrm{der}}}

\newcommand{\ql}{\mathbb{Q}_{\ell}}

\newcommand{\zpbreve}{\breve{\mathbb{Z}}_{p}}
\newcommand{\qpbreve}{\breve{\mathbb{Q}}_{p}}
\newcommand{\zpbr}{\breve{\mathbb{Z}}_{p}}

\newcommand{\Fp}{\mathbb{F}_p}
\newcommand{\Zp}{\mathbb{Z}_p}
\newcommand{\Qp}{\mathbb{Q}_{p}}
\newcommand{\qp}{\mathbb{Q}_{p}}
\newcommand{\zp}{\mathbb{Z}_p}
\newcommand{\Gm}{\mathbb{G}_m}
\newcommand{\calO}{\mathcal{O}}
\newcommand{\calL}{\mathcal{L}}
\newcommand{\calM}{\mathcal{M}}
\newcommand{\calN}{\mathcal{N}}

\newcommand{\calT}{\mathcal{T}}
\newcommand{\calV}{\mathcal{V}}

\renewcommand{\hom}{{\mathrm{Hom}}}

\newcommand{\Spec}{\operatorname{Spec}}
\newcommand{\Spf}{\operatorname{Spf}}
\newcommand{\spf}{\operatorname{Spf}}
\newcommand{\perf}{{\mathrm{perf}}}
\newcommand{\sep}{{\mathrm{sep}}}
\newcommand{\tildex}{\tilde{\mathbb{X}}}
\newcommand{\tildeu}{\tilde{\Pi}}

\newcommand{\igCSgsp}{\operatorname{Ig}_{\mathrm{CS},\lambda}}
\newcommand{\iggsp}{\operatorname{Ig}_{\mathrm{M},\lambda}}
\newcommand{\cgsp}{C_{(Y,\lambda)}}
\newcommand{\cgspp}{C_{(Y',\lambda')}}
\newcommand{\igCS}{\operatorname{Ig}_{\mathrm{CS}}}
\newcommand{\ig}{\operatorname{Ig}_{\mathrm{M}}}
\newcommand{\defsus}{\mathbf{Def}_{\mathrm{sus}}}
\newcommand{\algo}[1]{\mathrm{Alg}_{#1}^{\mathrm{op}}}
\newcommand{\cb}{C_{\g,\llbracket b \rrbracket}}
\newcommand{\cbp}{C_{\g,\llbracket b' \rrbracket}}
\newcommand{\mfb}{\mathfrak{b}}
\newcommand{\mfa}{\mathfrak{a}}

\newcommand{\mfbp}{\mfb^{+}}
\newcommand{\mfap}{\mathfrak{a}^{+}}

\newcommand{\mbx}{\mathbb{X}}
\newcommand{\DLz}{Dieudonn\'e--Lie $\zpbreve$-algebra}
\newcommand{\DLzs}{Dieudonn\'e--Lie $\zpbreve$-subalgebra}

\newcommand{\DL}{Dieudonn\'e--Lie $\qpbreve$-algebra}

\usepackage{relsize}
\usepackage[bbgreekl]{mathbbol}
\usepackage{amsfonts}
\DeclareSymbolFontAlphabet{\mathbb}{AMSb}
\DeclareSymbolFontAlphabet{\mathbbl}{bbold}
\newcommand{\prism}{{\mathlarger{\mathbbl{\Delta}}}}
\newcommand{\spec}{\operatorname{Spec}}

\newcommand{\aut}{\mathbf{Aut}}
\author{Marco D'Addezio \and Pol van Hoften}
	\address{Institut de Recherche Mathématique Avancée (IRMA), Université de Strasbourg, 7 rue René-Descartes, 67000 Strasbourg, France}
	\email{daddezio@unistra.fr}
\address{Stanford mathematics department, 450 Jane Stanford, Way Building 380x, Stanford, CA 94305, USA}
\email{p.van.hoften@vu.nl}

\title{Hecke orbits on Shimura varieties of Hodge type}
\date{\today}
\begin{document}
\begin{abstract}
We prove the Hecke orbit conjecture of Chai--Oort for Shimura varieties of Hodge type at odd primes of good reduction. We use a novel result for the local monodromy groups of $F$-isocrystals “coming from geometry”, which refines Crew's parabolicity conjecture. In the course of the proof, we also introduce a noncommutative generalisation of Serre--Tate coordinates for formal neighbourhoods of central leaves, built upon the previous work of Caraiani–-Scholze and Kim. Using these coordinates, we reinterpret Chai--Oort's notion of strongly Tate-linear subspaces and we establish upper bounds for their monodromy groups. For this step, we employ the notion of Cartier--Witt stacks, as introduced by Drinfeld and Bhatt--Lurie. Another crucial ingredient in the proof is a rigidity result proved by Chai--Oort, which shows that the relevant subspaces are strongly Tate-linear. On the way, we generalise de Jong's full faithfulness theorem for $F$-isocrystals.
\end{abstract}
\maketitle 

\tableofcontents
\section{Introduction}
\subsection{The Hecke orbit conjecture}
Let $p$ be a prime number and $g$ a positive integer. Problem 15 on Oort's 1995 list of open problems in algebraic geometry, \cite{OortQuestionsAG}, is the following conjecture.
\begin{MainConj} \label{Conj:HO}
Let $x=(A_x, \lambda)$ be an $\ovfp$-point of the moduli space $\mathcal{A}_g$ of principally polarised abelian varieties of dimension $g$ over $\ovfp$. The Hecke orbit of $x$, consisting of all points $y \in \mathcal{A}_g(\ovfp)$ corresponding to principally polarised abelian varieties related to $(A_x, \lambda)$ by symplectic isogenies, is Zariski dense in the Newton stratum of $\mathcal{A}_g$ containing $x$. 
\end{MainConj}
This article contains a proof of this conjecture.
More generally, we prove that for the special fibre of a Shimura variety of Hodge type at an odd prime of good reduction, the isogeny classes are Zariski dense in the Newton strata containing them.

\subsubsection{} There is a refined version of Conjecture \ref{Conj:HO}, also due to Oort, which considers instead the prime-to-$p$ Hecke orbit of $x$, consisting of all $y \in \mathcal{A}_g(\ovfp)$ related to $x$ by prime-to-$p$ symplectic isogenies. In this case, the quasi-polarised $p$-divisible group $(A_x[p^{\infty}], \lambda)$ is constant on prime-to-$p$ Hecke orbits (not just constant up to isogeny). Therefore, the prime-to-$p$ Hecke orbit of $x$ is contained in the central leaf
\begin{align}
    C(x) = \left\{y\in \mathcal{A}_g(\ovfp)\: | \; A_y[p^{\infty}]\simeq_{\lambda} A_x[p^{\infty}]\right\},
\end{align}
where $\simeq_{\lambda}$ denotes a symplectic isomorphism. Oort proved in \cite{OortFoliation} that $C(x)$ is a smooth closed subvariety of the Newton stratum of $\mathcal{A}_g$ containing $x$. He also conjectured that the prime-to-$p$ Hecke orbit of $x$ is Zariski dense in the central leaf $C(x)$. This conjecture is known as the \emph{Hecke orbit conjecture} (for $\mathcal{A}_g$). Thanks to the Mantovan--Oort product formula, \cites{OortFoliation, MantovanPEL}, the Hecke orbit conjecture implies Conjecture \ref{Conj:HO}.

Central leaves and prime-to-$p$ Hecke orbits can also be defined for the special fibres of Shimura varieties of Hodge type at primes of good reduction, by work of Hamacher and Kim \cites{Hamacher, KimCentralLeaves}. The Hecke orbit conjecture for Shimura varieties of Hodge type then predicts that the prime-to-$p$ Hecke orbits of points are Zariski dense in the central leaves containing them (see \cite[Question 8.2.1]{KretShin} and \cite[Conjecture 3.2]{ChaiConjecture}).

The Hecke orbit conjecture naturally splits up into a discrete part and a continuous part. The discrete part states that the prime-to-$p$ Hecke orbit of $x$ intersects each connected component of $C(x)$, and the continuous part states that the Zariski closure of the prime-to-$p$ Hecke orbit of $x$ is equidimensional of the same dimension as $C(x)$. The discrete part of the conjecture is \cite[Theorem C]{KretShin} (see \cite{vHXiao} for related results). In this paper, we will focus on the continuous part of the conjecture.

\subsection{Main result}
Let $\gx$ be a Shimura datum of Hodge type with reflex field $\mathsf{E}$, and assume for simplicity that $\gad$ is $\mathbb{Q}$-simple throughout this introduction. Let $p>2$ be a prime such that $G=\g \otimes \qp$ is quasi-split and split over an unramified extension, let $U_p \subseteq G(\qp)$ be a hyperspecial subgroup, and let $U^{p} \subseteq \gafp$ be a sufficiently small compact open subgroup. Choose a place $v$ of $\mathsf{E}$ dividing $p$ and set $E=\mathsf{E}_{v}$. Let $\mathbf{Sh}_{U}\gx$ be the canonical model of Shimura variety for $\gx$ of level $U\coloneqq U^pU_p$ over $\mathsf{E}$. Let $\mathscr{S}_{U}\gx$ be the canonical integral model over $\mathcal{O}_{E}$ constructed in \cite{KisinPoints}, and let $\shg$ be its geometric special fibre. Let $C \subseteq \shg$ be a central leaf as constructed in \cite{Hamacher} (cf. \cite{KimCentralLeaves}). 
\begin{mainThm}[Theorem \ref{Thm:MainTheorem2}] \label{Thm:MainThmHO}
If $Z \subseteq C$ is a non-empty reduced closed subvariety that is stable under the prime-to-$p$ Hecke operators, then $Z=C$.
\end{mainThm}
When $\shg$ is a Siegel modular variety, this result (also for $p=2$) is due to Chai--Oort, see their forthcoming book \cite{ChaiOortBook} for the continuous part and \cite{ChaiOort} for the discrete part. Their proofs do not generalise to more general Shimura varieties because they rely on the existence of hypersymmetric points in Newton strata, which is usually false for Shimura varieties of Hodge type, see \cite{Fung}. Moreover, their proof of the continuous part of the conjecture relies on the fact that any point $x \in \shg(\ovfp)$ is contained in a large Hilbert modular variety, and they use work of Yu--Chai--Oort, \cite{YuChaiOort}, on the Hecke orbit conjecture for Hilbert modular varieties at (possibly ramified) primes. There are many other partial results, e.g. for prime-to-$p$ Hecke orbits of hypersymmetric points in the PEL case, \cite{LXiao}, or for prime-to-$p$ Hecke orbits of $\mu$-ordinary points \cites{ChaiOrdinarySiegel,ZhouMotivic, ShankarEtranges, MaulikShankarTang, OrdinaryHO}.

We also prove that isogeny classes are dense in the Newton strata containing them, see Theorem \ref{Thm:IsogenyClassesDenseNewtonStrata}; this proves (a generalisation of) Conjecture \ref{Conj:HO}. Moreover, we prove results about $\ell$-adic Hecke orbits for primes $\ell \not=p$ generalising work of Chai, \cite{ChaiElladicmonodromy}, in the Siegel case, see Theorem \ref{Thm:LadicHeckeOrbits}.
\begin{Rem}
The assumption that $\gad$ is $\mathbb{Q}$-simple, can be relaxed at the expense of introducing more notation, see Theorem \ref{Thm:MainTheorem2} for a precise statement. The assumption that $p>2$ is inherited from the work of Kim \cite{KimCentralLeaves}, and is not necessary for Siegel modular varieties. 
\end{Rem}
\begin{Rem}
Bragg--Yang prove a potentially good reduction criterion for K3 surfaces, see \cite[Theorem 8.10]{BraggYang}, conditional on the Hecke orbit conjecture for certain orthogonal Shimura varieties, see [\textit{ibid.}, Conjecture 8.2]. In Section \ref{Sec:BraggYang} we explain that our results can be used to prove this conjecture for $p>2$.
\end{Rem}
\subsection{Local monodromy of \texorpdfstring{$F$}{F}-isocrystals}

One of the main tools of the proof of Theorem \ref{Thm:MainThmHO} is the theory of \textit{monodromy groups of $F$-isocrystals}, defined in \cite{CrewIsocrystals}. We will prove a new result for the local monodromy groups of $F$-isocrystals “coming from geometry”, which should be of independent interest. This result is used to prove that the local monodromy groups of the crystalline Dieudonné modules of the universal $p$-divisible groups over $Z\subseteq C$ (notation as in Theorem \ref{Thm:MainThmHO}) are “big”. We explain it in a more general setting. 

\subsubsection{}
Let $X$ be a smooth connected variety over a perfect field with a closed point $x\in X$ and let $(\mathcal{M}^{\dagger},\Phi_{\mathcal{M}^{\dagger}})$ be a semi-simple overconvergent $F$-isocrystal over $X$ with constant Newton polygon. Since the Newton polygon is constant, the associated $F$-isocrystal $(\mathcal{M},\Phi_\calM)$ admits the slope filtration. The main result of \cite{AddezioII} tells us that the \textit{monodromy group} of $\mathcal{M}$ (Definition \ref{Def:MonodromyGroups}), denoted by $G(\calM,x)$, is the parabolic subgroup $P\subseteq G(\calM^{\dag},x)$ associated to the slope filtration. We refine that result for the local monodromy group at $x$. Let $X^{/x}$ be the formal completion of $X$ in $x$ and write $G(\calM^{/x},x)$ for the monodromy group of the restriction of $\calM$ to $X^{/x}$ (see Notation \ref{Not:WarningFormalMonodromy}).
\begin{mainThm}[{Theorem \ref{Thm:LocalGlobalMonodromyTheorem}}] \label{Thm:MainThmMonodromy}
The monodromy group $G(\calM^{{/x}},x)$ of the restriction of $\mathcal{M}$ to $X^{/x}$ is the unipotent radical of the monodromy group $G(\calM,x)$.
\end{mainThm}
When $\mathcal{M}$ is the crystalline Dieudonn\'e module of an ordinary $p$-divisible group, this result is proved by Chai in \cite{ChaiOrdinary}. Our proof builds instead on the techniques developed in \cite{AddezioII} and uses new descent results for isocrystals from \cites{DrinfeldCrystals,MathewDescent}. Since $X^{/x}$ is geometrically simply connected, each isocrystal underlying an isoclinic $F$-isocrystal over $X^{/x}$ is trivial by \cite[Theorem 2.4.1]{BerthelotMessing}. This already implies that $G(\calM^{/x},x)$ is unipotent. To relate $G(\calM^{/x},x)$ and the unipotent radical of $G(\calM,x)$, we pass through the respective generic points. 

Write $k$ for the function field of $X$ and $k_x$ for the function field of $X^{/x}$. We first prove in Theorem \ref{Thm:MonodromyGenericPoint} that passing from $X$ to $\spec k$ we do not change the monodromy group of $\calM$, as in the étale setting. Then we show that if we extend $k$ to $k^\sep$ the monodromy group of $\calM$ becomes the unipotent radical of $G(\calM,x)$ (Proposition \ref{Prop:monodromy-sep-closure}). This means that the extension of scalars from $k$ to $k^{\sep}$ kills precisely a Levi subgroup of $G(\calM,x)$. Subsequently, using the fact that the field extension $k\subseteq k_x$ is separable, we show that when we extend $F$-isocrystals from $k^{\sep}$ to $k_x^{\sep}$, their slope filtration does not acquire new splittings (Proposition \ref{Prop:non-splitting}). This is enough to prove that the local monodromy group $G(\calM^{{/x}},x)$ is the same as the monodromy group over $k^{\sep}$. By the previous part of the argument we then deduce that $G(\calM^{{/x}},x)$ is precisely the unipotent radical of $G(\calM,x).$

As an additional outcome of our analysis, we prove the following result of independent interest.

\begin{Thm}[Theorem \ref{Thm:FullFaithfulness}]
If $X$ is an irreducible Noetherian Frobenius-smooth (see Definition \ref{Def:FrobeniusSmooth}) scheme over $\Fp$ with generic point $\eta$ and $\calM$ is an isocrystal over $X$ such that $F^*\calM\simeq \calM$, then 
$$H^0_{\mathrm{cris}}(\eta,\calM_\eta)=H^0_{\mathrm{cris}}(X,\calM).$$ 
\end{Thm}

With this theorem we are also able to deduce the following consequence.

\begin{Cor}[Corollary \ref{Cor:ExistenceSlopeFiltration}] Let $X$ be a Noetherian Frobenius-smooth scheme over $\Fp$. If $(\calM,\Phi_\calM)$ is an $F$-isocrystal over $X$ with locally constant Newton polygon, then it admits the slope filtration.
\end{Cor}
\subsection{An overview of the proof of Theorem \ref{Thm:MainThmHO}} The overall structure of the proof of Theorem \ref{Thm:MainThmHO} is similar to the proof of the ordinary Hecke orbit conjecture in \cite{OrdinaryHO} and is based on a strategy implicit in the work of Chai--Oort and sketched to us by Chai in a letter. 

To explain the proof we first need to establish some notation. Let $\shgb \subseteq \shg$ be the Newton stratum containing the central leaf $C$, and let $Z \subseteq C$ be a reduced closed subvariety that is stable under the prime-to-$p$ Hecke operators, as in the statement of Theorem \ref{Thm:MainThmHO}. The Newton stratum $\shgb \subseteq \shg$ corresponds to an element $[b] \in B(G)$ which has an associated Newton (fractional) cocharacter $\nu_b$. Attached to this cocharacter is a parabolic subgroup $P_{\nu_b}$ with unipotent radical $U_{\nu_b}$.

Let $Z^{\mathrm{sm}}$ be the smooth locus of $Z$. Then \cite[Corollary 3.3.3]{OrdinaryHO} tells us that the monodromy group of the crystalline Dieudonné module $\calM$ of the universal $p$-divisible group over $Z^{\mathrm{sm}}$ is isomorphic to $P_{\nu_b}$. Theorem \ref{Thm:MainThmMonodromy} then tells us that for $x \in Z^{\mathrm{sm}}(\ovfp)$, the monodromy of $\calM$ over $Z^{/x}\coloneqq \spf \widehat{\mathcal{O}}_{Z,x}$ is equal to $U_{\nu_b}$. We are going to leverage this fact to show that $Z^{/x}=C^{/x}$, which will allow us to conclude that $Z^{\mathrm{sm}}$ and hence $Z$ is equidimensional of the same dimension as $C$.

For this we will construct generalised Serre--Tate coordinates on the formal completion $C^{/x}\coloneqq \spf \widehat{\mathcal{O}}_{C,x}$. To be precise, we will show that there is a \DLz\footnote{See Definition \ref{Def:DieudonneLie}.} $\mfap$ governing the structure of $C^{/x}$. For example, if $\shg$ is a Siegel modular variety and $C$ is the ordinary locus, then $C^{/x}$ is a $p$-divisible formal group by the classical theory of Serre--Tate coordinates, and $\mfap=\mathbb{D}(C^{/x})$ is its Dieudonn\'e module, equipped with the trivial Lie bracket. 

More generally the perfection of $C^{/x}$ admits the structure of a (functor to) nilpotent Lie $\qp$-algebra(s) whose Dieudonn\'e module is $\mfa\coloneqq \mfap[\tfrac{1}{p}]$. More canonically, the perfection of $C^{/x}$ is a trivial torsor for a unipotent formal group $\tilde{\Pi}(\mfa)$ related to the aforementioned nilpotent Lie algebra by the Baker--Campbell--Hausdorff formula. In the Siegel case, this unipotent formal group is the identity component of the group of self-quasi isogenies (compatible with the polarisation up to a scalar) of the $p$-divisible group $A_x[p^{\infty}]$. These results come from the perspective of Caraiani--Scholze, \cite{CaraianiScholze}, on $C$ and the perspective of Kim \cite{KimCentralLeaves}, on $C^{/x}$.

\begin{Rem}
    We note that our generalisation of Serre--Tate coordinates is different from the notion of a $p$-divisible cascade given by Moonen in \cite{MoonenSerreTate}. Our alternative definition is necessary in our work since in the Hodge type case the deformation spaces we consider generally do not have a cascade structure. Already in the $\mu$-ordinary case, all we can hope for is a \textit{shifted subcascade} in the sense of \cite{ShankarZhou} (cf. \cite{HongHodgeNewton}) and, as far as we can see, there is no way to run our arguments with shifted subcascades. 
\end{Rem}
\begin{Rem}
    We note that the \DL \vphantom{} $\mfa$ is often non-commutative even when $[b]$ is $\mu$-ordinary. For example, for the Picard modular surface associated to an imaginary quadratic field $E$ at a prime $p$ which is inert in $E$, the Lie algebra $\mfa$ is an extension of an abelian \DL \vphantom{} of slope $1$ by an abelian \DL \vphantom{} of slope $1/2$.
\end{Rem}
\begin{Rem}
    Our notion of generalised Serre--Tate coordinates was discovered independently by Chai--Oort, see \cite{ChaiOortRigidity} and \cite{LetterChaiToMoonen}. They use a slightly different perspective and different terminology, see Section \ref{Sec:Rigidity} for a comparison.
\end{Rem}

The Dieudonn\'e--Lie $\qpbreve$-algebra $\mfa$ turns out to be isomorphic to $\operatorname{Lie} U_{\nu_b}$ equipped with a natural $F$-structure coming from $[b]$. If $\mfb \subseteq \mfa=\operatorname{Lie} U_{\nu_b}$ is an $F$-stable Lie subalgebra, we construct a formally smooth closed formal subscheme $Z(\mfbp) \subseteq C^{/x}$, which is strongly Tate-linear in the sense of Chai--Oort, see \cite{ChaiFoliation}. The construction is such that when $\mfb = \mfa$, the formal subscheme $Z(\mfbp)$ is $C^{/x}$ itself. It follows from work of Chai--Oort, see \cite[Theorem 5.1]{ChaiOortRigidity}, that $Z^{/x}$ admits such a description.

\begin{Thm}[Chai--Oort, Theorem \ref{Thm:Rigidity}]\label{Thm:IntroRigidity}
There is an $F$-stable Lie subalgebra $\mfb \subseteq \mfa$ such that $Z^{/x} = Z(\mfbp)$.
\end{Thm}

To prove Theorem \ref{Thm:MainThmHO}, it then suffices to show that the formal subschemes $Z(\mfbp)$ for $\mfb \subsetneq \mfa$ have small monodromy. In this direction, we prove the following result.

\begin{Thm}[Theorem \ref{Thm:MonodromyBoundedAbove}]\label{Thm:IntroMonodromyBoundedAbove}
There is a natural closed immersion from the Lie algebra of the monodromy group of $\calM$ over $Z(\mfbp)$ to $\mfb$.
\end{Thm}
By the previous discussion, we know that the Lie algebra of the monodromy group of $\calM$ over $Z^{/x}$ is equal to $\mfa$. Therefore, if $Z^{/x}=Z(\mfbp)$ for some $\mfb$, then $\mfa \subseteq \mfb\subseteq \mfa$, so that $Z^{/x}=Z(\mfbp)= Z(\mfap)=C^{/x}$. 

The proof of Theorem \ref{Thm:IntroMonodromyBoundedAbove} uses the Cartier--Witt stacks of Drinfeld \cite{DrinfeldPrismatization} and Bhatt--Lurie \cite{BhattLurieII}. We use these stacks to “geometrize” the monodromy groups of isocrystals, which makes it possible to give a geometric proof that they are bounded above. The argument with Cartier--Witt stacks happens in Section \ref{Sec:MonodromyTateLinear}.

Theorem \ref{Thm:IntroRigidity} is related to rigidity results for $p$-divisible formal groups of Chai \cite{ChaiRigidity}, and to rigidity results for biextensions of $p$-divisible formal groups of Chai--Oort \cite{ChaiOortRigidityBiextentions}. We would also like to mention unpublished work of Tao Song which proves rigidity results in the case of $p$-divisible $4$-cascades. As mentioned above, we deduce Theorem \ref{Thm:IntroRigidity} from \cite[Theorem 5.1]{ChaiOortRigidity}.
\begin{Rem}
The unipotent formal group $\tildeu(\mfa)$ for $\mfa=\operatorname{Lie} U_{\nu_b}$ is closely related to the unipotent group diamond $\tilde{G}_b^{>0}$ introduced in Chapter III.5 of Fargues--Scholze \cite{FarguesScholze}, see \cite[proof of Lemma 11.25]{ZhangThesis} for a result in this direction. 
\end{Rem}
\begin{Rem}
The first online version of this article contained a proof of Theorem \ref{Thm:Rigidity}, inspired by \cite{ChaiOortRigidityBiextentions}; unfortunately, we discovered in 2024 that this proof had a gap. Chai and Oort independently proved the rigidity theorem in \cite{ChaiOortRigidity}, which appeared about two months after the first version of our article. We have decided to cite their work instead of trying to fix the gap in our original proof.
\end{Rem}
\subsection{Structure of the article}
In Section \ref{Sec:Background} we cover some background theory on $p$-divisible groups, isocrystals, and prove some technical results about formal schemes. We prove Theorem \ref{Thm:MainThmMonodromy} in Section \ref{Sec:MonodromyTheorem}. In Section \ref{Sec:AutoEndo} we discuss internal-Hom $p$-divisible groups and groups of (quasi-)\linebreak automorphisms of $p$-divisible groups, and introduce Dieudonn\'e--Lie algebras. In Section \ref{Sec:CentralLeaves} we discuss central leaves for Shimura varieties of Hodge type and describe their formal neighborhoods in terms of Dieudonn\'e--Lie algebras using work of Kim \cite{KimCentralLeaves}. We prove Theorem \ref{Thm:IntroMonodromyBoundedAbove} in Section \ref{Sec:MonodromyTateLinear} and in Section \ref{Sec:Rigidity} we explain how Theorem \ref{Thm:Rigidity} follows from the work of Chai--Oort. In Section \ref{Sec:MainTheorems} we state our main Theorem \ref{Thm:MainTheorem2}, a variant for isogeny classes (Theorem \ref{Thm:IsogenyClassesDenseNewtonStrata}) and a variant for $\ell$-adic Hecke orbits (Theorem \ref{Thm:LadicHeckeOrbits}). We also discuss some future directions and applications.

\subsection{Acknowledgements}
We are indebted to Ching-Li Chai and Frans Oort for the ideas they developed and we used for our work on the Hecke orbit conjecture. The first named author would also like to thank Ching-Li Chai for his letter on the relation between the parabolicity conjecture and the Hecke orbit conjecture. We are grateful to Brian Conrad, Sean Cotner, Andrew Graham, and Richard Taylor for many helpful discussions about this work. We also thank Hélène Esnault, Kentaro Inoue, and Matteo Tamiozzo for their comments on an earlier version of this article. Finally, we are grateful to the anonymous referee for their insightful comments, which have contributed to improving the presentation of this work.

The first named author was funded by the Deutsche Forschungsgemeinschaft (DA-2534/1-1, project ID: 461915680).
\newcommand{\nilp}{\operatorname{Nilp}_{\Zp}}%
\newcommand{\nilpo}{\operatorname{Nilp}_{\Zp}^\mathrm{op}}%

\section{Preliminaries} \label{Sec:Background} In this section we will introduce some notation and recall some definitions about Dieudonn\'e modules and isocrystals. First, we prove some preliminary technical result about formal schemes. 

\subsection{Formal schemes}

In this article we use the notion of \textit{formal schemes} as defined in \cite{EGA}. Note that in the affine case, they correspond to \textit{classical affine formal algebraic spaces} in the sense of \cite[Tag 0AID]{stacks-project}. The admissible topological rings we will mainly consider will be {adic rings} with a finitely generated ideal of definition.

If $R$ is an adic ring over a field $\kappa$ (not asked to be perfect) with ideal of definition $I$, we write $h_{\Spf R}:\algo{\kappa}\to \mathrm{Sets}$ for the associated functor of points, where $\algo{\kappa}$ is the opposite of the category of $\kappa$-algebras endowed with the fpqc topology. This functor of points is defined as
\begin{align}
    \varinjlim_n h_{\spec R/I^n};
\end{align}
it is an fpqc sheaf by \cite[Tag 0AI2]{stacks-project}. If $S$ is another adic ring, then by \cite[Lemma 0AN0]{stacks-project} a homomorphism of functors $f:h_{\spf S} \to h_{\spf R}$ comes uniquely from a continuous homomorphism $f:R \to S$. In the text, we implicitly identify $\spf R$ and $h_{\spf R}$.


\begin{Lem} \label{Lem:ClosedImmersion}
Let $(A, \mathfrak{m})$ and $(B, \mathfrak{n})$ be complete Noetherian local rings over $\kappa$ with residue field $\kappa$. Let $f:\Spf B \to \Spf A$ be a monomorphism of affine formal schemes over $\kappa$, then it is a closed immersion.
\end{Lem}
\begin{proof}
We have a local homomorphism of $\kappa$-algebras $f:(A, \mathfrak{m}) \to (B,\mathfrak{n})$ that induces an isomorphism on residue fields and such that for every $\kappa$-algebra $R$ the induced map $\hom(B,R) \to \hom(A,R)$ is injective.

For $j\geq 0$, let $R=B/(\mathfrak{m}B+\mathfrak{n}^j)$ and assume for the sake of contradiction that $\mathfrak{m}B + \mathfrak{n}^j$ is strictly contained in $\mathfrak{n}$. Then the natural map $B \to R$ does not factor through $B/\mathfrak{n}$. Note that there is another natural map $B \to R$ given by $$B \to B/\mathfrak{n} = A/\mathfrak{m} \to B/\mathfrak{m}B \to R,$$ and these two natural maps are different by our assumption. But they agree after precomposition with $A$, so this contradicts the injectivity of $\hom(B, R) \to \hom(A, R)$, and thus $\mathfrak{m}_B + \mathfrak{n}^j = \mathfrak{n}$.

By \cite[Tag 0AMS]{stacks-project}, the closure of the ideal $\mathfrak{m}B$ is equal to $\mathfrak{n}$ and it follows that the closure of the ideal $\mathfrak{m}^j B$ is equal to $\mathfrak{n}^j$. Therefore the map $f$ is taut and thus by \cite[Tag 0APU]{stacks-project} the topology on $B$ is given by the $\mathfrak{m}$-adic topology. Thus $B/\mathfrak{m}^j_B$ has the discrete topology and it follows that the following diagram is cartesian
\begin{equation*}
    \begin{tikzcd}
    \spec B/\mathfrak{m}^j_B \arrow{r} \arrow{d} & \spf B \arrow{d} \\
    \spec A/\mathfrak{m}^j \arrow{r} & \spf A.
    \end{tikzcd}
\end{equation*}
This tells us that $f:\spf B \to \spf A$ is representable, and \cite[Tag 0GHZ]{stacks-project} allows us to conclude that $f$ is a closed immersion.
\end{proof}

\subsection{Frobenius-smooth schemes and the Tannakian category of isocrystals}
In the study of crystals, there is a special class of schemes with a particularly good behaviour.

\begin{Def}\label{Def:FrobeniusSmooth}We say that a scheme $X$ over $\mathbb{F}_p$ is \textit{Frobenius-smooth} if the absolute Frobenius $F:X \to X$ is syntomic. 
\end{Def}
The condition of being Frobenius-smooth is equivalent to $X$ being Zariski locally of the form $\operatorname{Spec} B$ where $B$ has a \textit{finite (absolute) $p$-basis}, \cite[Lemma 2.1.1]{DrinfeldCrystals}, which means that there exist elements $x_1, \cdots, x_n$ such that every element $b \in B$ can be uniquely written as
\begin{align}
    b=\sum b_{\alpha}^p x_{1}^{\alpha_1} \cdots  x_n^{\alpha_n},
\end{align}
where $0 \le \alpha_i \le p-1$ and $b_{\alpha} \in B$. 

The main examples of Frobenius-smooth schemes that we will encounter are smooth schemes over perfect fields and power series rings over perfect fields. If $X$ is a Noetherian Frobenius-smooth scheme, then it is regular by a result of Kunz (see \cite[Tag 0EC0]{stacks-project}).

\begin{Def}
If $X$ is a scheme over $\Fp$, we denote by $\operatorname{Isoc}(X)$ the category of crystals in \textit{coherent} $\mathcal{O}$-modules on the absolute crystalline site of $X$ with $\hom$-sets tensored by $\qqq$. When $X$ is irreducible, Noetherian, and Frobenius-smooth we define \begin{align}
    \kappa\coloneqq \bigcap_{i=1}^{\infty} \Gamma(X, \mathcal{O}_X)^{p^n}
\end{align}
to be the \textit{field of constants} of $X$. Note that $\kappa$ is a field thanks to \cite[Section 3.1.2]{DrinfeldCrystals} and when $X$ is in addition a geometrically connected scheme of finite type over a perfect field, $\kappa$ coincides with the base field. We also write $K$ for $W(\kappa)[\tfrac{1}{p}]$. 
\end{Def}
\begin{Prop}[{\cite[Corollary 3.3.3]{DrinfeldCrystals}}]\label{Prop:DrinfeldTannakianCategory}

If $X$ is an irreducible Noetherian Frobenius-smooth scheme, then $\operatorname{Isoc}(X)$ is a $K$-linear Tannakian category.
\end{Prop}

This allows us to define the monodromy groups of isocrystals in this situation.

\begin{Def}\label{Def:MonodromyGroups} Let $X$ be an irreducible Noetherian Frobenius-smooth scheme and let $\calM$ be an isocrystal over $X$. We define $\langle \calM \rangle$ to be the Tannakian subcategory of $\operatorname{Isoc}(X)$ generated by $\calM$. If $\xi$ is an $\Omega$-point of $X$ for some perfect field $\Omega$, we define $G(\calM,\xi)$ to be the Tannaka group of $\langle\calM \rangle$ with respect to the fibre functor induced by $\xi$. We call it the \textit{monodromy group of $\calM$} with respect to $\xi$.
\end{Def}

\begin{Not}\label{Not:WarningFormalMonodromy}
In Theorem \ref{Thm:MainThmMonodromy}, Section \ref{Sec:LocalMonodromyTheorem}, and Section \ref{Sec:MonodromyTateLinear} we also need to consider the monodromy group of the restriction of an isocrystal $\calM$, defined on a smooth variety $X$ over a perfect field of characteristic $p$, to $X^{/x}=\spf \widehat{\mathcal{O}}_{X,x}$ where $x\in X$ is a closed point. For simplicity, in this situation, we will rather consider the restriction of $\calM$ via the morphism $\spec \widehat{\mathcal{O}}_{X,x}\to X$, denoted by $\calM^{/x}$. Therefore in these sections we will treat $X^{/x}$ as a scheme rather than an affine formal scheme. If $\xi$ is a perfect point of $\spec \widehat{\mathcal{O}}_{X,x}$ we then define $G(\calM^{/x},\xi)$ as the monodromy group of $\calM^{/x}$ over $\spec \widehat{\mathcal{O}}_{X,x}$ (note that $\spec \widehat{\mathcal{O}}_{X,x}$ is Frobenius-smooth, so that Definition \ref{Def:MonodromyGroups} applies). Arguing as in \cite[Proposition 2.4.8]{DeJongDieudonne}, one can show that this monodromy group is the same as the one constructed considering the category of $F$-isocrystals over the formal scheme $\spf \widehat{\mathcal{O}}_{X,x}$.
\end{Not}

\subsection{\texorpdfstring{$p$}{p}-divisible groups and Dieudonn\'e theory} \label{Sec:PrelimBT}

We recall here the notions of $p$-divisible group and Dieudonné module. For simplicity, for the latter, we restrict ourselves to some special classes of \textit{quasi-syntomic} rings.
\begin{Def}
	For an $\Fp$-algebra $R$, a \textit{$p$-divisible group} over $R$ is an fpqc sheaf of abelian groups $Y$ over $\algo{R}$ such that the following conditions are satisfied:
	
	\begin{itemize}
		\item[{(1)}]  $Y(A)=\bigcup_{n=1}^\infty Y(A)[p^n]$ for every $A\in \algo{R}$.
		\item[(2)] $[p]:Y\to Y$ is surjective.
		\item[(3)] $Y[p]$ is representable by a finite locally-free group scheme.
	\end{itemize}
If $Y$ is a $p$-divisible group, we write $Y^\vee$ for its \textit{Cartier dual}, $T_p Y$ for the \textit{Tate module} of $Y$, defined as the fpqc sheaf 
    $\varprojlim_n Y[p^n]$, and $\tilde{Y}$ for the \textit{universal cover} of $Y$, defined as
\begin{align}
    \tilde{Y} \coloneqq  \varprojlim_{[p]} Y,
\end{align}
where the transition maps are given by multiplication by $p$. If $R$ is adic with ideal of definition $I$, a \textit{$p$-divisible group over $\spf R$} is a compatible family of $p$-divisible groups $Y_n$ over $R/I^n$. Thanks to \cite[Lemma 2.4.4]{DeJongDieudonne},  a $p$-divisible group over $\spf R$ defines a unique $p$-divisible group over $R.$

\end{Def}

\begin{Def} \label{Def:Quassyntomic}
We say that a scheme $S$ is \textit{quasi-syntomic} if the relative cotangent complex $L \Omega^1_{S/\mathbb{F}_p}$ has Tor-amplitude in $[-1,0]$. We also say that $T\to S$ is a \textit{quasi-syntomic cover } if the relative cotangent complex $L \Omega^1_{T/S}$ has Tor-amplitude in $[-1,0]$ and $T \to S$ is an fpqc cover. We call an $\mathbb{F}_p$-scheme \emph{quasi-regular semiperfect (qrsp)} if it is quasi-syntomic and semiperfect. For a perfect field $\kappa$ of characteristic $p$ we denote by $\mathrm{QRSP}_\kappa\subseteq \mathrm{Alg}_\kappa$ the full subcategory of \textit{qrsp $\kappa$-algebras}, namely those algebras with qrsp spectrum.
\end{Def}

For $R\in \mathrm{QRSP}_{\Fp}$, let $A_{\mathrm{cris}}(R)$ be Fontaine's ring of crystalline periods (see \cite[Proposition 4.1.3]{ScholzeWeinstein}) with $\varphi:A_{\mathrm{cris}}(R) \to A_{\mathrm{cris}}(R)$ induced by the absolute Frobenius on $R$, and let $B^+_{\mathrm{cris}}(R)$ be  $A_{\mathrm{cris}}(R)[\tfrac{1}{p}]$.
\begin{Def} \label{Def:DieudonneI}
A \emph{Dieudonn\'e module} over $R$ is a pair $(M^+, \varphi_{M^+})$, where $M^+$ is a finite locally free $A_{\mathrm{cris}}(R)$-module and 
\begin{align}
    \varphi_{M^+}:  M^+[\tfrac 1p] \to M^+[\tfrac 1p]
\end{align}
is a semilinear bijection such that  \begin{align}
    M^+ \subseteq \varphi_{M^+} (M^+) \subseteq \tfrac{1}{p} M^+.
\end{align}
A \emph{rational Dieudonn\'e module} $(M, \varphi_M)$ over $R$ is a $\varphi$-module over $B^+_{\mathrm{cris}}(R)$ which is isomorphic to $(M^+, \varphi_{M^+})[\tfrac{1}{p}]$ for some  {Dieudonn\'e module} $(M^+, \varphi_{M^+})$. \end{Def}

\begin{Thm}[Ansch\"utz--Le Bras]\label{Thm:DieudonneFullyFaithful}
	If $R$ is a qrsp $\Fp$-algebra, there exists a fully faithful functor $$\mathbb{D}\colon \{ \textrm{$p$-divisible groups over $R$}  \} \to \{ \textrm{Dieudonné modules over $R$}  \},$$ which is functorial in $R$. 
\end{Thm}
\begin{proof}
In {\cite[Definition 4.35]{AnschutzLeBras}} the authors attach to a $p$-divisible group over a quasi-syntomic ring (without assumptions on the characteristic) a \textit{prismatic Dieudonné crystal}. In their case, the induced functor is a contravariant functor. In [\textit{ibid.}, Theorem 4.44] they prove that for a qrsp $\Fp$-algebra $R$ their construction agrees with the standard contravariant Dieudonné module functor.

In particular, when $R$ is qrsp, the prismatic Dieudonné crystal of a $p$-divisible group $Y$ over $R$ is a finite locally free $A_{\mathrm{cris}}(R)$-module $M_\prism(Y)$ endowed with a semilinear bijection
\begin{align}
    \varphi_{\prism}:  M_\prism(Y)[\tfrac 1p] \to M_\prism(Y)[\tfrac 1p]
\end{align}
such that
\begin{align}
    pM_\prism(Y) \subseteq \varphi_{\prism}M_\prism(Y) \subseteq  M_\prism(Y).
\end{align}

We define $\mathbb{D}(Y)$ to be the module $M_\prism(Y^\vee)$ endowed with $\varphi_\prism/p.$ Thanks to [\textit{ibid.}, Theorem 4.82], the (covariant) functor $\mathbb{D}$ is fully faithful.
\end{proof}

\begin{Rem}
 Note that our choice of Frobenius structure on $\mathbb{D}(Y)$ differs from the standard definition of the covariant Dieudonné module functor. Rather than using the usual assignment $$Y\mapsto(M_\prism(Y^\vee),\varphi_\prism),$$
we adopt conventions that align with those in \cite{CaraianiScholze}. With our normalisation, $\mathbb{D}(\qp/\mathbb{Z}_p)$ is $A_{\mathrm{cris}}(R)$ equipped with the standard Frobenius and $\mathbb{D}(\mu_{p^{\infty}})$ is $A_{\mathrm{cris}}(R)$ equipped with the Frobenius divided by $p$. This normalisation also ensures that the contravariant Dieudonn\'e module of $Y$, viewed as an $F$-isocrystal, is isomorphic to the dual of $\mathbb{D}(Y)$. Finally, if $Y^{(1)}$ is the Frobenius twist of $Y$, the morphism $\varphi_\prism/p$ induces a ${B}_{\mathrm{cris}}^+(R)$-linear isomorphism $\mathbb{D}(Y^{(1)})[\tfrac{1}{p}]\to \mathbb{D}(Y)[\tfrac{1}{p}]$. The geometric Frobenius  $F_Y\colon Y\to Y^{(1)}$ is sent by $\mathbb{D}$ to the inverse of this linearisation.
\end{Rem}
\begin{Rem} \label{Rem:SlopeReversal}
    If $R=\ovfp$, we have that $A_{\mathrm{cris}}(R)=W(R)=\zpbreve$. A rational Dieudonn\'e module $(M, \varphi_M)$ over $\ovfp$ gives rise to a $\varphi$-isocrystal over $\qpbreve=\zpbreve[\tfrac{1}{p}]$, whose slopes are contained in $[-1,0]$. For a $p$-divisible group $Y$ over $\ovfp$, the slopes are contained in $[0,1]$ and they are of the form $-\lambda$, where $\lambda$ is a slope of $\mathbb{D}(Y)[\tfrac{1}{p}]$. 
\end{Rem}

We will also need the following variant of Theorem \ref{Thm:DieudonneFullyFaithful}.

\begin{Thm}[Berthelot--Messing]\label{Thm:BMDieudonne}
Let $R$ be a normal Frobenius-smooth $\Fp$-algebra. There exists a fully faithful functor $$\mathbb{D}(-)[\tfrac{1}{p}]\colon \{ \textrm{$p$-divisible groups over $R$}  \}_{\mathbb{Q}} \to F\textrm{-}\operatorname{Isoc}(\spec(R)),$$ where the source is the category of $p$-divisible groups up to isogeny.
\end{Thm}
\begin{proof}
    This is a special case of \cite[Théorème 7]{BerthelotMessing}. We construct $\mathbb{D}$ from Berthelot--Messing's contravariant functor using the same normalisation as in Theorem \ref{Thm:DieudonneFullyFaithful}.
\end{proof}
Let us recall now some well-known results on the representability of $p$-divisible groups.
\begin{Def}\label{Def:FormalLieVariety}
    We say that an affine formal scheme over $\ovfp$ is a \textit{formal Lie variety} if it is isomorphic to $\Spf \ovfp [[V]]$ for some finite dimensional $\ovfp$-vector space $V$. We say that $\dim(V)$ is the dimension of such a formal scheme.
\end{Def}

\subsubsection{Representability}\label{Sec:Representability}
Let $Y$ be a connected $p$-divisible group over $\ovfp$. Thanks to \cite[Chapter II, Theorem 2.1.8]{Messing}, $Y$ is representable by a formal Lie variety. By \cite[Proposition 4.1.2.(4)]{CaraianiScholze}, its Tate module $T_p Y$ is representable by an affine scheme isomorphic to the spectrum of 
\begin{align} \label{eq:qrsp}
    \ovfp[X_1^{1/p^{\infty}}, \cdots, X_m^{1/p^{\infty}}]/(X_1, \cdots, X_m)
\end{align}
for some $m\geq 0$. This is the classical example of a quasi-regular semiperfect scheme. We define
\begin{align}
    \ovfp[[X_1^{1/p^{\infty}}, \cdots, X_m^{1/p^{\infty}}]]
\end{align}
to be the $(X_1, \cdots, X_m)$-adic completion of $\ovfp[X_1^{1/p^{\infty}}, \cdots, X_m^{1/p^{\infty}}]$. Thanks to \cite[Proposition 3.1.3.(iii)]{ScholzeWeinstein}, the universal cover $\tilde{Y}$ is (non-canonically) representable by $$\spf \ovfp[[X_1^{1/p^{\infty}}, \cdots, X_m^{1/p^{\infty}}]].$$ From this discussion, it follows that both $\tilde{Y}$ and $T_p Y$ are determined by their restriction to $\mathrm{QRSP}_{\ovfp}^{\mathrm{op}}\subseteq\algo{\ovfp}$. Note that we can describe them explicitly as sheaves over $\mathrm{QRSP}_{\ovfp}^{\mathrm{op}}$ thanks to the following lemma.
\begin{Lem} \label{Lem:IntegralPHodgeTheoryTateModuleUniversalCover}
There is a commutative diagram of fpqc sheaves over $\mathrm{QRSP}_{\ovfp}^{\mathrm{op}}$, which evaluated at $R\in \mathrm{QRSP}_{\ovfp}^{\mathrm{op}}$ gives
\begin{equation} \label{eq:IntegralDieudonneDiagram}
    \begin{tikzcd}
    T_p Y(R) \arrow[d, hook] \arrow[r, "\simeq"] & \left(A_{\mathrm{cris}}(R) \otimes_{\zpbreve} \mathbb{D}(Y)\right)^{\varphi=1} \arrow{d} \\
    \tilde{Y}(R) \arrow[r, "\simeq"] & \left(B_{\mathrm{cris}}^+(R) \otimes_{\qpbreve} \mathbb{D}(Y)[\tfrac{1}{p}] \right)^{\varphi=1},
    \end{tikzcd}
\end{equation}
where $\varphi$ is given by the diagonal Frobenius.
\end{Lem}
\begin{proof}
For $R\in \mathrm{QRSP}_{\ovfp}^{\mathrm{op}}$, by Theorem \ref{Thm:DieudonneFullyFaithful} there exists a natural isomorphism
\begin{align}
    T_p Y(R) &= \operatorname{Hom}_R((\qp/\mathbb{Z}_p)_R,Y_R)\\
    &{\xrightarrow{\sim}} \operatorname{Hom}_{A_{\mathrm{cris}},\varphi}(A_{\mathrm{cris}}(R), A_{\mathrm{cris}}(R) \otimes_{\zpbreve} \mathbb{D}(Y)) \\
    &= \left(A_{\mathrm{cris}}(R) \otimes_{\zpbreve} \mathbb{D}(Y)\right)^{\varphi=1},
\end{align}
where the latter identification is induced by evaluation at $1$. Similarly, after inverting $p$, we get a natural isomorphism
\begin{align}
    \tilde{Y}(R) &= \operatorname{Hom}_R((\qp/\mathbb{Z}_p)_R,Y_R)[\tfrac{1}{p}]\\
    &{\xrightarrow{\sim}} \left(B_{\mathrm{cris}}^{+}(R) \otimes_{\qpbreve} \mathbb{D}(Y)[\tfrac{1}{p}] \right)^{\varphi=1}.
\end{align}
The induced diagram commutes by construction. 
\end{proof}

\subsubsection{Complete slope divisibility} We conclude this section with some recollections on the notion of complete slope divisibility.

\begin{Def}\label{Def:CompletelySlopeDivisible}
For a perfect ring $R$ we say that an isoclinic Dieudonné module $(M^+,\varphi_{M^+})$ over $R$ is \textit{completely slope divisible} if there exist integers $s$ and $a$ with $s\neq 0$ such that $\varphi_{M^+}^sM^+=p^aM^+$. We also say that a Dieudonné module $(M^+,\varphi_{M^+})$ over $R$ is \textit{completely slope divisible} if it is the direct sum of isoclinic completely slope divisible Dieudonné modules, and we say that a $p$-divisible group is \textit{completely slope divisible} if the associated Dieudonné module is so.
\end{Def}

\begin{Rem} \label{Rem:OortZink}
   Since we assumed $R$ perfect, the definition we gave is equivalent to the usual one thanks to \cite[Proposition 1.3]{OortZink}. Note also that by [\textit{ibid.}, Corollary 1.5], if $R$ is an algebraically closed field, an isoclinic Dieudonné module is completely slope divisible if and only if it is defined over a finite field.
\end{Rem}

\begin{Lem}\label{Lem:CSDDieudonneOvfp}
    A Dieudonné module over $\ovfp$ is completely slope divisible if and only if it is a direct sum of isoclinic Dieudonné modules.
\end{Lem}
\begin{proof}
    By \cite[Corollary 1.5]{OortZink} it is enough to prove that every isoclinic Dieudonné module is defined over a finite field. By Dieudonné theory, this follows from the fact that a $p$-divisible group which is geometrically isogenous to a $p$-divisible group defined over a finite field, is itself defined over a finite field. 
\end{proof}
\begin{Lem} \label{Lem:SaturedSubCSD}
    If $(M^+,\varphi_{M^+})$ is a completely slope divisible Dieudonné module over $\ovfp$ and $(N^+,\varphi_{N^+})$ is a Dieudonné submodule of $(M^+,\varphi_{M^+})$ such that $M^+/N^+$ is torsion-free, then $(N^+,\varphi_{N^+})$ is completely slope divisible.
\end{Lem}
\begin{proof}By Lemma \ref{Lem:CSDDieudonneOvfp}, we have to prove that $(N^+,\varphi_{N^+})$ is a direct sum of isoclinic Dieudonné modules. By the Dieudonné--Manin classification, there exists a Dieudonné submodule $( \bar N^+,\varphi_{\bar N^+})\subseteq (N^+,\varphi_{N^+})$ of finite index which decomposes into a direct sum $\oplus_{\lambda\in\mathbb{Q}} \bar N_\lambda^+$ of isoclinic Dieudonné modules. For an element $x\in N^+$, there exist by assumption $x_\lambda\in M_\lambda^+$ for $\lambda \in \mathbb{Q}$ almost all $0$ such that $x=\sum_{\lambda\in \mathbb{Q}}x_\lambda$. Since $\bar N^+\subseteq N^+$ is of finite index, there exists $n$ big enough such that $p^nx\in \bar N^+$, so that $p^nx_\lambda\in \bar N^+$ for every $\lambda$. This implies that $p^nx_\lambda\in N^+$ for every $\lambda$ and by the assumption that $M^+/N^+$ is torsion-free, we deduce that each $x_\lambda$ lies in $N^+$. This yields the desired result.
\end{proof} 
\section{Local monodromy of \texorpdfstring{$F$}{F}-isocrystals}
\label{Sec:MonodromyTheorem}
The main goal of this section is to prove Theorem \ref{Thm:MainThmMonodromy}. As it often happens, to show the relation between the two Tannaka groups in the statement we first find an equivalent categorical condition. We do this in Section \ref{Sec:TannakianCriterion}, where we prove a quite general Tannakian criterion to check that a unipotent subgroup of an algebraic group is the entire unipotent radical.

After that, in Section \ref{Sec:GenericPoint}, we prove that the global monodromy group of an $F$-isocrystal with constant Newton polygon is the same as its “generic” monodromy group (Theorem \ref{Thm:MonodromyGenericPoint}). This will be essential to reduce the entire problem to a problem of $F$-isocrystals defined over (imperfect) fields with finite $p$-basis. In Section \ref{Sec:Descent} we prove then some descent results, notably the descent of splittings of the slope filtration for separable field extensions with finite $p$-basis (Proposition \ref{Prop:non-splitting}), and in Section \ref{Sec:LocalMonodromyTheorem} we put all the ingredients together and we prove Theorem \ref{Thm:MainThmMonodromy}.

\subsection{Two Tannakian criterions}\label{Sec:TannakianCriterion} Let $K$ be a field and let $V$ be a finite-dimensional $K$-vector space. We recall the following well-known lemma.
\begin{Lem}\label{Lem:prep-Tann-crit}
If $U, U'$ are unipotent subgroups of $\operatorname{GL}(V)$, then the following two properties are equivalent.
\begin{enumerate}
\item[\normalfont{(i)}]$U'\subseteq U$.
\item[\normalfont{(ii)}]$W^U\subseteq W^{U'}$ for every algebraic representation $W$ of $\operatorname{GL}(V)$. 

\end{enumerate}
\end{Lem}
\begin{proof}
	The implication (i) $\Rightarrow$ (ii) is obvious. To prove (ii) $\Rightarrow$ (i) we just note that by Chevalley's theorem there is a representation $W$ of $\operatorname{GL}(V)$ such that $U$ is the stabiliser of a line $L\subseteq W$. Since $U$ does not admit non-trivial characters, we deduce that $L\subseteq W^U$. By (ii) this implies that $U'$ fixes $L$ and this yields the desired result.
\end{proof}

Thanks to this lemma, we can deduce two Tannakian criterions that we will use later on. For an algebraic group $G$ over a perfect field we will write $R_u(G)$ for the unipotent radical of $G$. For a cocharacter $\nu:\Gm\to \operatorname{GL}(V)$ we write $P_\nu$ for the parabolic subgroup attached to $\nu$ and $U_{\nu}$ for $R_u(P_{\nu})$.

\begin{Prop}\label{Prop:Tann-crit} Suppose $K$ is of characteristic $0$ and let $\nu:\Gm\to \operatorname{GL}(V)$ be a cocharacter and $U \subseteq G \subseteq P_\nu$ a chain of subgroups of the parabolic subgroup $P_\nu$. If $U$ and $R_u(G)$ are contained in $U_\nu$, then the following two properties are equivalent.
\begin{enumerate}
	\item[\normalfont{(i)}] $U=R_u(G)$.
	\item[\normalfont{(ii)}] For every representation $W$ of $\operatorname{GL}(V)$, the group $G$ stabilises $W^U$ and the induced representation factors through $G/R_u(G)$. 
\end{enumerate}
\end{Prop}
\begin{proof}
The implication (i) $\Rightarrow$ (ii) follows from the observation that $R_u(G)$ is normal in $G$. For (ii) $\Rightarrow$ (i) first note that by the assumptions we have that $U\subseteq R_u(G)$ since $R_u(G)=G\cap U_\nu$\footnote{Note that unipotent groups $U$ and $G\cap U_\nu$ are connected since we are working in characteristic $0$.}. For the other inclusion, thanks to (ii) we deduce that for every representation $W$ we have that $W^U\subseteq W^{R_u(G)}$. Thus by Lemma \ref{Lem:prep-Tann-crit}, we conclude that $R_u(G)\subseteq U$.
\end{proof}

\begin{Prop}\label{Prop:SecondTannakianCriterion}
    Let $U\subseteq U'$ be an inclusion of unipotent subgroups of $\operatorname{GL}(V)$. The following facts are equivalent.
        \begin{enumerate}
\item[\normalfont{(i)}]$U=U'$.
\item[\normalfont{(ii)}]$\mathrm{Ext}^1_{U'}(\mathbbm 1,W')\to \mathrm{Ext}^1_{U}(\mathbbm 1,W'|_{U}) $ is injective for every algebraic representation $W'$ of $U'$. 

\end{enumerate}
\end{Prop}
\begin{proof}
    It is clear that (i)$\Rightarrow$ (ii). For the other implication we want to use Lemma \ref{Lem:prep-Tann-crit}. For an algebraic representation $W$ of $\GL(V)$ we consider the \textit{socle filtration} $W_\bullet$ relative to $U'$, namely we define inductively $W_{0}=W^{U'}$ and for $i\geq 1$ we define $W_i$ as the preimage of $(W/W_{i-1})^{U'}$ via the quotient $W\to W/W_{i-1}$. The spaces $W_i$ are stabilised by $U'$. Suppose by contradiction that $W^U$ is not contained in $W^{U'}$. Then there exists a morphism $f\colon \mathbbm 1\to W$ that is $U$-equivariant, with image contained in some $W_i$ for some $i\geq 1$, and such that the induced morphism $\bar f\colon \mathbbm 1\to W_i/W_{i-1}$ is non-zero. This defines a non-trivial extension $$0\to W_{i-1}\to E\to \mathbbm 1\to 0$$ of $U'$-representations which becomes trivial after restriction to $U$. The existence of such an extension contradicts (ii). This yields the desired result.
\end{proof}

\subsection{Passing to the generic point}\label{Sec:GenericPoint}

Let $X$ be an irreducible Noetherian Frobenius-smooth scheme over $\Fp$ with generic point $\eta$ and let $\calM$ be an isocrystal over $X$ such that $F^*\calM\simeq \calM$. In this section we want to prove two results that compare $\calM$ with its restriction to the generic fibre $\calM_\eta$. For this purpose, we want to use the following theorem by de Jong.
\begin{Thm}[{\cite[Theorem 1.1]{DeJongHomomorphisms}}]\label{Thm:deJong}If $X=\spec(A)$ is affine with $A$ a DVR and $(\calM^+,\Phi_{\calM^+})$ is a free $F$-crystal over $X$ of finite rank, then for every $m,n\in \zz\times \zz_{>0}$ we have 
$$H^0_{\mathrm{cris}}(\eta,\calM_\eta^+)^{F^n=p^m}=H^0_{\mathrm{cris}}(X,\calM^+)^{F^n=p^m}.$$
\end{Thm}

\begin{Cor}\label{Cor:deJong}
If $X$ is as in Theorem \ref{Thm:deJong} and $\calM$ is an isocrystal over $X$ such that $F^*\calM\simeq \calM$, then
$$H^0_{\mathrm{cris}}(\eta,\calM_\eta)=H^0_{\mathrm{cris}}(X,\calM).$$
\end{Cor}
\begin{proof}
    Let $(\calM^+,\Phi_{\calM^+})$ be an $F$-crystal such that $\calM^+[\tfrac{1}{p}]=\calM$. Since $X$ is of dimension $1$, by \cite[Lemma 0B3N]{stacks-project}, after possibly replacing $\calM^+$ with its double dual we can assume that $\calM^+$ is free. We may further assume that the field of constants of $A$ is algebraically closed thanks to \cite[§3]{DeJongHomomorphisms}. By the Dieudonné--Manin classification both $H^0_{\mathrm{cris}}(\eta,\calM_\eta)$ and $H^0_{\mathrm{cris}}(X,\calM)$ are generated by vectors $v$ such that $\Phi_\calM^n(v)=p^mv$ for some $m,n\in \zz\times \zz_{>0}$. The result then follows from Theorem \ref{Thm:deJong}.
\end{proof}

We first want to extend de Jong's theorem to more general irreducible Noetherian Frobenius-smooth schemes by using a Hartogs' argument. Thanks to \cite[Section 1.1]{BerthelotMessing}, a ring $B$ with a $p$-basis admits a $p$-adic lift $\tilde{B}\twoheadrightarrow B$. By [\textit{ibid.}, Proposition 1.3.3], the datum of a crystal in quasi-coherent $\calO$-modules over $\spec(B)$ is then equivalent to the datum of a completed $\tilde{B}$-module $M^+$ and topologically $p$-nilpotent derivations of $M^+$ associated to some choice of a lift of a $p$-basis of $B$ to $\tilde{B}$. We need the following lemma.

\begin{Lem}\label{Lem:ACartesianSquare} Let $f\colon A\to A'$ be an injective morphism of Noetherian Frobenius-smooth rings which sends $p$-bases to $p$-bases, and write $\tilde{f}\colon\tilde{A}\to \tilde{A'}$ for a $p$-adic lift of $f$. For an isocrystal $\calM$ over $\spec(A)$ write $\calM'$ for the pullback to $\spec(A')$ and $M,M'$ for the associated modules over $\tilde{A}[\tfrac{1}{p}]$ and $\tilde{A}'[\tfrac{1}{p}]$. The following diagram is cartesian
    \begin{equation}
    \begin{tikzcd}
    H^0_{\mathrm{cris}}(\spec(A),\calM)\arrow{d} \arrow{r}& M\arrow{d} \\
    H^0_{\mathrm{cris}}(\spec(A'),\calM')\arrow{r}& M'.
    \end{tikzcd}
\end{equation}
\end{Lem}

\begin{proof}By \cite[Proposition 3.5.2]{DrinfeldCrystals}, the module $M$ is projective, which implies that $M\to M'$ is injective. We choose a lift $\{\tilde{x}_1,\dots,\tilde{x}_n\}\subseteq \tilde{A}$ of a $p$-basis of $A$. By the assumption, this is sent by $\tilde{f}$ to a lift of a $p$-basis of $A'$. This choice then defines differential operators $\partial_1,\dots,\partial_n$ of $M'$ that stabilise $M\subseteq M'$. By \cite[Proposition 1.3.3]{BerthelotMessing}, this implies that $H^0_{\mathrm{cris}}(\spec(A'),\calM')\subseteq M'$ (resp. $H^0_{\mathrm{cris}}(\spec(A),\calM)\subseteq M$) is the subspace of elements killed by $\partial_1,\dots,\partial_n$. This ends the proof.\end{proof}

\begin{Thm}\label{Thm:FullFaithfulness}
If $X$ is an irreducible Noetherian Frobenius-smooth scheme over $\Fp$ and $\calN$ is a subquotient of an isocrystal $\calM$ over $X$ such that $F^*\calM\simeq \calM$, then 
$$H^0_{\mathrm{cris}}(\eta,\calN_\eta)=H^0_{\mathrm{cris}}(X,\calN).$$ 
\end{Thm}
   
\begin{proof} By \cite[Theorem 5.10]{UniversalExtensions}, we know that $\calN$ is a subobject of some isocrystal $\calM'$ such that $F^*\calM'\simeq \calM'$. Thanks to [\textit{ibid.}, Lemma 5.6], it is then enough to prove the result for an isocrystal $\calM$ such that $F^*\calM\simeq \calM$. By Zariski descent, we may further assume that $X=\spec(A)$ is affine.

Let $\tilde A$ a $p$-adic lift of $A$ and let $\tilde{A}_\eta$ be a $p$-adic lift of $\mathrm{Frac}(A)$ equipped with a morphism $\tilde A\hookrightarrow \tilde{A}_\eta$ lifting the inclusion $A\subseteq \mathrm{Frac}(A)$. We write $S$ for the set of prime ideals $\mathfrak p$ of $\tilde A$ of codimension $1$ containing $p$ and for $\mathfrak p\in S$ we write $\tilde{A}_\mathfrak{p}\subseteq \tilde{A}_\eta$ for the $p$-adic completion of the localisation of $\tilde A$ at $\mathfrak p$. By construction, we have that $\tilde{A}_{\mathfrak p}/p=A_{\mathfrak p}$. This implies that $\bigcup_{\mathfrak p \in S}\tilde{A}_\mathfrak{p}\subseteq \tilde{A}_\eta$ is dense with respect to the $p$-adic topology since $\bigcup_{\mathfrak p \in S}{A}_\mathfrak{p}=\mathrm{Frac}(A).$ 

We first want to prove that the ring $$\tilde{B}\coloneqq \bigcap_{\mathfrak{p}\in S} \tilde{A}_\mathfrak{p}$$ is equal to $\tilde A$. To do this, we first note that $\tilde B$ is $p$-adically complete and $p$-torsion-free since each $\tilde{A}_\mathfrak{p}$ is so. In addition, by the $p$-torsion-freeness, we have that $p\tilde{B}=\bigcap_{\mathfrak{p}\in S} p\tilde{A}_\mathfrak{p}$, which implies that the morphism $$\tilde B/p\to \bigcap_{\mathfrak{p}\in S} \tilde{A}_\mathfrak{p}/p\subseteq \tilde{A}_\eta/p$$ is injective. On the other hand, thanks to the algebraic Hartogs' lemma, we have that $\tilde A/p= \bigcap_{\mathfrak{p}\in S} \tilde{A}_\mathfrak{p}/p.$ This implies that $\tilde A/p= \tilde B/p$ and in turn this shows that $\tilde A=\tilde B$.

Now, write $M$ for the module over $\tilde A [\tfrac{1}{p}]$ associated to $\calM$ and for every $\mathfrak p \in S$ write $M_{\mathfrak p}$ for the extension of scalars to $\tilde{A}_\mathfrak{p}[\tfrac{1}{p}]$. By \cite[Proposition 3.5.2]{DrinfeldCrystals}, we have that $M$ is a direct summand of $\tilde{A}[\tfrac{1}{p}]^{\oplus n}$ for some $n>0$. Combining this with the fact that $\tilde A [\tfrac{1}{p}]=\bigcap_{\mathfrak{p}\in S} \tilde{A}_\mathfrak{p}[\tfrac{1}{p}]$, we deduce that 
\begin{align}\label{Eq:Hartogs}
   M = \bigcap_{\mathfrak{p}\in S} M_\mathfrak{p}. 
\end{align}

By Lemma \ref{Lem:ACartesianSquare} applied to the inclusion $A\subseteq \mathrm{Frac}(A)$, if we denote by $M_\eta$ the $\tilde{A}_\eta[\tfrac{1}{p}]$-module associated to $\calM_\eta$, we get the cartesian square
    \begin{equation}
    \begin{tikzcd}
    H^0_{\mathrm{cris}}(\spec(A),\calM)\arrow{d} \arrow{r} \drar[phantom, "\square"]& M\arrow{d} \\
   H^0_{\mathrm{cris}}(\eta,\calM_\eta)\arrow{r}& M_\eta.
    \end{tikzcd}
\end{equation}

It remains to prove that every section $v\in  H^0_{\mathrm{cris}}(\eta,\calM_\eta)$ is also contained in $M$. By \eqref{Eq:Hartogs}, this is equivalent to showing that $v$ is in ${M_{\mathfrak p}}$ for every $\mathfrak p\in S$, which follows from Corollary \ref{Cor:deJong}.
\end{proof}

\begin{Rem}
Theorem \ref{Thm:FullFaithfulness} improves \cite[Theorem 2.2.3]{KedlayaDrinfeld}. As far as we can see, even if $X$ is a connected smooth variety over a perfect field, the result in [\textit{ibid.}] is not enough to deduce the full faithfulness of the restriction functor to the generic point.
\end{Rem}
Theorem \ref{Thm:FullFaithfulness} has as a corollary the following result of independent interest.
\begin{Cor}\label{Cor:ExistenceSlopeFiltration}Let $X$ be a Noetherian Frobenius-smooth scheme over $\Fp$. If $(\calM,\Phi_\calM)$ is an $F$-isocrystal over $X$ with locally constant Newton polygon, then it admits the slope filtration.
\end{Cor}

\begin{proof} First note that we may assume that $X$ is irreducible. As in \cite{KatzSlopeFiltrations}, by taking exterior powers, it is enough to prove that if $(\calM,\Phi_\calM)$ has minimal slope of multiplicity $1$, then there exists a rank $1$ sub-$F$-isocrystal of $(\calM,\Phi_\calM)$ of minimal slope. Note also that the result is known on the generic point $\eta$ of $X$ (see  [\textit{ibid.}] and \cite[Claim 2.8]{deJongOortPurity}). If $S_1(\calM_\eta)\subseteq \calM_\eta$ is the subobject of minimal slope, up to taking a power of the Frobenius structure for some $s>0$ and a Tate twist, we may assume that it corresponds to a lisse $\mathbb{Q}_{q^s}$-sheaf $\mathcal{F}_\eta$ over $\eta$. This lisse sheaf admits models over every codimension $1$ point by \cite[Proposition 2.10]{deJongOortPurity}, thus it admits an extension to a lisse $\mathbb{Q}_{q^s}$-sheaf $\mathcal{F}$ over $X$ by Zariski--Nagata purity theorem. The lisse sheaf $\mathcal{F}$ corresponds then to an $F^s$-isocrystal $(\calN,\Phi_\calN)$ over $X$ which, by Theorem \ref{Thm:FullFaithfulness}, embeds in $(\calM,\Phi_\calM^s)$ providing a model of the inclusion $S_1(\calM_\eta)\subseteq \calM_\eta$. This yields the desired result.
\end{proof}

\begin{Rem}
    Note that in \cite{KedlayaIsoshtukas} Kedlaya proves the analogue of Corollary \ref{Cor:ExistenceSlopeFiltration} for perfect schemes using arc-descent.
\end{Rem}

We also prove a stronger form of Theorem \ref{Thm:FullFaithfulness} under the additional assumption that $\calM$ upgrades to an $F$-isocrystal with slope filtration.
\begin{Thm}\label{Thm:MonodromyGenericPoint} Let $X$ be an irreducible Noetherian Frobenius-smooth scheme over $\Fp$.
If $(\calM,\Phi_\calM)$ is an $F$-isocrystal over $X$ with slope filtration, then $$G(\calM,\eta^\perf)=G(\calM_\eta,\eta^\perf).$$
\end{Thm}
\begin{proof} Let $K$ be the fraction field of $W(\kappa)$ with $\kappa$ the field of constants of $X$ and let $K'$ be the fraction field of the ring of Witt vectors of $\eta^\perf$. Thanks to \cite[Proposition 3.1.8]{ExtensionOfScalars} applied with $F=K, F'=K,$ and $F''=K'$, it is enough to prove that $\langle \calM \rangle\to \langle \calM_\eta \rangle$ is fully faithful and sends semi-simple objects to semi-simple objects. The first part is proven in Theorem \ref{Thm:FullFaithfulness} and does not need the assumption on the slope filtration. For the second part, for an irreducible object $\calN\in\langle \calM \rangle$, we want to prove that $\calN_\eta$ is semi-simple. 

Since $\calN$ is irreducible, it is a subquotient of $\calM^{\otimes m}\otimes (\calM^\vee)^{\otimes n}$ for some $m,n\geq 0$ and by the assumption $\calM^{\otimes m}\otimes (\calM^\vee)^{\otimes n}$ can be endowed with a Frobenius structure with slope filtration. After taking the $s$th-power of the Frobenius structure for some $s>0$ and making a Tate twist, we may further assume that $\calN$ appears in the unit-root part of an $F^s$-isocrystal. Therefore, taking a Jordan--H\"older filtration, we may assume that $\calN$ is a subquotient of an isocrystal $\calM'$ which admits a unit-root $F^s$-structure $\Phi_{\calM'}$ such that $(\calM',\Phi_{\calM'})$ is semi-simple. By \cite[Theorem 2.4.1]{BerthelotMessing}, the $F^s$-isocrystal $(\calM',\Phi_{\calM'})$ corresponds to a semi-simple lisse $\qqq_{p^s}$-sheaf over $X$. By the regularity of $X$, the lisse sheaf remains semi-simple when restricted to the generic point. This implies that $(\calM'_\eta,\Phi_{\calM'_\eta})$ is semi-simple. To conclude we have to prove that $\calM'_\eta$ is semi-simple as well. Let $\calN'_\eta\subseteq \calM'_\eta$ be the \textit{socle} of $\calM'_\eta$, namely the sum of all the irreducible subobjects of $\calM'_\eta$. By maximality, $\calN'_\eta$ is stabilised by the $F^s$-structure, thus it upgrades to a subobject $(\calN'_\eta,\Phi_{\calN'_\eta})\subseteq(\calM'_\eta,\Phi_{\calM'_\eta})$. By semi-simplicity, the inclusion admits a retraction, which induces in particular a retraction of $\calN'_\eta\subseteq \calM'_\eta$. This implies that $\calN'_\eta=\calM'_\eta$, as we wanted.
\end{proof}

\subsection{Descent for isocrystals}\label{Sec:Descent}

We prove now various descent results that we will need in the next section for ($F$-)isocrystals. Let $f:Y\to X$ be a pro-étale $\Pi$-cover of Noetherian Frobenius-smooth schemes over $\Fp$ where $\Pi$ is a profinite group and let $y\in Y(\Omega)$ be an $\Omega$-point of $Y$ with $\Omega$ a perfect field.
\begin{Lem}\label{Lem:desc-glob-sect}For every isocrystal $\calM$ over $X$, the maximal trivial subobject of $f^*\calM$ descends to a subobject $\calN\subseteq \calM$. Moreover, if $\calM$ is endowed with a Frobenius structure $\Phi_\calM$, the inclusion $\calN\subseteq \calM$ upgrades to an inclusion $(\calN,\Phi_\calN)\subseteq (\calM,\Phi_\calM)$ of $F$-isocrystals and $(\calN,\Phi_\calN)$ is a direct sum of isoclinic $F$-isocrystals.
\end{Lem}
\begin{proof}
Since the cover $Y\to X$ is a quasi-syntomic cover, it satisfies descent for isocrystals thanks to \cite[Proposition 3.5.4]{DrinfeldCrystals} (see also \cite{MathewDescent} or \cite[Section 2]{BhattScholzePrismaticCrystals}). By the assumption, $$Y\times_X Y\simeq \varprojlim_{U\subseteq \Pi} (Y\times_X Y)^U$$ where the limit runs over all the open normal subgroups of $\Pi$ and $(Y\times_X Y)^U\coloneqq \coprod_{[\gamma]\in \Pi/U} Y_{[\gamma]}$ is a disjoint union of copies of $Y$. The group $\Pi$ acts on $Y\times_X Y$ in the obvious way. Since $f^*\calM$ comes from $X$, it is endowed with a descent datum with respect to the cover $Y\to X$. This datum consists of isomorphisms $\gamma^*\calM_{(Y\times_X Y)^U}\simeq \calM_{(Y\times_X Y)^U}$ for each $U\subseteq \Pi$ and $\gamma\in \Pi$. The functor $\gamma^*$ sends trivial objects to trivial objects, which implies that the descent datum restricts to a descent datum of $\calT$, the maximal trivial subobject of $f^*\calM$. Therefore, $\calT$ descends to a subobject $\calN\subseteq \calM$, as we wanted. If $\calM$ is endowed with a Frobenius structure, then it induces a Frobenius structure on each isocrystal $\calM_{(Y\times_X Y)^U}$ and this structure preserves each maximal trivial subobject of given slope. This implies that the descended object $\calN\subseteq \calM$ is stabilised as well by the Frobenius and the induced Frobenius structure satisfies the desired property.
\end{proof}

\begin{Prop}\label{Prop:monodromy-after-pro-étale-cover}Let $(\calM,\Phi_\calM)$ be an $F$-isocrystal with the slope filtration and write $\nu$ for the associated Newton cocharacter. If $R_u(G(\calM,f(y)))\subseteq U_\nu$ and $\mathrm{Gr}_{S_\bullet}(f^*\calM)$ is trivial, then $G(f^*\calM,y)=R_u(G(\calM,f(y)))$.
\end{Prop}
\begin{proof}
Since $\mathrm{Gr}_{S_\bullet}(f^*\calM)$ is trivial, the group $G(f^*\calM,y)$ is a unipotent subgroup of $G(\calM,f(y))$ sitting inside $U_\nu$. Therefore, we are in the situation of Proposition \ref{Prop:Tann-crit} and we have to prove that (ii) is satisfied. This amounts to show that for every $m,n\geq 0$, the maximal trivial subobject $\calT\subseteq f^*(\calM^{\otimes m} \otimes (\calM^\vee)^{\otimes n})$ descends to a semi-simple isocrystal over $X$. By Lemma \ref{Lem:desc-glob-sect}, we know that $\calT$ descends to an isocrystal $\calN$ which is the direct sum of isocrystals which can be endowed with an isoclinic Frobenius structure. Since $R_u(G(\calM,f(y)))$ is contained in $U_\nu$, we deduce that $\calN$ is semi-simple, as we wanted.
\end{proof}

\begin{Lem}\label{Lem:tensor-finite-p-basis}
If $k'/k$ is a separable field extension and $k'$ admits a finite $p$-basis, then $k'\otimes_k k'$ admits a finite $p$-basis as well.
\end{Lem}
\begin{proof}

Thanks to \cite[Theorem 26.6]{Matsumura}, the field $k$ admits a finite $p$-basis $t_1,\cdots,t_d$ which extends to a finite $p$-basis $t_1,\cdots,t_d,u_1,\cdots u_e$ of $k'$. We claim that $\Gamma\coloneqq \{t_i\otimes 1\}_i\cup \{u_i\otimes 1\}_i\cup \{1\otimes u_i\}_i$ is a finite $p$-basis of $k'\otimes_k k'$. It is clear from the construction that the elements of $\Gamma$ generate $k'\otimes_k k'$ over $(k'\otimes_k k')^p$. On the other hand, the exact sequence

$$0\to \Omega^1_{k/\Fp}\otimes_k(k'\otimes_k k')\to \Omega^1_{k'\otimes_kk'/\Fp}\to (\Omega^1_{k'/k} \otimes_k k') \oplus (k'\otimes_k\Omega^1_{k'/k}) \to 0$$ 

shows that the elements $d\gamma$ with $\gamma\in \Gamma$ form a basis of the free module $\Omega^1_{k'\otimes_kk'/\Fp}$. We deduce the $p$-independence of the elements of $\Gamma$ by arguing as in \cite[Lemma 07P2]{stacks-project}.
\end{proof}
\begin{Lem}\label{Lem:no-global-sections}
Let $X$ be a Frobenius-smooth scheme over $\Fp$ and  $(\calM,\Phi_\calM)$ an $F$-isocrystal over $X$ with a locally free lattice and constant Newton polygon. If all the slopes of $(\calM,\Phi_\calM)$ are non-zero, the vector space $ H^0_{\mathrm{cris}}(X,\calM)^{F=1}$ vanishes.
\end{Lem}
\begin{proof}
Since $X$ is Frobenius-smooth, by \cite[§1.3.5.ii]{BerthelotMessing} the global sections of any isocrystal over $X$ embed into the global sections of the base change to $X^{\mathrm{perf}}$. Over $X^{\mathrm{perf}}$ we argue as in the proof of \cite[Theorem 2.4.2]{KatzSlopeFiltrations}, namely we assume that $X^{\mathrm{perf}}=\spec(A)$ is affine and we embed $A$ into a product of perfect fields. This reduces the problem to the case of perfect fields, where the result is well-known.
\end{proof}

\begin{Prop}\label{Prop:non-splitting}
Let $k\subseteq k'$ be a separable extension of characteristic $p$ fields with finite $p$-basis and let $(\calM,\Phi_{\calM})$ be a free $F$-isocrystal over $k$ with slope filtration $S_\bullet$ of length $n$. If $\calM_{k'}$ admits a Frobenius-stable splitting of the form $\calN_{k'}\oplus S_{n-1}(\calM_{k'})$ with $\calN_{k'}$ some subobject of $\calM_{k'}$, the same is true for $\calM$.
\end{Prop}
\begin{proof}
Since $\Spec k'\to \Spec k$ is a quasi-syntomic cover, it satisfies descent for isocrystals thanks to the descent results of Drinfeld and Mathew in \cites{DrinfeldCrystals,MathewDescent} (see \cite[Theorem 2.2]{BhattScholzePrismaticCrystals}). Therefore, in order to descend $\calN_{k'}$ to $k$ it is enough to show that the splitting $\calN_{k'\otimes_k k'}\oplus S_{n-1}(\calM_{k'\otimes_k k'})$ is unique. Suppose that $\calN'_{k'\otimes_k k'}\oplus S_{n-1}(\calM_{k'\otimes_k k'})$ was a different splitting. Then there would exist a non-trivial Frobenius-equivariant morphism $\calN'_{k'\otimes_k k'}\to S_{n-1}(\calM_{k'\otimes_k k'})$. In other words, the $F$-isocrystal $\underline{\hom}(\calN'_{k'\otimes_k k'},S_{n-1}(\calM_{k'\otimes_k k'}))$ would have a non-trivial Frobenius-invariant global section. Since the slopes of $\underline{\hom}(\calN'_{k'\otimes_k k'},S_{n-1}(\calM_{k'\otimes_k k'}))$ are all negative by definition and $k'\otimes_k k'$ admits a finite $p$-basis by Lemma \ref{Lem:tensor-finite-p-basis}, this would contradict Lemma \ref{Lem:no-global-sections}. \end{proof}

\begin{Lem}\label{Lem:TraceFiniteFlat}
    If $f\colon Y\to X$ is a finite flat surjective morphism of Noetherian Frobenius-smooth schemes and $\calM$ is an isocrystal over $X$, then the natural morphism $H^1_\mathrm{cris}(X,\calM)\to H^1_\mathrm{cris}(Y,f^*\calM)$
    is injective for every $i$.
    \end{Lem}
\begin{proof}
    Thanks to the Leray spectral sequence, we have a natural injective morphism $$H^1_{\mathrm{cris}}(X,f_*f^*\calM)\hookrightarrow H^1_{\mathrm{cris}}(Y,f^*\calM).$$ Since a composition of injective morphisms is injective, it remains to prove that the natural morphism $H^1_{\mathrm{cris}}(X,\calM)\to H^1_{\mathrm{cris}}(X,f_*f^*\calM)$ given by the adjunction is injective. For this we note that thanks to \cite[Proposition 3.5.2]{DrinfeldCrystals} locally on $X$ the isocrystal $\calM$ is a vector bundle over a $p$-adic lift of $X$ endowed with a $p$-adic connection. Thus by the projection formula, we have that $f_*f^*\calM=\calM\otimes_\calO f_*{\calO}.$ We deduce that there is a canonical trace morphism $f_*f^*\calM\to \calM$ such that the composition $$\alpha\colon \calM\to f_*f^*\calM\to \calM$$ is the endomorphism of $\calM$ that, locally on $X$, is given by the multiplication with respect to the degree of $f$. In particular, $\alpha$ is an automorphism of $\calM$, so that $H^1_{\mathrm{cris}}(X,\calM)\to H^1_{\mathrm{cris}}(X,f_*f^*\calM)$ admits a retraction.
\end{proof}

\begin{Prop} \label{Prop:FiniteFlatMonodromy}
    If $f\colon Y\to X$ is a finite flat surjective morphism of irreducible Noetherian Frobenius-smooth schemes and $\calM$ is an isocrystal over $X$ with unipotent monodromy group, then $$G(f^*\calM,y)=G(\calM,f(y)).$$ 
\end{Prop}
\begin{proof}
   We want to use the criterion of Proposition \ref{Prop:SecondTannakianCriterion} applied to the inclusion $G(f^*\calM,y)\subseteq G(\calM,f(y))$. Thanks to Lemma \ref{Lem:TraceFiniteFlat}, we have that $\mathrm{Ext}^1_{\mathrm{Isoc}(X)}(\calO_X,\calN)\to \mathrm{Ext}^1_{\mathrm{Isoc}(Y)}(\calO_Y,f^*\calN) $ is injective for every isocrystal $\calN$ over $X$. It remains to note that if $\kappa$ is the field of constants of $Y$, then the extension of scalars from $W(\kappa)[\tfrac{1}{p}]\subseteq W(\Omega)[\tfrac{1}{p}]$ commutes with the operation of taking ext-groups in the category of isocrystals. This yields the desired result.
\end{proof}

\subsection{The local monodromy theorem}\label{Sec:LocalMonodromyTheorem}

We are ready to put all the previous results together and prove Theorem \ref{Thm:MainThmMonodromy}. Let $X$ be a smooth irreducible variety over a perfect field and let $x$ be a closed point of $X$. In this section we view $X^{/x}$ as the scheme $\spec \widehat{\calO}_{X,x}$ (conventions of Notation \ref{Not:WarningFormalMonodromy} are in force). We denote by $k$ the function field of $X$ and by $k_x$ the function field of $X^{/x}$. We also write $\eta^{\mathrm{sep}}$ (resp. $\bar{\eta}$) for the points over the generic point of $X$ associated to a separable (resp. algebraic) closure of $k$. 
\begin{Lem}\label{Lem:comp-is-sepa} The fields $k$ and $k_x$ have a common finite $p$-basis. In particular, $k\subseteq k_x$ is a separable field extension.
\end{Lem}

\begin{proof} By \cite[Theorem 26.7]{Matsumura}, it is enough to show that $\Omega^1_{k/\Fp}\otimes_k k_x=\Omega^1_{k_x/\Fp}$. Write $A$ for the local ring of $X$ at $x$ and $A_x^\wedge$ for the completion with respect to the maximal ideal $\mathfrak{m}_x$. Since $A$ is regular, we can use \cite[Theorem 30.5 and Theorem 30.9]{Matsumura} to deduce that $\Omega^1_{A/\Fp}\otimes_A A_x^\wedge=\Omega^1_{A_x^\wedge/\Fp}$. We get the desired result after inverting $\mathfrak{m}_x-\{0\}$. 
\end{proof}

\begin{Prop}\label{Prop:monodromy-sep-closure}
If $(\calM, \Phi_\calM)$ is an $F$-isocrystal over $X$ such that $R_u(G(\calM,\bar{\eta}))\subseteq U_\nu$, then $G(\calM_{\eta^{\mathrm{sep}}},\bar{\eta})=R_u(G(\calM,\bar{\eta}))$.
\end{Prop}
\begin{proof}
By Theorem \ref{Thm:MonodromyGenericPoint} we have that $G(\calM,\bar{\eta})=G(\calM_\eta,\bar{\eta})$, so that we are reduced to prove the statement for $G(\calM_\eta,\bar{\eta})$. Note that the cover $f:\eta^{\mathrm{sep}}\to \eta$ is a pro-étale $\mathrm{Gal}(k^{\mathrm{sep}}/k)$-cover and $\mathrm{Gr}_{S_\bullet}(f^*\calM_\eta)$ is trivial because $\eta^{\mathrm{sep}}$ is simply connected. This shows that we can apply Proposition \ref{Prop:monodromy-after-pro-étale-cover} and deduce the desired result.
\end{proof}

\begin{Prop}\label{Prop:glob-sect-comp}If $(\calM, \Phi_\calM)$ is an $F$-isocrystal over $X$ coming from an irreducible overconvergent $F$-isocrystal with constant Newton polygon, then $$H_{\mathrm{cris}}^0(X^{/x},(S_1(\calM))^{/x})=H^0_{\mathrm{cris}}(X^{/x},\calM^{/x}).$$
\end{Prop}
\begin{proof}
By Galois descent we may assume that the ring of constants of $X$ is an algebraically closed field. The inclusion $$H^0_{\mathrm{cris}}(X^{/x},(S_1(\calM))^{/x})\subseteq H^0_{\mathrm{cris}}(X^{/x},\calM^{/x})$$ is an inclusion of $F$-isocrystals over $\kappa$. We suppose by contradiction that this is not an equality. Let $s_r>s_1$ be the greatest slope appearing in $H^0_{\mathrm{cris}}(X^{/x},\calM^{/x})$ and let $v$ be a non-zero vector such that $\Phi_{\calM^{/x}}^n(v)=p^{s_rn}v$ for $n\gg0$. Write $(\tilde{\calM},\Phi_{\tilde{\calM}})$ for the base change of $(\calM,\Phi_\calM)$ to $\eta^{\mathrm{sep}}.$

By the parabolicity conjecture, \cite{AddezioII}, $R_u(G(\calM,\bar{\eta}))$ is contained in $U_\nu$ because $(\calM,\Phi_\calM)$ comes from an irreducible overconvergent $F$-isocrystal. Proposition \ref{Prop:monodromy-sep-closure} then implies that the monodromy group $G(\tilde{\calM},\bar{\eta})$ is equal to $G(\calM,\bar{\eta})\cap U_\nu$. Therefore, the line spanned by $v$ determines a rank $1$ subobject $\tilde{\calL}\subseteq S_r(\tilde{\calM})/S_{r-1}(\tilde{\calM})$ stabilised by the Frobenius. The preimage of this isocrystal in $S_r(\tilde{\calM})$, denoted by $\tilde{\calN}$, is kept invariant by the Frobenius and sits in an exact sequence $$0\to S_{r-1}(\tilde{\calM}) \to \tilde{\calN}\to  \tilde{\calL}\to 0.$$ Since $(\calM,\Phi_{\calM})$ comes from an irreducible overconvergent $F$-isocrystal, the sequence does not admit a Frobenius-equivariant splitting by \cite[Theorem 4.1.3]{AddezioII}. By Proposition \ref{Prop:non-splitting}, the base change of this extension to $k_x^{\mathrm{sep}}$ does not split as well. This leads to a contradiction since $v$ is a vector in $H^0(X^{/x},\calM^{/x})$ which produces a non-trivial global section of $\tilde{\calN}\subseteq \tilde{\calM}.$
\end{proof}

We write $\eta_x$ for the generic point of $X^{/x}$ and $G(\calM^{/x},\eta^{\perf}_x)$ for the monodromy group of $\calM^{/x}$ (notation as in §\ref{Not:WarningFormalMonodromy}) with respect to the perfection of $\eta_x$.

\begin{Thm} \label{Thm:LocalGlobalMonodromyTheorem}
If $(\calM,\Phi_\calM)$ comes from a semi-simple overconvergent $F$-isocrystal with constant Newton polygon, then $$G(\calM^{/x},\eta^{\perf}_x)=R_u(G(\calM,\eta^{\perf}_x)).$$
\end{Thm}

\begin{proof}
Write $G$ for the group $G(\calM,\eta)$ and $V$ for the induced $G$-representation. By \cite{AddezioII} we have that $R_u(G)$ is contained in $U_\nu$ where $\nu$ is the Newton cocharacter. Since $X^{/x}$ is geometrically simply connected we deduce that $\mathrm{Gr}_{S_\bullet}(\calM)^{/x}$ is trivial (see \cite[Proposition 3.3.4]{AddezioII}). This implies that $G(\calM^{/x},\eta_x^\perf)\subseteq R_u(G)\subseteq U_\nu.$ Therefore, in order to apply the criterion of Proposition \ref{Prop:Tann-crit} it is enough to show that for every $\calN\in \langle \calM \rangle$, the space of global sections of $\calN^{/x}$ is the same as the fibre at $x$ of some direct sum of isoclinic subobjects of $\calN$. To prove this, we may assume that $\calN$ can be endowed with a Frobenius structure $\Phi_\calN$ and $(\calN,\Phi_\calN)$ is irreducible. Thanks to Proposition \ref{Prop:glob-sect-comp}, we deduce that the fibre of $S_1(\calN)$ at $x$ is the same as $H^0_{\mathrm{cris}}(X^{/x},\calN^{/x})$. This yields the desired result.
\end{proof}

\section{Automorphism groups of \texorpdfstring{$p$}{p}-divisible groups and Dieudonn\'e--Lie algebras} \label{Sec:AutoEndo}

The goal of this section is to define various groups of tensor-preserving automorphisms and endomorphisms of $p$-divisible groups with $G$-structure, correcting two definitions from \cite{KimCentralLeaves}. To do this, we introduce the notion of \textit{\DLz} and we prove various properties using this point of view. We end the section by studying actions of algebraic groups on nilpotent \DLz s and their associated unipotent groups.
\subsection{Hom groups of \texorpdfstring{$p$}{p}-divisible groups}
\label{Sec:EndomorphismsAutomorphisms}
We mainly follow \cite[Section 3]{ChaiFoliation}, \cite[Section 4]{CaraianiScholze}, and \cite{KimCentralLeaves}. For $p$-divisible groups $Y$ and $Z$ over a perfect field $\kappa$, Chai and Oort construct finite group schemes
\begin{align}
   \mathbf{Hom}^{\mathrm{st}}(Y[p^n], Z[p^n]) \subseteq\mathbf{Hom}(Y[p^n],Z[p^n]),
\end{align}
where $\mathbf{Hom}(Y[p^n],Z[p^n])$ is the sheaf of homomorphisms from $Y[p^n]$ to $Z[p^n]$. They also construct maps
\begin{align}
    \pi_{n}:\mathbf{Hom}^{\mathrm{st}}(Y[p^n], Z[p^n]) \to\mathbf{Hom}^{\mathrm{st}}(Y[p^{n+1}], Z[p^{n+1}])
\end{align}
such that (the additive group underlying)
\begin{align}
    \varinjlim_{n}\mathbf{Hom}^{\mathrm{st}}(Y[p^n], Z[p^n])
\end{align}
is a $p$-divisible group $\mathcal{H}_{Y,Z}$, called the \emph{internal-Hom $p$-divisible group}. Its scheme-theoretic $p$-adic Tate module $\varprojlim_n  \mathbf{Hom}^{\mathrm{st}}(Y[p^n], Z[p^n])$ is isomorphic to the sheaf $\mathbf{Hom}(Y,Z)$ of homomorphisms from $Y$ to $Z$, see \cite[Lemma 4.1.7]{CaraianiScholze}. In particular, this means that the subscheme $$\mathbf{Hom}^{\mathrm{st}}(Y[p^n], Z[p^n]) \subseteq\mathbf{Hom}(Y[p^n],Z[p^n])$$ can be characterised on $R$-points as the subgroup of those homomorphisms $Y[p^n] \to Z[p^n]$ that lift, fppf locally, to homomorphisms $Y[p^N] \to Z[p^N]$ for all $N \ge n$. 

\subsubsection{} It follows from [\textit{ibid.}, Lemma 4.1.8] that there is a canonical isomorphism
\begin{align}
    \mathbb{D}(\mathcal{H}_{Y,Z})[\tfrac 1p] = \underline{\hom}(\mathbb{D}(Y)[\tfrac 1p], \mathbb{D}(Z)[\tfrac 1p])^{\le 0},
\end{align}
where $(-)^{\le 0}$ denotes the operation of taking the subspace of slope $\le 0$ of an $F$-isocrystal, and  
\begin{align} \label{eq:Isocrystal}
    \underline{\hom}(\mathbb{D}(Y)[\tfrac 1p], \mathbb{D}(Z)[\tfrac 1p])
\end{align} 
denotes the internal-Hom in the category of $F$-isocrystals. By the proof of Lemma 4.1.8 of [\textit{ibid.}], there is a canonical isomorphism of group-valued functors
\begin{align} \label{eq:UniversalCoverHomPdiv}
    \widetilde{\mathcal{H}}_{Y,Z} = \mathbf{Hom}(Y,Z)[\tfrac 1p],
\end{align}
where $\widetilde{\mathcal{H}}_{Y,Z}$ is the universal cover of $\mathcal{H}_{Y,Z}$ in the sense of Scholze--Weinstein (see Section \ref{Sec:PrelimBT}). 

\subsubsection{} \label{Sec:BlockMatrixForm} Now let $\kappa=\ovfp$. We will mostly be interested in $\mathcal{H}_Y\coloneqq \mathcal{H}_{Y,Y}$ for a $p$-divisible group $Y$, in which case $T_p\mathcal{H}_Y=\mathbf{Hom}(Y,Y)$ and $\widetilde{\mathcal{H}}_{Y}=\mathbf{Hom}(\tilde{Y},\tilde{Y})=\mathbf{Hom}(Y,Y)[\tfrac 1p]$ have an algebra structure given by composition. We will assume from now on that $Y=\bigoplus_i Y_i$ is a direct sum of isoclinic $p$-divisible groups with $Y_i$ of slope $s_i$ and $s_1<s_2<\dots<s_n$ (this is always true up to isogeny). Then we can write endomorphisms of $Y$ or $\tilde{Y}$ in “block matrix form” to get decompositions
\begin{align}\label{eq:Decomposition}
    \widetilde{\mathcal{H}}_{Y}&=\bigoplus_{i,j} \widetilde{\mathcal{H}}_{Y_i,Y_j} \\
    T_p \mathcal{H}_Y &= \bigoplus_{i,j} T_p \mathcal{H}_{Y_i,Y_j}.
\end{align}
It follows from \cite[Lemma 4.1.8 and Corollary 4.1.10]{CaraianiScholze} that the $p$-divisible groups $\mathcal{H}_{Y_i,Y_j}$ are isoclinic. Moreover, they are zero when $i>j$, \'etale $p$-divisible groups when $i=j$, and connected $p$-divisible groups\footnote{For $i<j$ the $p$-divisible groups $\mathcal{H}_{Y_i,Y_j}$ are related to spaces of extensions of $Y_i$ by $Y_j$, see Remark \ref{Rem:TwoSlopes}.} when $i<j$. This means that we get a lower triangular block matrix form (see the proof of Proposition 4.2.11 of [\textit{ibid.}]), and that the connected parts
\begin{align}\label{eq:DecompositionConnected}
    \widetilde{\mathcal{H}}^{\circ}_{Y} &= \bigoplus_{i<j} \widetilde{\mathcal{H}}_{Y_i,Y_j} \\
    T_p \mathcal{H}^{\circ}_{Y} &= \bigoplus_{i<j}  T_p \mathcal{H}_{Y_i,Y_j}
\end{align}
map to nilpotent endomorphisms under the natural maps to $\mathbf{Hom}(\tilde{Y}, \tilde{Y}) = \widetilde{\mathcal{H}}_Y$ and $\mathbf{Hom}(Y,Y)=T_p \mathcal{H}_{Y}$ respectively. Sending an endomorphism to the induced endomorphism on the associated graded of the slope filtration defines morphisms of formal groups
\begin{align}
    \mathbf{Hom}(\tilde{Y}, \tilde{Y}) &\to \bigoplus_i \mathbf{Hom}(\tilde{Y}_i, \tilde{Y}_i) \\
    \mathbf{Hom}(Y,Y) &\to \bigoplus_i \mathbf{Hom}(Y_i,Y_i)
\end{align}
with kernels $\widetilde{\mathcal{H}}^{\circ}_{Y}$ and $T_p \mathcal{H}^{\circ}_{Y}$ respectively. The targets of these morphisms are locally profinite group schemes. Therefore, since $\mathcal{H}_{Y_i}$ is \'etale, we may identify them with their locally profinite groups\footnote{If $V$ is a topological space we use the notation $\underline{V}$ for the functor on $\algo{\ovfp}$ sending $R$ to $\operatorname{Hom}_{\mathrm{cont}}(|\spec R|,V)$, where $|\spec R|$ is the topological space underlying the scheme $\spec R$. The functor $\underline{V}$ is representable by a finite scheme if $V$ is finite and discrete, and therefore also representable if $T$ is profinite or locally profinite.} of $\ovfp$-points
\begin{align}
    \bigoplus_i \mathbf{Hom}(\tilde{Y}_i, \tilde{Y}_i) &= \bigoplus_{i} \underline{\hom(\tilde{Y}_i, \tilde{Y}_i)(\ovfp)} \\
    \bigoplus_i \mathbf{Hom}(Y_i,Y_i) &= \bigoplus_{i} \underline{\hom(Y_i,Y_i)(\ovfp)}.
\end{align}

\subsubsection{} On Dieudonné modules, we have the inclusion
\begin{align}\label{Inclusion}
    \mathbb{D}({\mathcal{H}}_{Y})[\tfrac{1}{p}]=\underline{\mathrm{Hom}}(\mathbb{D}(Y)[\tfrac 1p],\mathbb{D}(Y)[\tfrac 1p] )^{\le 0}\subseteq \operatorname{Lie} \mathrm{GL}(\mathbb{D}(Y)[\tfrac 1p]))
\end{align}
as $\qpbreve$-vector spaces.\footnote{We will see later in Example \ref{Example:DieudonneLieInternalHom} that this inclusion upgrades to an inclusion of \textit{Dieudonné--Lie $\qpbreve$-algebras.}} If we write $P_\nu\subseteq \mathrm{GL}(\mathbb{D}(Y)[\tfrac{1}{p}])$ for the parabolic subgroup stabilising the slope filtration of $\mathbb{D}(Y)[\tfrac{1}{p}]$ and $U_\nu\subseteq P_\nu$ for its unipotent radical, we deduce from the previous discussion the following lemma.
\begin{Lem}\label{Lem:LieAlgebrasAndDieudonne}The inclusion \eqref{Inclusion} identifies $\mathbb{D}({\mathcal{H}}_{Y})[\tfrac{1}{p}]$ with $\operatorname{Lie} P_\nu$ and $\mathbb{D}({\mathcal{H}}^\circ_{Y})[\tfrac{1}{p}]$ with $\operatorname{Lie} U_\nu$.
\end{Lem}
\begin{proof}When $i\leq j$ the $F$-isocrystal $\underline{\hom}(\mathbb{D}(Y_i)[\tfrac 1p], \mathbb{D}(Y_j)[\tfrac 1p])$ is isoclinic of slope $s_i-s_j\leq 0$. Therefore, in this case
    \begin{align}
    \mathbb{D}(\mathcal{H}_{Y_i,Y_j})[\tfrac 1p] = \underline{\hom}(\mathbb{D}(Y_i)[\tfrac 1p], \mathbb{D}(Y_j)[\tfrac 1p])^{\leq 0}=\underline{\hom}(\mathbb{D}(Y_i)[\tfrac 1p], \mathbb{D}(Y_j)[\tfrac 1p]). 
\end{align}
On the one hand, \eqref{eq:Decomposition} implies that $\mathbb{D}({\mathcal{H}}_{Y})[\tfrac{1}{p}]=\operatorname{Lie} P_\nu.$ On the other hand, \eqref{eq:DecompositionConnected} gives  $\mathbb{D}({\mathcal{H}}^\circ_{Y})[\tfrac{1}{p}]=\operatorname{Lie} U_\nu$.
\end{proof}

\begin{Lem} \label{Lem:InternalHomCSD}
   If $\kappa=\ovfp$ and $Y=\bigoplus_i Y_i$ is a direct sum of isoclinic $p$-divisible groups, then $\mathcal{H}_{Y}$ is completely slope divisible.
\end{Lem}
\begin{proof} The lemma follows from Lemma \ref{Lem:CSDDieudonneOvfp} since by \cite[Lemma 4.1.8]{CaraianiScholze} each $\mathcal{H}_{Y_i,Y_j}$ is isoclinic. 
\end{proof}

\subsubsection{Automorphisms} \label{subsub:automorphisms} Let $\kappa=\ovfp$ as before. Let $\aut(\tilde{Y})$ be the functor on $\algo{\ovfp}$ of automorphisms of $\tilde{Y}$ and let $\aut(Y)$ be the functor on functor on $\algo{\ovfp}$ of automorphisms of $Y$. Since $Y=\varinjlim_n Y[p^n]$, this means that 
\begin{align}
    \aut(Y) = \varprojlim_n \aut(Y[p^n]).
\end{align}
\begin{Lem} \label{Lem:Representable}
    The functor $\aut(\tilde{Y})$ is representable by a formal scheme, and the functor $\aut(Y)$ is represented by a scheme.
\end{Lem}
\begin{proof}
There is a monomorphism
\begin{align}
    \alpha: \aut(\tilde{Y}) &\to \mathbf{Hom}(\tilde{Y}, \tilde{Y})^{\oplus 2} \\
    \gamma &\mapsto (\gamma, \gamma^{-1}) 
\end{align}
and the image consists of those pairs $(\gamma, \gamma')$ such that $\gamma \circ \gamma'=1=\gamma' \circ \gamma$. It follows from this that $\alpha$ is representable in closed immersions, and since 
\begin{align}
    \mathbf{Hom}(\tilde{Y}, \tilde{Y}) = \mathbf{Hom}(Y,Y)[\tfrac 1p] = \widetilde{\mathcal{H}}_{Y},
\end{align}
we see that $\aut(\tilde{Y})$ is representable by a formal scheme. Similarly, there is a monomorphism
\begin{align}
    \aut(Y) &\to \mathbf{Hom}(Y,Y)^{\oplus 2} \\
    \gamma &\mapsto (\gamma, \gamma^{-1}) 
\end{align}
representable in closed immersions, showing that $\aut(Y)$ is represented by a scheme since $$\mathbf{Hom}(Y,Y)=T_p \mathcal{H}_Y$$ is represented by a scheme by \cite[Proposition 4.1.2.(4)]{CaraianiScholze}.
\end{proof}

\subsubsection{} \label{subsub:SemiDirectProduct}
For each $i$, we can identify $\aut(\tilde{Y}_i)$ with the locally profinite group scheme $\underline{\aut(\tilde{Y}_i)(\ovfp)}$ since it is a closed subscheme of the locally profinite ring scheme $\operatorname{End}(\tilde{Y}_i)$. Similarly, we may identify $\aut(Y_i)$ with the profinite group scheme $\underline{\aut(Y_i)(\ovfp)}$. As above, there are natural maps
\begin{align}
    \aut(\tilde{Y}) &\to \prod_{i} \aut(\tilde{Y}_i) \\
    \aut(Y) &\to \prod_{i} \aut(Y_i)
\end{align}
taking a morphism to its induced morphism on the associated graded of the slope filtration. We write $\aut(\tilde{Y})^{\circ}$ respectively $\aut(Y)^{\circ}$ for its kernel. These morphisms split the natural inclusions of $\prod_{i} \aut(\tilde{Y}_i)$ into $\aut(\tilde{Y})$ and $\prod_i \aut(Y_i)$ into $\aut(Y)$ and thus induce semidirect product decompositions 
\begin{align}
    \aut(\tilde{Y}) &\simeq \aut(\tilde{Y})^{\circ} \rtimes \underline{\aut(\tilde{Y})(\ovfp)} \\
    \aut(Y) &\simeq \aut(Y)^{\circ} \rtimes \underline{\aut(Y)(\ovfp)}
\end{align}
The group scheme $\operatorname{Aut}^{\circ}(Y)$ has underlying topological space equal to a point, because it is a subscheme of $1 + \prod_{i<j} \mathbf{Hom}(\mathcal{H}_i, \mathcal{H}_j) \simeq \prod_{i <j} T_p \mathcal{H}_{Y_i,Y_j}$, and the latter has underlying topological space equal to a point since $\mathcal{H}_{Y_i,Y_j}$ is a connected $p$-divisible group for $i<j$. Thus we can think of $\operatorname{Aut}^{\circ}(Y)$ as the identity component of $\operatorname{Aut}(Y)$, explaining the notation. 

\subsubsection{} Our block matrix description of $\operatorname{End}(\tilde{Y})$ implies that all elements $\gamma$ of $\aut(\tilde{Y})^{\circ}$ are unipotent automorphisms of $\tilde{Y}$. This means that the logarithm map
\begin{align}
    L:\aut(\tilde{Y})^{\circ} & \to \widetilde{\mathcal{H}}^{\circ}_{Y} \\
    X &\mapsto \sum_{i=1}^{\infty} (-1)^{i+1} \tfrac{(X-1)^{i}}{i}
\end{align}
and the exponential map
\begin{align}
    E:\widetilde{\mathcal{H}}^{\circ}_{Y} &\to \aut(\tilde{Y})^{\circ} \\
    X &\mapsto \sum_{i=0}^{\infty} \tfrac{X^i}{i!}
\end{align}
are well-defined and we have $E \circ L=1$ and $L \circ E=1$. 
\begin{Rem} \label{Rem:ColimitQRSP}
In particular, $\aut(\tilde{Y})^{\circ}$ is isomorphic as a functor to $\widetilde{\mathcal{H}}^{\circ}_{Y}$. Thus it is represented by a formal scheme that is a filtered colimit of spectra of qrsp rings. It moreover implies that $\aut(\tilde{Y})^{\circ}$ is connected, as the notation suggests. 
\end{Rem}
\begin{Rem} \label{Rem:IntegralBCH}
It is important to note that the exponential map $E$ does not send $T_p \mathcal{H}_Y^{\circ} \subset \widetilde{\mathcal{H}}^{\circ}_{Y}$ to $\aut(\tilde{Y})^{\circ} \subset \aut(\tilde{Y})^{\circ}$ unless $p \gg 0$ because of the denominators in the exponential map. More precisely, if there are elements of $T_p \mathcal{H}_Y^{\circ} \subseteq \mathbf{\hom}(Y,Y)$ that are nilpotent of order $\ge p$, then the denominators of the exponential map will be divisible by $p$. In this case, $E$ does not send $T_p \mathcal{H}_Y^{\circ}$ to $\aut(\tilde{Y})^{\circ}$. Note that this can only happen if $Y$ has at least $p+1$ slopes. 
\end{Rem}

\subsubsection{} In order to work with Shimura varieties, it will be more fruitful to use the commutator bracket on $\mathbf{Hom}(Y,Y)[\tfrac 1p]$ to equip $\widetilde{\mathcal{H}}_{Y}$ with the structure of a Lie algebra. More precisely, this is a Lie algebra over the locally profinite ring-scheme $\underline{\qp}$ associated to the topological ring $\qp$.

The Baker--Campbell--Hausdorff (BCH) formula gives an expression for the group structure on $\aut(\tilde{Y})^{\circ}$ in terms of the Lie bracket on $\widetilde{\mathcal{H}}^{\circ}_{Y}$ (see \cite[Part I, Chapter IV, §7-8]{SerreLie}). In fact if $V$ is a (possibly infinite dimensional) vector space over a field of characteristic zero and $X,Y$ are two nilpotent endomorphisms of $V$, then the BCH formula expresses $\operatorname{exp}(X) \operatorname{exp}(Y)$ in terms of $X$ and $Y$ and their iterated Lie brackets.
\subsection{Dieudonn\'e--Lie algebras} This section is an interlude on Dieudonn\'e modules endowed with a Lie bracket. Write $\varphi$ for the Frobenius on both $\zpbreve$ and $\qpbreve$.

\begin{Def} \label{Def:DieudonneLie}
A \textit{\DLz} is a triple $(\mfap, \varphi_{\mfap}, [-,-])$ where $(\mfap, \varphi_{\mfap})$ is a Dieudonn\'e module over $\ovfp$ (see Definition \ref{Def:DieudonneI}) and
\begin{align}
    [-,-]:\mfap \times \mfap \to \mfap
\end{align}
is a Lie bracket such that the following diagram commutes
\begin{equation}
    \begin{tikzcd}
     \mfap[\tfrac 1p] \times  \mfap[\tfrac 1p] \arrow{d}{[-,-]} \arrow{r}{\varphi_{\mfap} \times \varphi_{ \mfap}}& \mfap[\tfrac 1p] \times \mfap[\tfrac 1p] \arrow{d}{[-,-]} \\
      \mfap[\tfrac 1p] \arrow{r}{\varphi_{ \mfap}}& \mfap[\tfrac 1p].
    \end{tikzcd}
\end{equation}
A morphism of \DLz s is a $\zpbreve$-linear map $f:\mfap \to \mfbp$ that respects the Lie brackets and induces a homomorphism of Dieudonn\'e modules. If $f$ is injective with finite cokernel we say that $f$ is an \textit{isogeny}. We write $\mbx(\mfap)$ for the $p$-divisible group attached to the Dieudonné module $(\mfap, \varphi_{\mfap})$. We also say that a \DLz\ is \textit{completely slope divisible} if the underlying Dieudonné module is so (see Definition \ref{Def:CompletelySlopeDivisible}).

Similarly, a \textit{Dieudonn\'e--Lie $\qpbreve$-algebra} is a triple $(\mfa, \varphi_{\mfa}, [-,-])$ where $(\mfa, \varphi_{\mfa})$ is a rational Dieudonn\'e module over $\ovfp$ and $[-,-]$ is a Lie bracket on $\mfa$ satisfying the same compatibility. We write $$(\mfa, \varphi_{\mfa}, [-,-])$$ for the Dieudonn\'e--Lie $\qpbreve$-algebra obtained from a  \DLz\ $$(\mfap, \varphi_{\mfap}, [-,-])$$ by inverting $p$.  We also denote by $\tildex({\mfa})$ the universal cover of the $p$-divisible group associated to an integral lattice of $(\mfa, \varphi_{\mfa})$. This assignment does not depend on the choice of the lattice. We will very often omit the Frobenius structure and the Lie bracket in the notation of Dieudonné--Lie algebras.
\end{Def}
\begin{Rem}
    Alternatively, Dieudonné--Lie $\qpbreve$-algebras can be defined as the Lie algebra objects in the symmetric monoidal category of $F$-isocrystals over $\ovfp$ such that the underlying $F$-isocrystal is a rational Dieudonné module.
\end{Rem}

\begin{Vb} \label{Example:DieudonneLieInternalHom}
The first example of a \DLz \ is the Dieudonn\'e module of the internal-Hom $p$-divisible group $\mathcal{H}_{Y}$ attached to a $p$-divisible group $Y$ over $\ovfp$, denoted by $\mathbb{D}(\mathcal{H}_{Y}).$
Indeed, the Lie bracket coming from the commutator bracket on the associative algebra $T_p \mathcal{H}_Y=\mathbf{Hom}(Y,Y)$,
induces an $\varphi$-equivariant Lie bracket on $\mfa$ by \cite[Corollary 1.2.5]{ExteriorPowers}. The Lie bracket on $\mathbf{Hom}(Y,Y)$ clearly sends the identity component $\mathbf{Hom}(Y,Y)^{\circ}$ to itself. This leads to our second example of a \DLz, the one induced on $\mathbb{D}(\mathcal{H}_{Y}^{\circ}).$ Arguing as in \cite[Lemma 3.2.23]{ChaiOortRigidity}, one can show that any Dieudonné--Lie $\qpbreve$-algebra embeds in $\mathbb{D}(\mathcal{H}_{Y})$ for some big $Y$. 
\end{Vb}

Given a \DLz, there is also a natural procedure to get many other isogenous \DLz s by using the Frobenius structure. Let us see this more in detail, as it will play an important role.

\begin{Cons}\label{Cons:PhiN}
    For a  \DLz\ $(\mfap, \varphi_{\mfap}, [-,-])$ and $n\in \mathbb{Z}$, we define $\Phi^n (\mfap)$ to be the $\zpbreve$-submodule $\varphi^n_{\mfap}(\mfap)\subseteq \mfa$. This $\zpbreve$-submodule is preserved by the Frobenius and the Lie bracket of $\mfa$. We then get on $\Phi^n (\mfap)$ a \DLz\ structure. Note that $\varphi_{\mfa}^{-n}$ induces an isomorphism $\Phi^n (\mfap)\xrightarrow{\sim} \mfa^{+,(n)}$ of \DLz s, where $\mfa^{+,(n)}$ denotes the $n$th Frobenius twist of $\mfa^{+}.$
\end{Cons}

\begin{Lem}\label{Lem:TateModulePhiN}
For $n\geq 0$, there is a natural exact sequence of fpqc sheaves $$ 0\to T_p\mathbb{X}(\mfap)\to T_p\mathbb{X}(\Phi^n(\mfap))\to \mathbb{X}(\mfap)[F^n]\to 0.$$
\end{Lem}

\begin{proof}
    Thanks to the isomorphism $\varphi_{\mfa}^{-n}\colon\Phi^n (\mfap)\xrightarrow{\sim} \mfa^{+,(n)}$, this is equivalent to proving that we have an exact sequence $$ 0\to T_p\mathbb{X}(\mfap)\to T_p\mathbb{X}(\mfap)^{(n)}\to \mathbb{X}(\mfap)[F^n]\to 0,$$ where $T_p\mathbb{X}(\mfap)\to T_p\mathbb{X}(\mfap)^{(n)}$ is induced by the $n$th-power of the (relative) Frobenius of $\mathbb{X}(\mfap)$. The result then follows by a classical diagram chase, using the fact that the Frobenius of $\mathbb{X}(\mfap)$ is faithfully flat.
\end{proof}
We will often use the following lemma to study \DL s and reduce ourself to the abelian case.
\begin{Lem} \label{Lem:MaximalSlopeCentral}
Let $\mfa$ be a Dieudonn\'e--Lie $\qpbreve$-algebra where all the slopes are negative. Let $\mu_1$ be the smallest slope of $\mfa$ and let $\mfb \subseteq \mfa$ be an $F$-stable $\qpbreve$-subspace that is isoclinic of that slope. Then $\mfb$ is contained in the centre of $\mfa$.
\end{Lem}
\begin{proof}
There are no non-zero $F$-equivariant maps $\mfb \otimes \mfa \to \mfa$
because, by assumption, all the slopes of $\mfb \otimes \mfa$ are strictly smaller than the slopes of $\mfa$. Hence the restriction of the Lie bracket to $\mfb \times \mfa$ is trivial.
\end{proof}

\subsubsection{Integrability} In general, for Dieudonné--Lie $\zpbreve$-algebras, the BCH formula is not well defined. We want to clarify here how to bypass this issue in our context. Let us first recall the formula presented in \cite[Part I, Chapter IV, §7-8]{SerreLie}. We focus on the \textit{nilpotent} setting, since it is the only one that we will need.

\begin{Def}
We say that a Dieudonn\'e--Lie $\qpbreve$-algebra $(\mfa, \varphi_{\mfa}, [-,-])$ is \textit{nilpotent} if the underlying Lie algebra $(\mfa, [-,-])$ is nilpotent. We also say that a \DLz \ is \textit{nilpotent} if the associated Dieudonn\'e--Lie $\qpbreve$-algebra is so.
\end{Def}

For positive integers $d,n$, we denote by $\Delta_n(d)\subseteq \mathbb{N}^n\times \mathbb{N}^n$ the subset of elements $(\underline{r},\underline{s})\in \mathbb{N}^n\times \mathbb{N}^n$ with $\underline{r}=(r_1,\dots,r_n)$ and $\underline{s}=(s_1,\dots,s_n)$ such that $\sum_{i=1}^n (r_i+s_i)=d$ and $r_i+s_i\neq 0$ for every $i$. The BCH series can be written as $$\mathrm{BCH}(X,Y)\coloneqq \sum_{d=1}^\infty \sum_{n=1}^\infty\sum_{(\underline{r},\underline{s})\in\Delta_n(d)}(-1)^{n-1}\frac{[X^{r_1}Y^{s_1}\cdots X^{r_n}Y^{s_n}]}{dn\prod_{i=1}^nr_i!s_i!}.$$

If $\mfa$ is a nilpotent Dieudonné--Lie $\qpbreve$-algebra, then for every $a,b\in \mfa$, there is a well-defined element $\mathrm{BCH}(a,b)\in \mathfrak a$. Indeed, for $d$ and $n$ large enough, the terms of the series vanish. For Dieudonné--Lie $\zpbreve$-algebras, instead, the series might not be defined when the denominators are too divisible by $p$. It then makes sense to consider the following definition.
\begin{Def}\label{Def:IntegrableDLAlgebras}
We say that a nilpotent Dieudonné--Lie $\zpbreve$-algebra is \textit{integrable} if for every $a,b\in \mfa^+$, each summand $$(-1)^{n-1}\frac{[a^{r_1}b^{s_1}\cdots a^{r_n}b^{s_n}]}{dn\prod_{i=1}^nr_i!s_i!}$$ of $\mathrm{BCH}(a,b)$ lies in $\mfa^+$.
\end{Def}

\begin{Lem} \label{Lem:ExistenceLattice}
    If $\mfa^+$ is a nilpotent \DLz, then $p^2\mfa^+$ is an integrable Dieudonn\'e--Lie $\zpbreve$-subalgebra.
\end{Lem}
\begin{proof}
 Write $\mfbp$ for the Dieudonn\'e--Lie $\zpbreve$-subalgebra $p^2\mfa^+\subseteq \mfap$. We have to prove that for every $v,w\in \mfbp$, each term $$(-1)^{n-1}\frac{[v^{r_1}w^{s_1}\cdots v^{r_n}w^{s_n}]}{dn\prod_{i=1}^nr_i!s_i!}$$ lies in $\mfbp$. By the bilinearity of the Lie bracket, we have that $[\mfbp,\mfbp]\subseteq p^2\mfbp$. Thus by induction, we deduce that for every $e\geq 2$ and every set of elements $\{v_1,\dots,v_{e}\}$ in $\mfbp$, the nested bracket $[v_1\cdots v_{e}]$ lies in $p^{2e-2}\mfbp$. The nested bracket $[v^{r_1}w^{s_1}\cdots v^{r_n}w^{s_n}]$ is then an element of $p^{2d-2}\mfbp$. Since the denominator $dn\prod_{i=1}^nr_i!s_i!$ divides $(d!)^2$, we get the desired result thanks to the fact that $\tfrac{p^{d-1}}{d!}$ is an element of $\zpbreve$.\end{proof}

\subsubsection{Dieudonn\'e--theory and bilinear maps}
Let $(\mfap, \varphi_{\mfap}, [-,-])$ be a Dieudonn\'e--Lie algebra and let $\mathbb{X}(\mfap)$ be the unique $p$-divisible group over $\ovfp$ with (covariant) Dieudonn\'e module $(\mfap, \varphi_{\mfap})$. We want to equip its universal cover and its Tate-module with a bilinear map, coming from the Lie bracket on $\mfap$. For this we record a result on the interaction between Dieudonn\'e theory for universal covers and $\qp$-bilinear maps. The analogous result for bihomomorphisms of finite flat group schemes is \cite[Corollary 1.2.5]{ExteriorPowers}.
\begin{Lem} \label{Lem:LieBracketDiedonneInducesLieBracket}
Let $Y_1,Y_2,Y_3$ be $p$-divisible groups over $\ovfp$, then there is a functorial and $\qp$-linear bijection between the space of $\qpbreve$-bilinear maps
\begin{align}
    g:\mathbb{D}(Y_1)[\tfrac{1}{p}] \times \mathbb{D}(Y_2)[\tfrac{1}{p}] \to \mathbb{D}(Y_3)[\tfrac{1}{p}]
\end{align}
that satisfy $g(\varphi_{Y_1} x, \varphi_{Y_2} y) = \varphi_{Y_3}(g(x,y)$ and the space of bilinear maps
\begin{align}
    f:\tilde{Y}_1 \times \tilde{Y}_2 \to \tilde{Y}_3.
\end{align}
Moreover if $Y_1=Y_2=Y_3$, then $f$ satisfies the Jacobi identity if and only if $g$ does.
\end{Lem}
\begin{proof}
Recall that the internal-Hom $p$-divisible group $\mathcal{H}_{Y_2,Y_3}$ satisfies
\begin{align}
\widetilde{\mathcal{H}}_{Y_2,Y_3}=\mathbf{Hom}(Y_2,Y_3)[\tfrac{1}{p}]=\mathbf{Hom}(\tilde{Y}_2, \tilde{Y_3}). 
\end{align}
It follows from the usual tensor-hom adjunction for $\qp$-vector spaces that bilinear maps
\begin{align}
    \tilde{Y}_1 \times \tilde{Y}_2 \to \tilde{Y}_3
\end{align}
are in bijection with homomorphisms
\begin{align}
    \tilde{Y}_1 \to \widetilde{\mathcal{H}}_{Y_2,Y_3}.
\end{align}
It also follows from the tensor-hom adjunction in the category of $F$-isocrystals that $\qpbreve$-bilinear maps
\begin{align}
    \mathbb{D}(Y_1)[\tfrac{1}{p}] \times \mathbb{D}(Y_2)[\tfrac{1}{p}] \to \mathbb{D}(Y_3)[\tfrac{1}{p}]
\end{align}
that commute with the Frobenius as above are in bijection with morphisms of $F$-isocrystals
\begin{align}
    \mathbb{D}(Y_1)[\tfrac{1}{p}] \to \underline{\hom}(\mathbb{D}(Y_2)[\tfrac{1}{p}], \mathbb{D}(Y_3)[\tfrac{1}{p}]),
\end{align}
where $\underline{\hom}$ denotes the internal-Hom in the category of $F$-isocrystals. Since the slope of $\mathbb{D}(Y_1)[\tfrac{1}{p}]$ is bounded above by $0$, these are also in bijection with morphisms of $F$-isocrystals
\begin{align}
    \mathbb{D}(Y_1)[\tfrac{1}{p}] \to \underline{\hom}(\mathbb{D}(Y_2)[\tfrac{1}{p}], \mathbb{D}(Y_3)[\tfrac{1}{p}])^{\le 0}.
\end{align}
 By \cite[Lemma 4.1.7]{CaraianiScholze} there is an isomorphism
\begin{align}
    \mathbb{D}(\mathcal{H}_{Y_2,Y_3})[\tfrac{1}{p}] = \underline{\hom}(\mathbb{D}(Y_2)[\tfrac{1}{p}], \mathbb{D}(Y_3)[\tfrac{1}{p}])^{\le 0}
\end{align}
and Dieudonn\'e theory over $\ovfp$ tells us that homomorphisms
\begin{align}
    \tilde{Y}_1 \to \widetilde{\mathcal{H}}_{Y_2,Y_3}
\end{align}
are in bijection with morphisms of $F$-isocrystals
\begin{align}
    \mathbb{D}(Y_1)[\tfrac{1}{p}] \to \mathbb{D}(\mathcal{H}_{Y_2,Y_3})[\tfrac{1}{p}],
\end{align}
which is what we wanted to prove. Similarly, we can prove a correspondence for trilinear maps. Since the Jacobi identity can be expressed as the vanishing of a trilinear map, this concludes the proof of the lemma.
\end{proof}
\begin{Rem} \label{Rem:ExplicitBilinear}
The correspondence between bilinear maps of rational Dieudonn\'e modules and bilinear maps of universal covers of $p$-divisible groups can be described explicitly for $R\in \mathrm{QRSP}_{\ovfp}$ using the isomorphism
\begin{align}
    \tilde{Y}_i(R) = \left(B^+_{\mathrm{cris}}(R) \otimes_{\qpbreve} \mathbb{D}(Y_i)[\tfrac{1}{p}]  \right)^{\varphi=1}
\end{align}
of Lemma \ref{Lem:IntegralPHodgeTheoryTateModuleUniversalCover}. Indeed, a $\qpbreve$-bilinear map
\begin{align}
    \mathbb{D}(Y_1)[\tfrac{1}{p}] \times \mathbb{D}(Y_2)[\tfrac{1}{p}] \to \mathbb{D}(Y_3)[\tfrac{1}{p}]
\end{align}
that respects the Frobenius as above induces a $B^+_{\mathrm{cris}}(R)$-bilinear map
\begin{align}
    B^+_{\mathrm{cris}}(R) \otimes_{\qpbreve} \mathbb{D}(Y_1)[\tfrac{1}{p}] \times B^+_{\mathrm{cris}}(R) \otimes_{\qpbreve} \mathbb{D}(Y_2)[\tfrac{1}{p}] \to B^+_{\mathrm{cris}}(R) \otimes_{\qpbreve} \mathbb{D}(Y_3)[\tfrac{1}{p}]
\end{align}
that induces a $\qp$-bilinear map on the $\varphi=1$ subspaces.
\end{Rem}
\begin{Cor} \label{Cor:LieBracketIntegral}
Suppose that we are given a $\zpbreve$-bilinear map
\begin{align}
    g^+:\mathbb{D}(Y_1) \times \mathbb{D}(Y_2) \to \mathbb{D}(Y_3)
\end{align}
satisfying $g^+(\varphi_{Y_1}(x), \varphi_{Y_2} (y)) = \varphi_{Y_3}(g^+(x,y))$. Then the induced map $f:\tilde{Y}_1 \times \tilde{Y}_2 \to \tilde{Y}_3$ restricts to a $\mathbb{Z}_p$-bilinear Lie bracket 
\begin{align}
    f^{+}:T_p Y_1 \times T_p Y_2 \to T_p Y_3.
\end{align}
\end{Cor}
\begin{proof}
It suffices to show this on $R$-valued points for semiperfect $R$ and in fact, by Yoneda, it suffices to check it in the universal case when $R$ is the ring underlying $T_p Y_1 \times T_p Y_2$. But this $R$ is qrsp, so we can use Lemma \ref{Lem:IntegralPHodgeTheoryTateModuleUniversalCover} which tells us that
\begin{align}
    T_p Y_i(R) = \left({A}_{\mathrm{cris}}(R) \otimes_{\zpbreve} \mathbb{D}(Y_i) \right)^{\varphi=1},
\end{align}
and we see that $f^{+}$ sends $T_p Y_1(R) \times T_p Y_2(R)$ to $T_p Y_3(R)$ because $g^{+}$ sends $\mathbb{D}(Y_1) \times \mathbb{D}(Y_2)$ to $\mathbb{D}(Y_3)$.
\end{proof}

\subsection{Formal homogeneous spaces} In this section we focus on the fundamental constructions obtained from the datum of a nilpotent \DLz. They will be widely used in the text to describe the infinitesimal behaviour of central leaves of Shimura varieties.

\begin{Cons}\label{Def:TildePi}
Let $\mfa$ be a nilpotent Dieudonn\'e--Lie $\qpbreve$-algebra (see Definition \ref{Def:DieudonneLie}). We can attach to $\mfa$ an fpqc sheaf of $\qp$-vector spaces $\tildex(\mfa)$ over $\algo{\ovfp}$, defined as the universal cover of any $p$-divisible group corresponding to a Frobenius-stable lattice $\mfap\subseteq\mfa$. By Lemma \ref{Lem:LieBracketDiedonneInducesLieBracket}, the Lie bracket on $\mfa$ induces a nilpotent Lie bracket on $\tildex(\mfa)$. The BCH formula defines then a formal group structure $$m_{\mathrm{Lie}}\colon \tildex(\mfa) {\times} \tildex(\mfa)\to \tildex(\mfa).$$ We write $\tilde{\Pi}(\mfa)$ for the formal group $(\tilde{\mathbb{X}}(\mfa), m_{\mathrm{Lie}})$. The assignment $\mfa\mapsto \tilde{\Pi}(\mfa)$ defines a functor 
    $$\tilde{\Pi}\colon \left\{ \textrm{Nilpotent Dieudonn\'e--Lie } \qpbreve \textrm{-algebras} \right\}\to \left\{ \textrm{Perfect formal groups over }\ovfp \right\}.$$ \end{Cons}

\begin{Cons} \label{Def:PiAplus}    
If $\mfap$ is an integrable nilpotent Dieudonn\'e--Lie $\zpbreve$-algebra, then by Corollary \ref{Cor:LieBracketIntegral}, the Lie bracket on $\mfap$ induces a Lie bracket on $T_p \mbx(\mfap)$, where $\mathbb{X}(\mathfrak{a}^+)$ is the $p$-divisible group attached to $\mfap$. This shows that the formal group structure $m_{\mathrm{Lie}}$ on $\tilde{\mathbb{X}}(\mfa)$ restricts to a group scheme structure on $T_p \mathbb{X}(\mfap)$. We denote by $\Pi(\mfap)$ the group scheme $(T_p \mathbb{X}(\mfap), m_{\mathrm{Lie}})$. Since $p^m\mfap\subseteq\mfap$ is an ideal for every $m\geq 0$, we have a natural exact sequence
$$0\to \Pi(p^m\mfap)\to \Pi (\mfap)\to \Pi_m(\mfap)\to 0,$$
 where $\Pi_m(\mfap)$ is the affine finite scheme $\mbx(\mfap)[p^m]$ endowed with the (non-commutative) group scheme structure induced by $m_\mathrm{Lie}.$ This shows that $$\Pi(\mfap)=\varprojlim_{m\geq 0} \Pi_m(\mfap) $$ is a \textit{profinite group scheme}\footnote{We recall that a \emph{profinite group scheme} over $\fpbar$ is a group scheme over $\fpbar$ that is the inverse limit of finite (not necessarily commutative) group schemes over $\fpbar$.}

In this case, the assignment $\mfap\mapsto {\Pi}(\mfap)$ defines a functor 
    $${\Pi}\colon \left\{ \textrm{Integrable nilpotent Dieudonn\'e--Lie } \zpbreve \textrm{-algebras} \right\}\to \left\{ \textrm{Profinite group schemes over }\ovfp \right\}.$$

\end{Cons}
\begin{Rem} 
    If the slopes of $\mfa$ are negative, then $\tilde{\Pi}(\mfa)$ is a connected affine formal group. By \cite[Lemma 3.2.23]{ChaiOortRigidity}, such a formal group always embeds into $\aut(\tilde{Y})^\circ$ for some big $Y$.  In the integral situation, if the slopes of $\mfap$ are negative, then $\Pi(\mfap)$ is a connected profinite group scheme sitting inside $\tilde{\Pi}(\mfa)$.
\end{Rem}

To continue our analysis, it will be convenient to work under additional assumptions on $\mfap.$

\begin{Def}\label{PlainZpbreveAlgebras} A \textit{plain \DLz} is a completely slope divisible integrable \DLz\ $\mfap$ such that all its slopes are negative. For every slope $\mu$, we write $\mfap_\mu\subseteq \mfap$ for the $\zpbreve$-subspace of slope $\mu$. Note that if $\mu_1$ is the minimal slope of $\mfap$, then $\mfap_{\mu_1}\subseteq \mfap$ is a central plain \DLzs\ by Lemma \ref{Lem:MaximalSlopeCentral} and Lemma \ref{Lem:SaturedSubCSD}.
\end{Def}

As we have seen in Construction \ref{Cons:PhiN}, for every $\mfap$ and $n\in \mathbb{Z}$, there is a \DLz\ $\Phi^n(\mfap)$, constructed using the Frobenius, which is isogenous to $\mfap$. It is easy to check that if $\mfap$ is plain, then $\Phi^n(\mfap)$ is plain for every $n$. 
\begin{Cons}\label{Def:ZA}Let $\mfap$ be a plain Dieudonn\'e--Lie $\zpbreve$-algebra. For every $n\geq 0$, we write $\Pi^n(\mfap)$ for $\Pi(\Phi^n(\mfap))$ and, for $m\geq 0$, we write $\Pi^n_m(\mfap)$ for $\Pi_m(\Phi^n(\mfap))$. There are natural maps $\alpha_n\colon \Pi_n(\mfap)\to \Pi^n_n(\mfap).$ We define $ Z^n(\mfap)$ to be the fppf quotient $$\Pi_{n}^n(\mfap)/\alpha_n(\Pi_n(\mfap))$$ over $\algo{\ovfp}$. We also write $Z(\mfap)$ for the fppf sheaf obtained as the inductive limit $$\varinjlim_n  Z^n(\mfap).$$ There is a natural action of $\tilde{\Pi}(\mfa)$ on the formal scheme $Z(\mfap)$ and an equivariant map $\tildeu(\mfa) \to Z(\mfap)$. We say that $Z(\mfap)$ is the \textit{formal homogeneous space associated to $\mfap$}. 
\end{Cons}

\begin{Rem}By Lemma \ref{Lem:TateModulePhiN}, if $\mfap$ is abelian, then $ Z^n(\mfap)=\mathbb{X}(\mfap)[F^n]$ and $Z(\mfap)=\mathbb{X}(\mfap).$
In general, the formal scheme $Z(\mfap)$ is the fpqc quotient $\tilde{\Pi}(\mfa)/\Pi(\mfap)$ (see Lemma \ref{Lem:FpqcQuotient}). 
\end{Rem}

We want to prove a fundamental representability result for $Z(\mfap).$ For this we use a construction that allows us to reduce many statements to the case when $\mfap$ is abelian.

\begin{Cons}\label{Cons:TorsorForInductionBySlopes}
    Let $\mfap$ be a plain \DLz\ and let $\mu_1$ be the minimal slope of $\mfap$. The formal group structure $m_\mathrm{Lie}$ induces a morphism of affine schemes $\Pi^{n}_n(\mfap_{\mu_1})\times\Pi^{n}_n(\mfap)\to \Pi^{n}_n(\mfap)$. By Lemma \ref{Lem:MaximalSlopeCentral}, the group $\tilde{\Pi}(\mfa_{\mu_1})$ is in the centre of $\tilde{\Pi}(\mfa)$, thus we also get a morphism $$\mathbb{X}(\mfap_{\mu_1})[F^n]\times  Z^n(\mfap)\to  Z^n(\mfap)$$
   which endows $ Z^n(\mfap)$ with a left action of $\mathbb{X}(\mfap_{\mu_1})[F^n]$. This makes $ Z^n(\mfap)$ an fppf torsor over $ Z^n(\mfap/\mfap_{\mu_1})$ under the finite syntomic group scheme $\mathbb{X}(\mfap_{\mu_1})[F^n]$.
\end{Cons}

\begin{Prop}\label{Lem:RepresentabilityZFiniteLevel}
 If $\mfap$ is a plain \DLz, then for every $n\geq 1$ the fppf sheaf $ Z^n(\mfap)$ is represented by $\Spec R_{n,m}$ where $m$ is the dimension of $\mathbb{X}(\mfap)$ and  $$R_{n,m}\coloneqq \ovfp[X_1,\cdots, X_m]/(X_1^{p^n},\cdots,X_m^{p^n}).$$ Moreover, the torsor $ Z^n(\mfap)\to  Z^n(\mfap/\mfap_{\mu_1})$ of Construction \ref{Cons:TorsorForInductionBySlopes} is trivial.
\end{Prop}

\begin{proof}
  We want to prove the result by induction on the number of slopes of $\mfap$. In the isoclinic case the result follows from \cite[Proposition 2.1.2]{Messing}. For the inductive step, we first note that by Construction \ref{Cons:TorsorForInductionBySlopes}, the fppf sheaf $ Z^n(\mfap)$ is represented by a connected scheme which is finite and syntomic over $ Z^n(\mfap/\mfap_{\mu_1})$. The result is then a consequence of \cite[Proposition 3.6.8 of Chapter III]{DemazureGabriel}. 
\end{proof}

\begin{Cor}\label{Cor:RepresentabilityFormalHomogeneousSpace} If $\mfap$ is a plain \DLz, then $Z(\mfap)$ is a formal Lie variety of the same dimension as $\mathbb{X}(\mfap)$.
\end{Cor}

By Corollary \ref{Cor:RepresentabilityFormalHomogeneousSpace} we get a functor
 $$Z\colon \left\{ \textrm{Plain Dieudonn\'e--Lie } \zpbreve \textrm{-algebras} \right\}\to \left\{ \textrm{Formal Lie varieties over }\ovfp \right\}.$$ 

We can summarise the constructions associated to a plain \DLz\ $\mfap$ with the following table.

		\renewcommand{\arraystretch}{1.5}
		\begin{center}
			\vskip1em
			
			\begin{tabular}{|lc||c|c|c|}
				
				\hline
				&&Tate module &Universal cover&Quotient\\
				\hline\hline
				&Commutative&
				$T_p\mathbb{X}(\mfap)$&$\tildex(\mfa)$&$\mathbb{X}(\mfap)$
				
				\\ \hline
				&Non-commutative\ \ \ &$\Pi(\mfap)$&$\tildeu(\mfa)$&$Z(\mfap)$
				\\ \hline
			\end{tabular}
			\vskip1.8em
		\end{center}

 \begin{Lem} \label{Lem:IsTorsor}
    The natural map $\tildeu(\mfa) \to Z(\mfap)$ is an fpqc torsor under the affine group scheme $\Pi(\mfap)$. 
\end{Lem}
\begin{proof}
 Thanks to Lemma \ref{Lem:TateModulePhiN}, for every $n\geq 0$ we have a cartesian diagram
\begin{equation}
    \begin{tikzcd}
        \Pi(\mfap) \ar[dr, phantom, "\square"]\arrow{r} \arrow{d} & \Pi^n(\mfap) \arrow{d} \\
        \alpha_n(\Pi_n(\mfap)) \arrow{r} & \Pi^n_n(\mfap).
    \end{tikzcd}
\end{equation}

By diagram chasing, we deduce that $\Pi^n(\mfap) \to Z^n(\mfap)$ is a torsor for $\Pi(\mfap)$. This yields the desired result.
\end{proof}

 \begin{Lem} \label{Lem:UniversalProperty}
If $X$ is an fpqc sheaf, $P \to X$ is an fpqc torsor under the affine group scheme $\Pi(\mfap)$ and $P \to \tildeu(\mfa)$ is a $\Pi(\mfap)$-equivariant map, then there is a unique induced map $X \to Z(\mfap)$ lying in the commutative diagram
    \begin{equation}
        \begin{tikzcd}
            P \arrow{r} \arrow{d} & \tildeu(\mfa) \arrow{d} \\
            X \arrow[r, densely dotted] & Z(\mfap).
        \end{tikzcd}
    \end{equation}
\end{Lem}
\begin{proof}
If $x:\spec R \to X$ is an $R$-point of $X$, then we need to produce an $R$-point $y:\spec R \to Z(\mfap)$. After passing to the fpqc cover $\spec R \times_{X} P=:\spec R' \to \spec R$, we can lift $x$ to a $\spec R'$-point $\tilde{x}$ of $P$. We then let $y'$ be the image of $\tilde{x}$ in $Z(\mfap)(R')$. We need to show that $y'$ descends to an $R$-point $y$ of $Z(\mfap)$. Since the latter is an fpqc-sheaf, by Corollary \ref{Cor:RepresentabilityFormalHomogeneousSpace} we simply need to show that the two pullbacks of $y'$ to $\spec R'':=\spec R' \times_{\spec R} \spec R'$ agree. But $\spec R''=\Pi(\mfap) \times \spec R$, and the condition on the two pullbacks boils down to asking that $y'$ is fixed by the action of $\Pi(\mfap)$. This follows from Lemma \ref{Lem:IsTorsor}. 
\end{proof}

\begin{Lem}\label{Lem:FpqcQuotient}
Let $\alpha$ be an ordinal as in \cite[Tag 000J]{stacks-project} and let $\mathrm{Alg}^{\mathrm{op}}_{\ovfp,\alpha}\subseteq \mathrm{Alg}^{\mathrm{op}}_{\ovfp}$ be the subcategory of algebras of cardinality at most $\alpha$. If $$\tilde{\Pi}(\mfa)/_\alpha \Pi(\mfap)$$ is the quotient as an fpqc sheaf over $\mathrm{Alg}^{\mathrm{op}}_{\ovfp,\alpha}$, then $\tilde{\Pi}(\mfa)/_\alpha \Pi(\mfap)=Z(\mfap).$ In particular, the quotient is independent of $\alpha.$
\end{Lem}
\begin{proof}
Arguing as in Lemma \ref{Lem:UniversalProperty}, since $\tildeu(\mfa) \to Z(\mfap)$ is a $\Pi(\mfap)$-torsor of fpqc sheaves over $\mathrm{Alg}^{\mathrm{op}}_{\ovfp,\alpha}$, there is a canonical map $\tilde{\Pi}(\mfa)/_\alpha \Pi(\mfap)\to Z(\mfap)$ factorising the morphism $\tilde{\Pi}(\mfa)\to Z(\mfap).$
Both $\tilde{\Pi}(\mfa)\to \tilde{\Pi}(\mfa)/_\alpha \Pi(\mfap)$ and $(\tilde{\Pi}(\mfa)/_\alpha \Pi(\mfap))\times_{Z(\mfap)}\tilde{\Pi}(\mfa)\to \tilde{\Pi}(\mfa)/_\alpha \Pi(\mfap)$ are fpqc torsors under $ \Pi(\mfap)$, thus the diagram
\begin{equation}
        \begin{tikzcd}
            \tilde{\Pi}(\mfa) \arrow[r,"="]\arrow{d} & \tilde{\Pi}(\mfa) \arrow{d} \\
            \tilde{\Pi}(\mfa)/_\alpha \Pi(\mfap) \arrow[r] & Z(\mfap)
        \end{tikzcd}
    \end{equation}
    is cartesian. Thanks to the fact that $\tilde{\Pi}(\mfa)\to Z(\mfap)$ is a surjection of fpqc sheaves, we deduce that $\tilde{\Pi}(\mfa)/_\alpha \Pi(\mfap)\to Z(\mfap)$ is an isomorphism, as we wanted.
\end{proof}

\begin{Lem} \label{Lem:PerfectionZ}
For every plain \DLz \ $\mfap$, the natural map $\tilde{\Pi}(\mfa) \to Z(\mfap)$ induces an isomorphism $$\tilde{\Pi}(\mfa) \simeq Z(\mfap)^\mathrm{perf}.$$
\end{Lem}
\begin{proof}
    For every $n\geq 0$, we have the following factorisation of the $n$th-power of the absolute Frobenius of $Z(\mfap)$ 
\begin{center}
       \begin{tikzcd}
        Z(\mfap)\arrow[rr,"F^n"] \arrow[dr,"\sim"] & & Z(\mfap)\\
        & Z(\Phi^{-n}(\mfap)), \arrow[ur]
    \end{tikzcd} 
\end{center}
where $Z(\Phi^{-n}(\mfap))\to Z(\mfap)$ is induced by the natural inclusion $\Phi^{-n}(\mfap)\subseteq \mfap.$
This implies that $$Z(\mfap)^\mathrm{perf}=\varprojlim_n Z(\Phi^{-n}(\mfap)).$$ If $s$ is the maximal slope of $\mfap$, we have $$p^n\mfap\subseteq \Phi^{-n}(\mfap)\subseteq p^{-ns}\mfap$$ for every $n$ divisible by the denominator of $s$. This shows that $$\varprojlim_n Z(\Phi^{-n}(\mfap))=\varprojlim_n Z(p^n\mfap).$$
By Lemma \ref{Lem:FpqcQuotient}, we have $$Z(p^n\mfap)=\tildeu(\mfa)/\Pi(p^n\mfap)$$ as fpqc sheaves for every $n\geq0$. We deduce that $\varprojlim_n Z(p^n\mfap)\to Z(\mfap)$ is naturally a torsor under the profinite group scheme $\Pi(\mfap)=\varprojlim_m\Pi_m(\mfap)$.

The natural map $$\tilde{\Pi}(\mfa)\to \varprojlim_n Z(p^n\mfap)=Z(\mfap)^\mathrm{perf}$$ is then an equivariant morphism of $\Pi(\mfap)$-torsors over $Z(\mfap)$, thus an isomorphism.
\end{proof}

\subsection{Automorphism groups for Shimura varieties of Hodge type} \label{Sec:AutHodge}
Let $Y$ be a $p$-divisible group over $\ovfp$ and fix an isomorphism $\mathbb{D}(Y)[\tfrac{1}{p}] \simeq V \otimes_{\qp} \qpbreve$, where $V$ is a vector space over $\qp$. In order to generalise Section \ref{Sec:EndomorphismsAutomorphisms} to Shimura varieties of Hodge type, we let $b \in \operatorname{GL}_V(\qpbreve)$ be the Frobenius of the $F$-isocrystal $\mathbb{D}(Y)[\tfrac 1p]$. Then the internal-Hom $F$-isocrystal 
\begin{align}
   \underline{\hom}(\mathbb{D}(Y)[\tfrac 1p], \mathbb{D}(Y)[\tfrac 1p])
\end{align}
is isomorphic to the $F$-isocrystal (where $\operatorname{Ad} \sigma b$ is the Frobenius given by the adjoint action of $b \in \operatorname{GL}_V(\qpbreve)$ on $\mathfrak{gl}_V$)  
\begin{align}
    (\mathfrak{gl}_V \otimes \qpbreve, \operatorname{Ad} \sigma \; b).
\end{align}
If we are given a reductive group $G \subseteq \operatorname{GL}_{V}$ such that $b \in G(\qpbreve)$, then we get an inclusion of $F$-isocrystals
\begin{align}
    (\mathfrak{g} \otimes \qpbreve, \operatorname{Ad} \sigma \; b) \subseteq (\mathfrak{gl}_V \otimes \qpbreve, \operatorname{Ad} \sigma \;b).
\end{align}
The slope filtration of this $F$-isocrystal is described in \cite[Section 3]{KimCentralLeaves}. The slope $t$ part is given by 
\begin{align}
    \bigoplus_{\substack{\alpha_0 \in \Phi_0^{+} \\ \langle \alpha_0, \nu_b \rangle=t}} \mathfrak{u}_{\alpha_0}, 
\end{align}
where $\alpha_0$ runs over the relative roots of $G_{\qpbreve}$ with respect to a maximal $\qpbreve$-split torus $S$ defined over $\qp$ that is contained in a Borel subgroup $B$ with respect to which the Newton cocharacter $\nu_b$ of $b$ is dominant (see \cite[Section 1.1.2]{KMPS} for the definition of the Newton cocharacter). We see that the slope $\le 0$ part corresponds precisely to the Lie algebra of the standard parabolic subgroup $P_b=P_{\nu_{b}}$ associated to $\nu_b$ and that the slope $0$ part corresponds to the Lie algebra of the Levi $M_b$.

Taking non-positive slope parts we get an $F$-stable sub-isocrystal
\begin{align}
     \operatorname{Lie} P_{\nu_b} \subseteq \mathbb{D}(\mathcal{H}_{Y})[\tfrac 1p]
\end{align}
and intersecting with $\mathbb{D}(\mathcal{H}_Y)$ we get a \DLz\ which we write as $\mathbb{D}(\mathcal{H}_{Y}^{G})$ for a $p$-divisible group $\mathcal{H}_{Y}^{G} \subseteq \mathcal{H}_Y$. Note that by construction and Corollary \ref{Cor:LieBracketIntegral},
\begin{align}
    T_p \mathcal{H}_{Y}^{G} \subseteq T_p \mathcal{H}_{Y}=\mathbf{Hom}(Y,Y)
\end{align}
is stable under the commutator bracket.
\begin{Cor} \label{Cor:DimHGY}
The dimension of $\mathcal{H}^G_Y$ is equal to $\langle 2 \rho, \nu_b \rangle$.
\end{Cor}
\begin{proof}
This is explained at the end of the proof of \cite[Proposition 3.1.4]{KimCentralLeaves}.
\end{proof}
Thanks to Lemma \ref{Lem:LieBracketDiedonneInducesLieBracket}, we have that
\begin{align}
    \widetilde{\mathcal{H}}^G_{Y} \subseteq \widetilde{\mathcal{H}}_{Y}=\mathbf{Hom}(Y,Y)[\tfrac 1p]
\end{align}
is closed under the Lie bracket (the commutator) of $\mathbf{Hom}(Y,Y)[\tfrac 1p]$. The Dieudonn\'e module
\begin{align}
    \mathbb{D}(\mathcal{H}_{Y}^{G,\circ})[\tfrac 1p]
\end{align}
is stable under the Lie bracket and consists precisely of the strictly negative slope part of the isocrystal $\operatorname{Lie} P_{\nu_b}$, which can be identified with
\begin{align}
    \operatorname{Lie} U_{\nu_b} \subseteq \operatorname{Lie} P_{\nu_b},
\end{align}
where $U_{\nu_b}$ is the unipotent radical of $P_{\nu_b}$.
\begin{Cor} \label{Cor:CSDHodge}
    If $Y$ is completely slope divisible, then $\mathcal{H}_Y^G$ is completely slope divisible.
\end{Cor}
\begin{proof}
    This follows from Lemma \ref{Lem:SaturedSubCSD} and Lemma \ref{Lem:InternalHomCSD}.
\end{proof}

\subsubsection{} \label{Sec:Tensors}
Let us fix some notation to define the group of automorphisms of $\tilde{Y}$ which preserve the $G$-structure. Here we use the convention that for an object $M$ in a rigid category we write $M^{\otimes}$ for the direct sum of $M^{\otimes n} \otimes (M^{\ast})^{\otimes m}$ for all pairs of integers $m \ge 0, n \ge 0$.

As in \cite[Section 1.3.4]{KMPS}, we can choose a collection of tensors $\{s_{\alpha}\}_{\alpha \in \mathscr{A}} \subseteq V^{\otimes}$ such that $G$ is their pointwise stabiliser in $\operatorname{GL}_V$. If we identify the Lie algebra $\mathfrak{gl}_V$ of $\operatorname{GL}_V$ with the vector space of endomorphisms of $V$, then by \cite[Lemma 5.3.3]{OrdinaryHO} $\mathfrak{g} \subseteq \mathfrak{gl}_V$ consists of those endomorphisms $g$ satisfying $g^{\otimes} s_{\alpha}=0$, let us call such endomorphisms \emph{tensor-annihilating endomorphisms}. It follows that
\begin{align}
     \mathbb{D}(\mathcal{H}_{Y}^{G,\circ})[\tfrac 1p] \subseteq \mathbb{D}(\mathcal{H}_{Y})[\tfrac 1p]
\end{align}
is the subspace of $\mathbb{D}(\mathcal{H}_{Y})[\tfrac 1p]$ of tensor-annihilating endomorphisms. Therefore by Remark \ref{Rem:ExplicitBilinear} it follows that for $R\in \mathrm{QRSP}_{\ovfp}$ we have that
\begin{align}
\widetilde{\mathcal{H}}_Y^G(R) \subseteq \widetilde{\mathcal{H}}_Y(R) = \hom(Y_R, Y_R)[\tfrac{1}{p}]
\end{align}
consists of the endomorphisms $\tilde{Y}_R \to \tilde{Y}_R$ such that the induced endomorphism
\begin{align}
    \mathbb{D}(g):\mathbb{D}(Y)[\tfrac{1}{p}] \otimes_{\qpbreve} B_{\mathrm{cris}}^+(R) \to \mathbb{D}(Y)[\tfrac{1}{p}] \otimes_{\qpbreve} B_{\mathrm{cris}}^+(R)
\end{align}
satisfies $g^{\otimes} (s_{\alpha} \otimes 1)=0$ for all $\alpha$. It follows that
\begin{align}
    T_p \mathcal{H}_Y^G(R) \subseteq T_p \mathcal{H}_Y(R) = \hom(Y_R,Y_R)
\end{align}
consists precisely of the endomorphisms $g:Y_R \to Y_R$ such that the induced endomorphism
\begin{align}
    \mathbb{D}(g):\mathbb{D}(Y)\otimes_{\zpbreve} A_{\mathrm{cris}}(R) \to \mathbb{D}(Y) \otimes_{\zpbreve} A_{\mathrm{cris}}(R)
\end{align}
satisfies $g^{\otimes} (s_{\alpha} \otimes 1)=0$ for all $\alpha$. In both cases we will use the term \emph{tensor-annihilating endomorphisms} to denote such endomorphisms of $\tilde{Y}_R$ or $Y_R$. We use also the notion of \emph{tensor-preserving automorphism} of $\tilde{Y}_R$ which is an automorphism $g$ of $\tilde{Y}_R$ such that the induced automorphism
\begin{align}
    \mathbb{D}(g):\mathbb{D}(Y)[\tfrac{1}{p}] \otimes_{\qpbreve} B_{\mathrm{cris}}^+(R) \to \mathbb{D}(Y)[\tfrac{1}{p}] \otimes_{\qpbreve} B_{\mathrm{cris}}^+(R)
\end{align}
satisfies $g^{\otimes} (s_{\alpha} \otimes 1)=(s_{\alpha} \otimes 1)$ for all $\alpha$.
\begin{Lem} \label{Lem:SubgroupClosed}
There is a (unique) closed subgroup
\begin{align}
   \aut_G(\tilde{Y}) \subseteq \aut(\tilde{Y})
\end{align}
such that on qrsp $\ovfp$-algebras $R$ the subgroup
\begin{align}
    \aut_G(\tilde{Y})(R) \subseteq \aut(\tilde{Y})(R)
\end{align}
consists precisely of the tensor-preserving automorphisms. Moreover, there is an isomorphism of formal groups
\begin{align}
    \aut_G(\tilde{Y}) \simeq \aut_G(\tilde{Y})^{\circ} \rtimes \underline{J_b(\qp)},
\end{align}
where $\aut_G(\tilde{Y})^{\circ}$ is the intersection of $\aut(\tilde{Y})^{\circ}$ with $\aut_G(\tilde{Y})$ inside $\aut(\tilde{Y})$ and $J_b(\Qp) \subseteq G(\qpbreve)$ is the twisted centraliser of $b$.
\end{Lem}
\begin{proof}
It is clear from the definition that the exponential of a nilpotent tensor-annihilating endomorphism of $\tilde{Y}_R$ is a unipotent tensor-preserving automorphism of $\tilde{Y}_R$. Conversely the logarithm of a unipotent tensor-preserving automorphism of $\tilde{Y}_R$ is a nilpotent tensor-annihilating endomorphism of $\tilde{Y}_R$. Since $\aut(\tilde{Y})^{\circ} \subseteq \aut(\tilde{Y})$ is precisely the subgroup of unipotent automorphisms, the exponential map defines an isomorphism of functors
\begin{align}
    \mathcal{H}^{G,\circ}_{Y} \simeq \aut_G(\tilde{Y})^{\circ},
\end{align}
and thus $\aut_G(\tilde{Y})^{\circ} \subseteq \aut(\tilde{Y})^{\circ}$ is representable in closed immersions. We have seen that there is a semi-direct product decomposition
\begin{align}
    \aut(\tilde{Y}) \simeq \aut(\tilde{Y})^{\circ} \rtimes \underline{\aut(\tilde{Y})(\ovfp)}.
\end{align}
There is a closed subgroup of the locally profinite group $\underline{\aut(\tilde{Y})(\ovfp)}$ consisting of those automorphisms of $\tilde{Y}$ that are tensor-preserving. By Dieudonn\'e theory we can identify this group with the group of tensor-preserving automorphisms of the $F$-isocrystal $\mathbb{D}(Y)[\tfrac{1}{p}]$; this group is precisely the twisted centraliser $J_b(\Qp) \subseteq G(\qpbreve)$ of $b$. 

The action of $J_b(\Qp) \subseteq G(\qpbreve)$ on $\operatorname{Lie} G$ stabilises $\operatorname{Lie} U_{\nu_b}$ since $J_b(\Qp)$ is contained in the centraliser of $\nu_b$ inside $G(\qpbreve)$. Therefore we get a closed subgroup
\begin{align}
    \aut_G(\tilde{Y})^{\circ} \rtimes \underline{J_b(\qp)} \subseteq \aut(\tilde{Y})^{\circ} \rtimes \underline{\aut_G(\tilde{Y})(\ovfp)},
\end{align}
whose $R$-points for qrsp $R$ give the group of tensor-preserving automorphisms of $\tilde{Y}_R$.
\end{proof}
\begin{Rem}
    Our construction does \emph{not} agree with \cite[Definition 2.3.1]{KimCentralLeaves}, which defines $\aut_G(\tilde{Y})$ as the intersection of 
\begin{align} \label{Eq:Alternativedefinition}
    \left(\widetilde{\mathcal{H}}^G_{Y} \times \widetilde{\mathcal{H}}^G_{Y}\right) \cap \aut(\tilde{Y}).
\end{align}
By the discussion above, the $R$-points of this functor are given by automorphisms $g$ of $\tilde{Y}_R$ such that the induced automorphism
\begin{align} 
    \mathbb{D}(g):\mathbb{D}(Y)[\tfrac{1}{p}] \otimes_{\qpbreve} B_{\mathrm{cris}}^+(R) \to \mathbb{D}(Y)[\tfrac{1}{p}] \otimes_{\qpbreve} B_{\mathrm{cris}}^+(R)
\end{align}
satisfies $g^{\otimes} (s_{\alpha} \otimes 1)=0$ for all $\alpha$. Note that an automorphism $g$ of $\tilde{Y}_R$ induces an automorphism of $\mathbb{D}(Y)[\tfrac{1}{p}] \otimes_{\qpbreve} B_{\mathrm{cris}}^+(R)$. Thus if $g$ is an $R$-point of this intersection then $g^{\otimes} (s_{\alpha} \otimes 1)$ cannot be zero unless $s_{\alpha} \otimes 1=0$ which implies that $s_{\alpha}=0$. Therefore the functor \eqref{Eq:Alternativedefinition} is empty unless $G=\operatorname{GL}_V$.
\end{Rem}
\subsubsection{} Fortunately, the rest of \cite{KimCentralLeaves} only uses the characterisation of $\aut_G(\tilde{Y})$ of Lemma \ref{Lem:SubgroupClosed}. In particular, \cite[Proposition 3.2.4]{KimCentralLeaves} is correct as stated. Therefore the rest of \cite{KimCentralLeaves} is not affected. We end this section by defining
\begin{align}
    \aut_G(Y) =  \aut_G(\tilde{Y}) \times_{\aut(\tilde{Y})} \aut(Y)
\end{align}
so that
\begin{align}
    \aut_G(Y)^{\circ} =  \aut_G(\tilde{Y})^{\circ} \times_{\aut(\tilde{Y})^{\circ}} \aut(Y)^{\circ}
\end{align}
\begin{Rem}
If  $p \gg 0$, then $\mathbb{D}(\mathcal{H}_{Y}^{\circ})$ is plain and  the exponential map for $\aut(Y)^{\circ}$ has no denominators $p$. We can then identify the group schemes
\begin{align}
\Pi(\mathbb{D}(\mathcal{H}_{Y}^{\circ})) \simeq \aut_G(Y)^{\circ}.
\end{align}
In general, there is no containment in either direction. 
\end{Rem}
\subsection{Strongly nontrivial actions} \label{Sec:Action} Let $(\mfap, \varphi_{\mfap}, [-,-])$ be a plain \DLz \ with associated Dieudonn\'e--Lie $\qpbreve$-algebra $(\mfa, \varphi_{\mfa}, [-,-])$. Let $\aut(\mfa, \varphi_{\mfa})$ be the automorphism group of the underlying $F$-isocrystal, considered as an algebraic group over $\qp$ (as in \cite[Section 1.1]{RapoportRichartz}). Then the Lie algebra of $\aut(\mfa, \varphi_{\mfa})$ can be identified with the endomorphism algebra of the $F$-isocrystal, equipped with the commutator bracket. 

There is a closed subgroup
\begin{align}
    \aut(\mfa, \varphi_{\mfa},[-,-]) \subseteq \aut(\mfa, \varphi_{\mfa})
\end{align}
consisting of those automorphisms preserving the Lie bracket. 

\begin{Def}\label{Def:StronglyNonTrivial}
Let $Q$ be an algebraic group over $\qp$ equipped with a group homomorphism
\begin{align}
    Q \to \aut(\mfa, \varphi_{\mfa},[-,-]).
\end{align}
We call such a group homomorphism (or action) \emph{strongly non-trivial} if the induced linear representation of $Q$ on $\mfa$ does not have trivial subquotients.
\end{Def}
\begin{Vb}
If $\mfa=\operatorname{Lie} U_{\nu_b}$ for some element $b \in G(\qpbreve)$ admissible with respect to some Shimura datum on $G$, then the algebraic group $J_b$ has a natural and strongly non-trivial action on $\mfa$. Moreover, by Proposition 5.5.1 of \cite{OrdinaryHO} the restriction of this action to a maximal torus $T \subseteq J_b$ is still strongly non-trivial.
\end{Vb}

\section{Formal neighbourhoods of central leaves} \label{Sec:CentralLeaves}
\subsection{Introduction}
In this section we will recall the constructions of the canonical integral models of Shimura varieties of Hodge type at hyperspecial level from \cite{KisinModels}. We then recall the definitions of central leaves $\cb$ inside the special fibers $\shg$ of these canonical integral models, and also the definition of Igusa varieties from \cites{KimCentralLeaves, Hamacher, HamacherKim}. The main goal of this section is to study the structure of the formal completions $\cb^{/x}$ of central leaves at points $x \in \cb(\ovfp)$. We will show, using the work of Caraiani--Scholze and Kim, that the perfection of the formal scheme $\cb^{/x}$ has a  canonical (generally non-commutative) group structure. These groups are of the form $\tildeu(\mfa)$ for a Dieudonn\'e--Lie $\qpbreve$-algebra $\mfa$, as introduced in Section \ref{Sec:AutoEndo}. 

In Section \ref{Sec:StronglyTateLinear} we will show that $F$-stable Lie subalgebras $\mfb \subseteq \mfa$ give rise to formally smooth closed subschemes of $\cb^{/x}$. These are precisely the \emph{strongly Tate-linear} subspaces of Chai and Oort. We end by stating a conjecture on the monodromy group of the universal isocrystal over such subspaces.

\subsection{Integral models} \label{Sec:IntegralModels} For a symplectic space $(V, \psi)$ over $\mathbb{Q}$, we write $\g_V \coloneqq \operatorname{GSp}(V, \psi)$ for the group of symplectic similitudes of $V$ over $\mathbb{Q}$. It admits a Shimura datum $\gvx$, where $\mathsf{H}_V$ is the union of the Siegel upper and lower half-spaces. Let $\gx$ be a Shimura datum of Hodge type with reflex field $\mathsf{E}$ and let $\gx \to \gvx$ be a Hodge embedding. 

Fix a prime $p>2$ and a place $v$ above $p$ of the reflex field $\mathsf{E}$, and write $G=\g \otimes \qp$, write $E=\mathsf{E}_v$ for the $v$-adic completion of $\mathsf{E}$ and $\mathcal{O}_E$ for its ring of integers. Let $K_p \subseteq G(\qp)$ be a hyperspecial subgroup. Then after possibly changing the Hodge embedding and the symplectic space $V$, we can find a self dual $\mathbb{Z}_{(p)}$-lattice $V_{(p)}$, such that $K_p$ is the stabiliser in $G(\qp)$ of $V_p\coloneqq V_{(p)} \otimes_{ \mathbb{Z}_{(p)}} \zp$, see Section 2.3.15 of \cite{KisinPappas}. Write $G_{\mathbb{Z}_{(p)}}$ for the Zariski closure of $G$ in $GL(V_{\mathbb{Z}_{(p)}})$, then $G_{\mathbb{Z}_{(p)}}\otimes_{\mathbb{Z}_{(p)}}\mathbb{Z}_p$ is a reductive integral model $\mathcal{G}$ of $G$. 

For every sufficiently small compact open subgroup $U^p \subseteq \gafp$, we can find $\mathcal{U}^p \subseteq \g_V(\afp)$ such that the Hodge embedding induces a closed immersion (see Lemma 2.1.2 of \cite{KisinModels})
\begin{align}
    \mathbf{Sh}_{U}\gx \to \mathbf{Sh}_{\mathcal{U}}\gvx \otimes_{\mathbb{Q}} \mathsf{E}
\end{align}
of (canonical models of) Shimura varieties of levels $U=U^pU_p$ and $\mathcal{U}=\mathcal{U}^p \mathcal{U}_p$, respectively. We let $\mathcal{S}_{\mathcal{U}}\gvx$ over $\zp$ be the moduli-theoretic integral model of $\mathbf{Sh}_{\mathcal{U}}\gvx$; it is a moduli space of polarised abelian schemes $(A, \lambda)$ up to prime-to-$p$ isogeny with level $\mathcal{U}^p$-structure. Let
\begin{align}
    \mathscr{S}_{U}\coloneqq \mathscr{S}_U\gx \to \mathcal{S}_{\mathcal{U}}\gvx \otimes_{\zp} \mathcal{O}_{E}
\end{align}
be the normalisation of the Zariski closure of $\mathbf{Sh}_{U}\gx$ in $\mathcal{S}_{\mathcal{U}}\gvx \otimes_{\zp} \mathcal{O}_{E}$. Then by the main result of \cite{KisinModels} and \cite{MadapusiPeraKim}, see Section 2 of \cite{MadapusiPeraKim}, the scheme $\mathbf{Sh}_{U}\gx$ is smooth and in fact isomorphic to the canonical integral model of $\mathbf{Sh}_U\gx$. The main result of \cite{xu2020normalization} tells us that 
\begin{align}
    \mathscr{S}_{U}\to \mathcal{S}_{\mathcal{U}} \otimes_{\mathbb{Z}_{p}} \mathcal{O}_{E,v}
\end{align}
is a closed immersion. Choose an algebraic closure $\ovfp$ of the residue field of $\mathcal{O}_{E}$ and let $\shg$ be the base change to $\ovfp$ of $\mathscr{S}_{U}$. We will write $\shgv$ for the base change to $\ovfp$ of $\mathcal{S}_{\mathcal{U}}$. Then the pullback of the universal abelian scheme (up to prime-to-$p$ isogeny) over $\mathcal{S}_{\mathcal{U}}$ gives rise to an abelian scheme (up to prime-to-$p$ isogeny) $A$ over $\operatorname{Sh}_{G,U}$ with associated $p$-divisible group $X=A[p^{\infty}]$. 

\subsubsection{Tensors} Recall the notation $V_{\mathbb{Z}_{(p)}}^\otimes$ from Section \ref{Sec:Tensors}. By \cite[Lemma 1.3.2]{KisinModels}, the subgroup $G_{\mathbb{Z}_{(p)}} \subseteq GL(V_{\mathbb{Z}_{(p)}})$ is the stabiliser of a collection of tensors $s_\alpha \in V_{\mathbb{Z}_{(p)}}^\otimes$.

For $x \in \shg(\ovfp)$, we write $A_x$ for the abelian variety up to prime-to-$p$ isogeny over $\ovfp$ corresponding to the image of $x \in \shgv$. Let $x \in \shg(\ovfp)$ and let $\mathbb{D}_{\mathrm{contr},x}$ be the (contravariant!)\footnote{As in \cite{OrdinaryHO}, we use both covariant and contravariant Dieudonn\'e theory in this paper. Since we mostly use the covariant theory, we will denote all our contravariant Dieudonn\'e modules with the subscript $_\mathrm{contr}$. The reason for this is that most of the literature on integral models of Shimura varieties uses the contravariant theory, while results about internal-Hom $p$-divisible groups are best expressed in the covariant theory.} Dieudonn\'e module of $A_x[p^{\infty}]$. It is explained in \cite[Section 6.3]{ShankarZhou} that there are canonical tensors $\{s_{\alpha,\mathrm{cris},x}\}$ in $
    \mathbb{D}_{\mathrm{contr},x}^{\otimes},$
that are invariant under the Frobenius on $\mathbb{D}_{\mathrm{contr},x}[\tfrac{1}{p}]$. It is moreover explained there that there is an isomorphism
\begin{align} \label{Eq:Basis}
   \mathbb{D}_{\mathrm{contr},x} \simeq V_{(p)} \otimes_{\mathbb{Z}_{(p)}} \zpbreve
\end{align}
taking $s_{\alpha,\mathrm{cris},x}$ to $s_{\alpha} \otimes 1$. Under such an isomorphism, the Frobenius is given by an element $b_x \in G(\qpbreve)$, which is well defined up to $\sigma$-conjugacy by $\mathcal{G}(\zpbreve)$, where $\sigma:G(\qpbreve) \to G(\qpbreve)$ is induced by the Frobenius $\sigma:\qpbreve \to \qpbreve$ (which we called $\varphi$ before). We will denote by $\llbracket b_x\rrbracket$ the $\mathcal{G}(\zpbreve)$-$\sigma$-conjugacy class of $b_x$ and by $[b_x]$ the $G(\qpbreve)$-$\sigma$-conjugacy class of $b_x$. 

\subsection{Central leaves} 
It follows from \cite[Corollary 3.3.8]{HamacherKim} that for $b \in G(\qpbreve)$ there are (reduced) locally closed subschemes
\begin{align}
    \cb \subseteq \shgb \subseteq \shg
\end{align}
of $\shg$ such that their $\ovfp$-points are given by
\begin{align}
    \cb(\ovfp)&=\{x \in \shg(\ovfp) \; : \; \llbracket b_x\rrbracket=\llbracket b \rrbracket \} \\
    \shgb(\ovfp)&=\{x \in \shg(\ovfp) \; : \; [b_x]=[b] \}.
\end{align}
The subscheme $\shgb$ is called the \emph{Newton stratum} attached to $[b]$, and the subscheme $\cb \subseteq \shgb$ is called the \emph{central leaf} attached to $\llbracket b \rrbracket$. We note that the natural map $\cb \to \shgb$ is a closed immersion by \cite[Corollary 3.3.8]{HamacherKim} and that the central leaf $\cb$ is smooth and equidimensional by \cite[Corollary 5.3.1]{KimCentralLeaves}. The following remark is \cite[Remark 2.1.4]{vHXiao}.

\begin{Rem} \label{Rem:ChaiOortLeaves}
When $\gx=\gvx$, then the $\mathcal{G}(\zpbreve)$-conjugacy class $\llbracket b_x \rrbracket$ captures precisely the isomorphism class of the polarised $p$-divisible group $(A_x[p^{\infty}], \lambda_x)$, where an isomorphism of polarised $p$-divisible groups $f:(Y, \mu) \to (Y', \mu')$ is an isomorphism $f:Y \to Y'$ such that $f^{\ast} \mu'=c \mu$ for some $c \in \mathbb{Z}_p^{\times}$. In particular, when $\gx=\gvx$ our central leaves do \emph{not} agree with those defined in \cite{ChaiOort}, which are defined using isomorphisms $f:(Y, \mu) \to (Y', \mu')$ with $f^{\ast} \mu'=\mu$. However, note that both variants central leaves are closed in their respective Newton strata and are smooth of the same dimension; so our central leaves are finite disjoint unions of the central leaves considered in \cite{ChaiOort}.
\end{Rem}
\subsubsection{} Fix a point $x \in \shg(\ovfp)$ and write $Y=A_x[p^{\infty}]$ and write $\lambda$ for the induced polarisation. We write $\llbracket b \rrbracket \coloneqq  \llbracket b_x \rrbracket$ for the $\mathcal{G}(\zpbreve)$-$\sigma$-conjugacy class of elements of $G(\qpbreve)$ associated to $x$. We write $\cgsp \subseteq \shgv$ for the central leaf associated to the polarised $p$-divisible group $(Y, \lambda)$. \smallskip 

Let $X=(A[p^{\infty}], \lambda_X)$ for the universal polarised $p$-divisible group over $\cgsp$. We let 
\begin{align}
    \igCSgsp  \to \cgsp,
\end{align}
be the fpqc sheaf over $\cgsp$ which sends a scheme $T \to \cgsp$ to the set of isomorphisms $X_T \xrightarrow{\sim} Y_{T}$ compatible with the polarisation $\lambda_X$ and $\lambda$ up to a scalar in $\underline{\zp^{\times}}(T)$. This is tautologically a quasi-torsor for the closed subgroup scheme $\aut_{\lambda}(Y) \subseteq \aut(Y)$ of automorphisms that preserve polarisations up to a scalar. It follows from \cite[Proposition 4.3.3]{CaraianiScholze} that $\igCSgsp \to \cgsp$ is representable by a scheme, and from \cite[Corollary 4.3.5]{CaraianiScholze} that this scheme is perfect. They moreover show that $\igCSgsp  \to \cgsp$ is faithfully flat, see \cite[Corollary 4.3.9]{CaraianiScholze}, and thus a torsor for $\aut_{\lambda}(Y)$ in the fpqc topology. Finally, by \cite[Corollary 4.3.5]{CaraianiScholze}, the action of $\aut_{\lambda}(Y)$ on $\igCSgsp$ extends to an action of $\aut_{\lambda}(\tilde{Y})$, the group of automorphisms of $\tilde{Y}$ preserving the polarisation up to a scalar in $\underline{\qp^{\times}}$.

We now make the following assumption. 
\begin{Hyp} \label{Hyp:CSD}
The $p$-divisible group $Y$ is completely slope divisible. 
\end{Hyp}
It follows from \cite[Proposition 2.4.5]{KimCentralLeaves} that for every Newton stratum $\shgb \subseteq \shg$ we can always find a central leaf $\cbp \subseteq \shgb$ such that $\cbp \subseteq \cgspp$ where $Y'$ is a completely slope divisible $p$-divisible group. Thus this is not an unreasonable assumption. 

As explained in \cite[Section 3.2.3]{MantovanPEL}, this implies that the universal $p$-divisible group $X=A[p^{\infty}]$ over $\cgsp$ admits a slope filtration and we will denote the associated graded pieces for the slope filtration by $X_i$. We can then consider the Mantovan Igusa variety  
\begin{align}
    \iggsp \to \cgsp,
\end{align}
which represents the functor over $\cgsp$ sending $T \to \cgsp$ to the set of isomorphisms $X_{i,T} \xrightarrow{\sim} Y_{i, T}$ compatible with the polarisations up to a scalar in $\underline{\zp^{\times}}(T)$. Note that $\iggsp \to \cgsp$ is tautologically a quasi-torsor for the closed subgroup scheme of $\prod_i \aut(Y_i)$ of isomorphisms that preserve polarisations up to a scalar. Under the identification
\begin{align}
    \prod_i \aut(Y_i) \xrightarrow{\sim} \underline{\aut(Y)(\ovfp)}
\end{align}
of Section \ref{subsub:SemiDirectProduct}, we can identify this closed subgroup with the profinite group scheme $\underline{\aut_{\lambda}(Y)(\ovfp)} \subseteq \underline{\aut(Y)(\ovfp)}$. It follows from work of Mantovan, see \cite[Proposition 1, Proposition 4]{MantovanPEL}, that $\cgsp$ is smooth and that $\iggsp \to \cgsp$ is representable and profinite pro-\'etale. In particular, $\iggsp \to \cgsp$ is a torsor for $\underline{\aut_{\lambda}(Y)(\ovfp)}$. 

\subsubsection{} There is an obvious natural map $\igCSgsp \to \iggsp$ over $\cgsp$ sending an isomorphism to the induced isomorphism on the associated graded of the slope filtration. This morphism is equivariant for the morphism $\aut_{\lambda}(Y) \to \underline{\aut(Y)(\ovfp)}$ of Section \ref{subsub:SemiDirectProduct}. By \cite[Proposition 4.3.8]{CaraianiScholze}, this map identifies $\igCSgsp \to \iggsp$ with the perfection of the scheme $\iggsp$. 
\begin{Lem}
    The following commutative diagram is Cartesian
    \begin{equation} \label{eq:CartesianGSp}
    \begin{tikzcd}
    \igCSgsp \arrow{d} \arrow{r} & \iggsp \arrow{d} \\
    \cgsp^{\mathrm{perf}} \arrow{r} & \cgsp,
    \end{tikzcd}
\end{equation}
\end{Lem}
\begin{proof}
Since profinite pro-\'etale covers of perfect schemes are again perfect schemes, we see that the fibre product of \eqref{eq:CartesianGSp} is given by the perfection of $\iggsp$. Since $\igCSgsp \to \iggsp$ identifies $\igCSgsp$ with the perfection of $\iggsp$ by \cite[Proposition 4.3.8]{CaraianiScholze}, we are done.
\end{proof}

\subsection{Igusa varieties for Shimura varieties of Hodge type} \label{Sec:CentralLeavesHodgeType} Let $x \in \shg(\ovfp)$ as above and choose an isomorphism $$\mathbb{D}_{\mathrm{contr},x} \simeq V_{(p)} \otimes_{\mathbb{Z}_{(p)}} \zpbreve$$
taking $s_{\alpha,\mathrm{cris},x}$ to $s_{\alpha} \otimes 1$ as in \eqref{Eq:Basis}. This induces an isomorphism from $V_{(p)}^{\ast} \otimes \zpbreve$ to the covariant Dieudonn\'e module $\mathbb{D}(Y)$ and thus gives us Frobenius invariant tensors $\{s_{\alpha,\mathrm{cris},x}\} \subseteq \mathbb{D}(Y)^{\otimes}$. Let $b \in G(\qpbreve) \subseteq \operatorname{GL}(V^{\ast})(\qpbreve)$ be the element corresponding to the Frobenius on $\mathbb{D}(Y)[\tfrac{1}{p}]$. Under such an isomorphism, the Frobenius is given by an element $b=b_x \in G(\qpbreve) \subseteq \operatorname{GL}(V^{\ast})(\qpbreve)$. In particular, we can apply the result of Section \ref{Sec:AutHodge} and form the objects $\mathcal{H}_Y^{G, \circ} \subseteq \mathcal{H}_Y$ and $\aut_G(\tilde{Y})$. We then consider the Dieudonn\'e--Lie $\zpbreve$-algebra $\mfap=\mathbb{D}(\mathcal{H}_Y^{G, \circ})$ and $\mfa=\mfap \otimes \qpbreve$. We have seen in Section \ref{Sec:AutHodge} that $\mfa \simeq \operatorname{Lie} U_{\nu_b}$ and that $\aut_G(\tilde{Y})^{\circ} \simeq \tilde{\Pi}(\mfa)$. 

\subsubsection{} Hamacher and Kim define an Igusa variety $\igCS \to \cb^{\mathrm{perf}}$ by pulling back $\igCSgsp \to \cgsp^{\mathrm{perf}}$ along $\cb^{\mathrm{perf}} \to \cgsp^{\mathrm{perf}}$, and then taking the closed subset where the pull-back of the universal isomorphism 
\begin{align}
    X \simeq Y_{\igCSgsp}
\end{align}
preserves the tensors on geometric points, see \cite[Section 5.1 and Lemma 5.1.1]{HamacherKim}. They then prove that $\igCS \to \cb^{\mathrm{perf}}$ is a pro-\'etale torsor for $\underline{\aut_G(Y)(\ovfp)} \subseteq \underline{\aut_{\lambda}(Y)(\ovfp)}$ and that $\igCS$ is stable under the action of $\aut_G(\tilde{Y}) \subseteq \aut(\tilde{Y})$. Using the invariance of the \'etale site under perfections, see \cite[Tag 0BTY]{stacks-project}, we get a pro-\'etale torsor $\ig \to \cb$ sitting in a cartesian diagram
\begin{equation}
    \begin{tikzcd}
    \igCS \arrow{r} \arrow{d} & \ig \arrow{d} \\
    \cb^{\mathrm{perf}} \arrow{r} & \cb,
    \end{tikzcd}
\end{equation}
that fits in a commutative cube with \eqref{eq:CartesianGSp}. The map $\igCS \to \cb$ is faithfully flat since $\igCS \to \cb^{\mathrm{perf}}$ is a pro-étale cover and $\cb$ is smooth. We want to prove that it is actually an fpqc torsor under $\aut_G(Y)$. We first prove it at the level of infinitesimal neighbourhoods (or rather formal completions in a point, as defined as in \cite[Tag 0AIX]{stacks-project}).

\begin{Lem} \label{Lem:FormalTorsor}
If $y \in \igCS(\ovfp)$ is an $\ovfp$-point over $x\in \cb(\ovfp)$, then $\igCS^{/y} \to \cb^{/x}$ is a torsor in the fpqc topology under the group scheme $\aut^{\circ}_G(Y)$. 
\end{Lem}

\begin{proof}
In \cite[Section 5.2]{HamacherKim}, Hamacher and Kim show that the formal completion at a point $y$ of $\igCS$ is isomorphic to $\aut^{\circ}_G(\tilde{Y})$ and that the natural action of $\aut^{\circ}_G(\tilde{Y})$ on $\igCS$ corresponds to the multiplication map under this isomorphism. Moreover, the morphism $\igCS^{/y} \to \cb^{/x}$ corresponds to the restriction of the action map (constructed in \cite[Theorem 4.3.1]{KimCentralLeaves})
\begin{align}
    \aut^{\circ}_G(\tilde{Y}) \times \shgb^{/x} \to \shgb^{/x}
\end{align}
to the closed point $x \in \shgb^{/x}$. It follows from \cite[Theorem 5.1.3]{KimCentralLeaves} that the scheme-theoretic image of this restriction
\begin{align}
    \aut^{\circ}_G(\tilde{Y})  \to \shgb^{/x}
\end{align}
is $\cb^{/x} \subseteq \shgb^{/x}$ and in fact this identifies $\aut^{\circ}_G(\tilde{Y})$ with the perfection of $\cb^{/x}$. Now we apply this to $\gx=\gvx$, where we already know that 
\begin{align}
    \igCSgsp \to \iggsp
\end{align}
is a torsor for $\aut_{\lambda}^{\circ}(Y)$. Since the map $\iggsp \to \cgsp$ is pro-étale, it induces isomorphisms of formal completions (see the proof of \cite[Proof of Proposition 6.1.1]{OrdinaryHO})
\begin{align}
    \iggsp^{/z} \to \cgsp^{/x},
\end{align}
where $z$ is the image of $y$ under $\igCSgsp \to \iggsp$. It follows that on formal neighborhoods we get an equivariant map
\begin{align}
    \aut_{\lambda}^{\circ}(\tilde{Y}) \to \cgsp^{/x}
\end{align}
which is an $\aut_{\lambda}^{\circ}(Y)$-torsor in the fpqc topology. Because $\aut_{G}^{\circ}(Y)$ is the intersection of $\aut_{\lambda}^{\circ}(Y)$ with $\aut_G(\tilde{Y})$, it follows that $\aut^{\circ}_G(Y)\subseteq \aut^{\circ}_G(\tilde{Y})$ is the stabiliser of $x\in \cb^{/x}$. Therefore, the natural map
\begin{align}
    \aut_{G}^{\circ}(\tilde{Y}) \to \cb^{/x}
\end{align}
is a quasi-torsor for $\aut_{G}^{\circ}(Y)$. It is also an fpqc cover, because $\igCS \to \cb$ is, and thus it is a torsor for $\aut_{G}^{\circ}(Y)$ in the fpqc topology. \end{proof}

\begin{Prop} \label{Prop:TorsorHodgeType}
The map $\igCS \to \cb$ is an fpqc torsor under $\aut_G(Y)$.
\end{Prop}
\begin{proof}
Since $\igCS \to \cb$ is faithfully flat, it suffices to prove that it is a quasi-torsor under $\aut_G(Y)$. In other words, we want to show that the action map
\begin{align} \label{eq:action}
    \aut_G(Y) \times \igCS \to \igCS \times_{\cb} \igCS
\end{align}
is an isomorphism. This map is clearly a homeomorphism because $\ig \to \cb$ is an $\underline{\aut_G(Y)(\ovfp)}$-torsor and both $\igCS \to \ig$ and $\aut_G(Y) \to \underline{\aut_G(Y)(\ovfp)}$ are universal homeomorphisms.

It follows from \cite[Lemma 5.2.3]{HamacherKim} that when $Z$ is either the source or the target of \eqref{eq:action} and $z\in Z(\ovfp)$ is an $\ovfp$-point over $x$, then $Z^{/z}$ is pro-represented by the formal spectrum of the $I$-adic completion of $\mathcal{O}_{Z,x}$, where $I$ is the maximal ideal of $\mathcal{O}_{\cb,x}$. Moreover, they prove that $\mathcal{O}_{Z,x} \to \mathcal{O}_{Z,x}^{\wedge,I}$ is faithfully flat.

It follows from Lemma \ref{Lem:FormalTorsor} that the action map \eqref{eq:action} is flat at all closed points and therefore it is flat by \cite[Lemma 00HT.(7)]{stacks-project}. We moreover know that the action map is a closed immersion because this is true in the Siegel case, so the action map is a surjective flat closed immersion and therefore an isomorphism (see
\cite[Tag 04PW]{stacks-project}).
\end{proof}

\subsubsection{Sustained deformation space} \label{Sec:DeformationTheoryLeaves} Chai and Oort proved that the deformation theory of $\cgsp$ is completely determined by the $\aut_{\lambda}(Y)$-torsor $\igCSgsp \to \cgsp$. 

\begin{Thm}[{\cite[Lemma 3.6, Theorem 4.3]{ChaiFoliation}}]\label{Thm:ChaiOortSustained}
    There exists a formal scheme $\mathbf{Def}_{\mathrm{sus}}(Y, \lambda)$ such that, when evaluated on Artinian local $\ovfp$-algebras, represents deformations of the trivial $\aut_{\lambda}(Y)$-torsor over $\ovfp$. This formal scheme is  formally smooth and formally of finite type. In addition, the morphism \begin{align}
    \cgsp^{/x} \to \mathbf{Def}_{\mathrm{sus}}(Y, \lambda),
\end{align} which for every Artinian $\ovfp$-algebra $R$ maps an $R$-point  of $\cgsp^{/x}$ to its fibre with respect to $\igCSgsp \to \cgsp^{/x}$, is an isomorphism of formal schemes.
\end{Thm}It is possible to similarly show that $\cb^{/x}$ can be identified with the deformation space of the trivial $\aut_{G}(Y)$-torsor, but this result does not enter into our proof.
\subsubsection{} \label{Sec:NoPolarisation} Let us analyse more in detail the total deformation space of $Y$, without taking into account for the moment any extra structures. Recall that there exists a formally smooth and formally of finite type formal scheme $\textbf{Def}\, (Y)$ such that, on Artinian local $\ovfp$-algebras, it represents all the deformations of $Y$. This formal scheme can be interpreted as the reduction modulo $p$ of the formal completion of a Rapoport--Zink space (see \cite[Section 4.3]{KimCentralLeaves} or \cite[Lemma 4.2.2]{OrdinaryHO}). Using this description, for every $R\in \algo{\ovfp}$ the set
    $\textbf{Def}\, (Y)(R)$ parametrises isomorphism classes of $(X,\iota)$ where $X$ is a $p$-divisible group over $\Spec R$ and $\iota\colon X\dashrightarrow Y_R$ is a quasi-isogeny which becomes an isomorphism of $p$-divisible groups on $\Spec R^{\mathrm{red}}$. The formal group $\aut(\tilde{Y})^\circ$ acts (formally) on $\textbf{Def}\, (Y)$ by sending $$(\gamma,(X,\iota))\in \aut(\tilde{Y})^\circ(R)\times \textbf{Def}\, (Y)(R)$$ to $(X,\gamma\circ\iota)$. The action is well-defined because $\aut(\tilde{Y})^\circ(R_\mathrm{red})=1$. The stabiliser of the closed point $x\in \textbf{Def}\, (Y)$ is given by those $\gamma\in \aut(\tilde{Y})^\circ(R)$ such that $(Y_R,\gamma)\simeq (Y_R,\mathrm{id})$. This corresponds to the closed formal subscheme $\aut(Y)^\circ\subseteq \aut(\tilde{Y})^\circ$.

As in Theorem \ref{Thm:ChaiOortSustained} it follows from \cite[Lemma 3.6, Theorem 4.3]{ChaiFoliation} that there exists a formally smooth closed formal subscheme $\defsus(Y)\subseteq \textbf{Def}\, (Y)$ which represents deformations of the trivial $\aut(Y)$-torsor on Artinian $\ovfp$-algebras. It can be identified with the subfunctor which consists of those $(X,\iota)$ with $X$ fpqc locally isomorphic to the base change of $Y$. It follows that the $\aut(\tilde{Y})^\circ$-action on $\textbf{Def}\, (Y)$ preserves $\defsus(Y)$. There is a well-defined map $\aut(\tilde{Y})^\circ \to \defsus(Y)$ given by the orbit through the closed point.  

\begin{Lem}
    The orbit map through the closed point $\aut(\tilde{Y})^\circ \to \defsus(Y)$ is a torsor in the fpqc topology for the natural left translation action of $\aut(Y)^\circ$ on $\aut(\tilde{Y})^\circ$.
\end{Lem}
\begin{proof}
Write $\defsus^{\mathrm{perf}}(Y)$ for the perfection of $\defsus(Y)$ as a formal scheme. It follows from the universal property of the perfection that the orbit map $\aut(\tilde{Y})^\circ \to \defsus(Y)$ lifts to a map $\aut(\tilde{Y})^\circ \to \defsus^{\mathrm{perf}}(Y)$. By \cite[Theorem 5.1.3]{KimCentralLeaves}, this induced map is an isomorphism. The map $\defsus^{\mathrm{perf}}(Y)\to \defsus(Y)$ is an fpqc cover because $\defsus(Y)$ is the formal spectrum of a complete regular Noetherian local ring, and thus it is Frobenius-smooth. We deduce that the orbit map $\aut(\tilde{Y})^\circ \to \defsus(Y)$ is an fpqc cover. Since it is clearly a quasi-torsor for the natural action of $\aut(Y)^\circ$, the stabilizer of the closed point $x$, we see that it must be a torsor.
\end{proof}

\subsubsection{} \label{subsub:DescriptionTorsor} We would now like to identify the $\aut(Y)^\circ$-torsor over $\defsus(Y)$ that we have constructed above. Consider the fpqc sheaf $P$ over $\algo{\ovfp}$ sending $R$ to the set of isomorphism classes of triples $(X, \iota, \alpha)$, where $(X, \iota)$ is as above and $\alpha:X \to Y_T$ is an isomorphism that agrees with $\iota$ on $\spec R^{\mathrm{red}}$.\footnote{If $R$ is Noetherian, then the map $R \to R^{\mathrm{red}}$ has finitely generated kernel and hence the quasi-isogeny $\gamma$ is the unique lift of the isomorphism over $R^{\mathrm{red}}$; this forces $\alpha=\gamma$. Thus we see that maps $\spec R \to P$ with $R$ a Noetherian $\ovfp$-algebra factor through the closed point $\spec \ovfp \to P$. This reflects the fact that $P \simeq \aut(\tilde{Y})^{\circ}$ is the perfection of the formal spectrum of a power series ring.} There is a natural map $P \to \defsus(Y)$ given by forgetting $\alpha$, and this map is a quasi-torsor for $\aut(Y)^\circ$ by construction. There is a natural map
\begin{equation}
\label{eq:UniversalIsomorphism}\aut(\tilde{Y})^\circ \to P
\end{equation}
which on $\spec R$ points sends $\gamma:Y_{R} \to Y_{R}$ to $(Y_{R}, \gamma, \operatorname{Id})$; this is well defined since $\gamma$ is the identity on $\spec R^{\mathrm{red}}$ by assumption. This map is clearly $\aut(Y)^\circ$-equivariant and commutes with the natural map to $\defsus(Y)$ of the source and the target. It follows that it is an isomorphism since the source is a torsor and the target is a quasi-torsor (and thus a posteriori a torsor).

\begin{Rem} \label{Rem:TwoSlopes}
If $Y=Y_1 \oplus Y_2$ has two slopes, then $\aut(\tilde{Y})^{\circ}$ is isomorphic to $\widetilde{\mathcal{H}}_{Y_1,Y_2}$ and $\aut(Y)^{\circ} \simeq T_p \mathcal{H}_{Y_1,Y_2}$ so that $\defsus(Y) \simeq \mathcal{H}_{Y_1,Y_2}$. This gives $\defsus(Y)$ the structure of a $p$-divisible formal group. When $Y$ is ordinary, then $\defsus(Y)=\textbf{Def}\, (Y)$ and the formal group structure on $\textbf{Def}\, (Y)$ is the one coming from the classical Serre--Tate coordinates, see \cite[Section 4]{OrdinaryHO}. 

In both cases, the $p$-divisible formal group $\mathcal{H}_{Y_1,Y_2}$ is closely related to the (group-valued) functor on Artinian local rings of extensions of $Y_1$ by $Y_2$, see \cite{ChaiTwoSlopes}. In the ordinary case, the space of extensions is actually isomorphic to $\defsus(Y)=\textbf{Def}\, (Y)$ via the natural map. 
\end{Rem}
\subsection{Strongly Tate-linear subspaces} \label{Sec:StronglyTateLinear}
Let $Y$ be a completely slope divisible $p$-divisible group over $\ovfp$ and let $X$ be the universal $p$-divisible group over the sustained deformation space $\defsus(Y)$. Recall that by \cite[Lemma 2.4.4]{DeJongDieudonne}, the $p$-divisible group $X$, defined over a formal scheme, admits a unique algebraisation to the scheme associated to $\defsus(Y)$.  Thanks to \ref{Thm:BMDieudonne} we can then attach to $X$ its rational Dieudonné module $\calM=\mathbb{D}(X)[\tfrac 1p]$. By Corollary \ref{Cor:ExistenceSlopeFiltration}, the $F$-isocrystal $\calM$ admits the slope filtration. Since $\defsus(Y)$ is simply connected, the graded object $\mathrm{Gr}_{S_\bullet}(\calM)$ is trivial as an isocrystal (see \cite[Proposition 3.3.4]{AddezioII}).  

\subsubsection{}
Let $Z \subseteq \defsus(Y)$ be a formally smooth formal closed subscheme with closed point $z$ and write $\calM_Z$ for the restriction of $\calM$ to $Z$. The monodromy group $G(\calM_Z)$ of the isocrystal $\calM_Z$ with respect to $z$ admits by definition an embedding into the automorphism group scheme of the fibre $\calM_z$. Since $\calM_z$ is the Dieudonné module of $Y$, we get the inclusion $$ G(\calM_Z)\subseteq \mathrm{GL}(\mathbb{D}(Y)[\tfrac{1}{p}]).$$ Thanks to the fact that $\calM_Z$ admits the slope filtration and $\mathrm{Gr}_{S_\bullet}(\calM)$ is trivial as an isocrystal, we deduce that $G(\calM_Z)$ is contained in the unipotent radical of the parabolic subgroup of $\mathrm{GL}(\mathbb{D}(Y)[\tfrac{1}{p}])$ stabilising the slope filtration. By Lemma \ref{Lem:LieAlgebrasAndDieudonne} we deduce that we have the following chain of containments
\begin{align}
    \operatorname{Lie} G(\calM_Z) \subseteq \mathbb{D}(\mathcal{H}_{Y}^{\circ})[\tfrac 1p] \subseteq \mathbb{D}(\mathcal{H}_{Y}){[\tfrac 1p]}\subseteq\operatorname{Lie} \operatorname{GL}(\mathbb{D}(Y)){[\tfrac 1p]}.
\end{align}
 Since $\calM$ has the structure of an $F$-isocrystal, we get by \cite[Section 2.2]{CrewIsocrystals} an isomorphism
\begin{align}
   G(\calM_Z)^{(1)} \xrightarrow{\sim} G(\calM_Z),
\end{align}
which induces an isomorphism
\begin{align}
  \operatorname{Lie} G(\calM_Z)^{(1)} \xrightarrow{\sim} \operatorname{Lie} G(\calM_Z)
\end{align}
compatible with the $F$-structure on $\operatorname{Lie} \operatorname{GL}(\mathbb{D}(Y)){[\tfrac 1p]}$. In particular
\begin{align}
    \operatorname{Lie} G(\calM_Z) \subseteq \mathbb{D}(\mathcal{H}_{Y}^{\circ})[\tfrac 1p]=:\mfa
\end{align}
is an $F$-stable Lie subalgebra.

\subsubsection{} We have seen in Section \ref{Sec:NoPolarisation} that there is an action of $\aut(\tilde{Y})^{\circ}$ on $\defsus(Y)$ together with an equivariant map
\begin{align}
    \aut(\tilde{Y})^{\circ} \to \defsus(Y)
\end{align}
which is an $\aut(Y)^{\circ}$-torsor in the fpqc topology. If we consider $\mfap=\mathbb{D}(\mathcal{H}_{Y}^{\circ})$ as a nilpotent \DLz, then it is completely slope divisible by Lemma \ref{Lem:InternalHomCSD}. Nonetheless, it is generally \emph{not} integrable, so we will often consider $p^2 \mfap$ instead which is plain by Lemma \ref{Lem:ExistenceLattice}. 
\begin{Lem}  \label{Lem:Inclusion}
The exponential isomorphism $E\colon \tilde{\mathcal{H}}^\circ_Y\xrightarrow{\sim} \mathbf{Aut}(\tilde{Y})^\circ$ identifies
$\Pi(p^2 \mfap)$ with the normal  profinite  subgroup scheme $$\mathbf{Aut}(Y)^\circ_{p^2}\subseteq \mathbf{Aut}(Y)^\circ,$$ of those automorphisms which are trivial when restricted to $Y[p^2].$ In particular, we get a short exact sequence of fpqc sheaves
    \begin{align}
        1 \to \Pi(p^2 \mfap) \to \aut(Y)^{\circ} \to H \to 1,
    \end{align}
    with $H\subseteq \mathbf{Aut}(Y[p^2])$ a finite group scheme over $\ovfp$.
\end{Lem}
\begin{proof}The exponential restricts to a morphism $$ p^2T_p\mathcal{H}_Y^\circ\to \aut(Y)^{\circ}_{p^2},$$ since $\tfrac{p^{2n}}{n!}\in p^2\zp$ for every $n\geq 1$. Similarly, the logarithm induces a morphism $$\aut(Y)^{\circ}_{p^2}\to p^2T_p\mathcal{H}_Y^\circ.$$ These morphisms are mutual inverses. 

For the second part, noting that $\aut(Y[p^{2}])$ is an affine finite type group scheme,\footnote{The automorphism group of an affine finite $k$-scheme $\spec A$ is a closed subgroup of $\operatorname{GL}(A)$, hence affine and of finite type.} it follows from \cite[second theorem on page 144]{Waterhouse} that there exists an affine closed subgroup $H\subseteq \aut(Y[p^2])$ such that $\aut(Y)\to \aut(Y[p^2])$ factors through $H$ and $\aut(Y)\to H$ is faithfully flat. Arguing as in \cite[Lemma 4.1.5]{CaraianiScholze}, we deduce that $H$ is a finite group scheme. This completes the proof.
\end{proof}

\begin{Lem} \label{Lem:Plainification}
There is a unique morphism $Z(p^2 \mfap) \to \defsus(Y)$ making the following diagram commute
\begin{equation}
    \begin{tikzcd}
            \tildeu(\mfa) \arrow{d} \arrow[r, equals] & \tildeu(\mfa) \arrow{d} \\
            Z(p^2 \mfap) \arrow[r, densely dotted] & \defsus(Y).
     \end{tikzcd}
\end{equation}
Moreover, it is finite, faithfully flat and $\tildeu(\mfa)$-equivariant.
\end{Lem}
\begin{proof}
The free action of $\aut(Y)^{\circ}$ on $\tildeu(\mfa)$ induces a free action of $H$ on $Z(p^2 \mfap)$ and we consider the fppf quotient $Z(p^2 \mfap)/H$. Since the $H$-action preserves the presentation of $Z(p^2 \mfap)$ as the colimit of the schemes $Z^n(p^2 \mfap)$, it follows that $Z(p^2 \mfap)/H$ is representable by a formal scheme. The map $\tildeu(\mfa) \to Z(p^2 \mfap)/H$ is faithfully flat and a quasi-torsor for $\aut(Y)^{\circ}$, it is thus a torsor for $\aut(Y)^{\circ}$. It follows from the proof of Lemma \ref{Lem:UniversalProperty} that there is an isomorphism $Z(p^2 \mfap)/H \to \defsus(Y)$, proving the lemma. 
\end{proof}

\subsubsection{} For a sub-$F$-isocrystal $\mfb \subseteq \mfa$ stable under the Lie bracket we define $\mfbp = \mfb \cap \mfap$. This is a completely slope divisible and nilpotent \DLz, see Lemma \ref{Lem:SaturedSubCSD}. It follows from Lemma \ref{Lem:ExistenceLattice} that $p^2 \mfbp$ is plain. We can then consider the formal homogenous space $Z(p^2 \mfbp)$, which comes equipped with a natural monomorphism $Z(p^2 \mfbp) \to Z(p^2 \mfap)$ which is a closed immersion by Lemma \ref{Lem:ClosedImmersion}. Let us explain how to construct a formal homogeneous space associated with $\mathfrak{b}^+$ itself, despite it not being integrable as a Lie $\zpbreve$-algebra.

\begin{Cons}[$Z(\mfbp)$] Let $\aut_{\mfbp}(Y)^{\circ}$ be the intersection of $\aut(Y)^{\circ}$ with $\tildeu(\mfb)$. This is an affine group scheme containing $\Pi(p^2\mfbp)$ as a closed subgroup; in fact if $H$ is an in the statement of Lemma \ref{Lem:Inclusion}, then $\Pi(p^2\mfbp) \subseteq \aut_{\mfbp}(Y)^{\circ}$ is the kernel of the natural map $ \aut_{\mfbp}(Y)^{\circ} \to H$. If we let $H_2$ be the scheme-theoretic image of this natural map, then it follows from \cite[second theorem on page 144]{Waterhouse} that $\aut_{\mfbp}(Y)^{\circ} \to H_2$ is faithfully flat. We deduce from this and Lemma \ref{Lem:FpqcQuotient} that there is a natural action of $H_2$ on $Z(p^2 \mfbp)$. Moreover, the natural map $Z(p^2 \mfbp) \to Z(p^2 \mfap) \to Z(\mfap)$ is $H$-invariant inducing a monomorphism $Z(p^2 \mfbp) / H \to Z(\mfap)$, which is a closed immersion by Lemma \ref{Lem:ClosedImmersion}. We define $Z(\mfbp)$ to be the quotient $Z(p^2 \mfbp) / H$. 
\end{Cons}
 We are now in the position to state the following conjecture.
\begin{Conj} \label{Conj:LocalMonodromySTL}
Let $Z \xhookrightarrow{} \defsus(Y)$ be a closed immersion. For each $F$-stable Lie subalgebra $\mfb \subseteq \mfa$ there is an inclusion $Z \subseteq Z(\mfbp)$ if and only if $\operatorname{Lie} G(\calM_Z) \subseteq \mfb.$ In particular $\operatorname{Lie} G(\calM_{Z(\mfbp)})=\mfb$.
\end{Conj}
\begin{Vb}
Let the notation be as in Section \ref{Sec:IntegralModels} and consider a completely slope divisible central leaf $\cb$ lying in a $\mathbb{Q}$-non-basic Newton stratum (see Definition \ref{Def:QNonBasic}). Take a point $x \in \cb$ with corresponding $p$-divisible group $Y$, and let $\mfbp=\mathbb{D}(\mathcal{H}_Y^{G,\circ})$. It follows from Lemma \ref{Lem:FormalTorsor} and the discussions above that $Z(\mfbp) \subseteq \defsus(Y)$ can be identified with $\cb^{/x} \subseteq \defsus(Y)$ (both are the scheme-theoretic image of $\tildeu(\mfb) \to \defsus(Y)$). In this case, the unipotent radical of the monodromy group $G(\mathcal{M}_{\cb})$ is isomorphic to $U_{\nu_b}$ by \cite[Corollary 3.3.5]{OrdinaryHO}\footnote{The statement of 
\cite[Corollary 3.3.5]{OrdinaryHO} contains the assumption that \cite[Hypothesis 2.3.1]{OrdinaryHO} holds. This is true for us because $K_p$ is hyperspecial, see \cite[Lemma 2.3.2]{OrdinaryHO}.}, which uses \cites{d2020monodromy,AddezioII}, and then Theorem \ref{Thm:LocalGlobalMonodromyTheorem} allows us to conclude. In particular, we see that the monodromy group of $\mathcal{M}$ over $Z(\mfbp)$ has Lie algebra $\mfb$. \smallskip 

As a special case of this if $Y$ has height $h$ and dimension $d$, then $\defsus(Y)$ can be realised as the complete local ring of a central leaf in a PEL type unitary Shimura variety of signature $(h-d,d)$ associated to an imaginary quadratic field $\mathsf{E}$ in which $p$ splits. In particular, we know that the monodromy group of $\calM$ over $\defsus(Y)$ is isomorphic to the unipotent group corresponding to the nilpotent Lie algebra $\mathbb{D}(\mathcal{H}_{Y}^{\circ})[\tfrac{1}{p}]=\mfa$.
\end{Vb}

\begin{Cons}[$U(\mathfrak{b}^+)$]There is a $\zpbreve$-algebra structure on $\mathbb{D}(\mathcal{H}_{Y})$ and $\mfap:=\mathbb{D}(\mathcal{H}_{Y}^{\circ}) \subseteq \mathbb{D}(\mathcal{H}_{Y})$ is an algebra ideal (and Lie subalgebra). In particular $1+p^2\mfap$ is a subgroup of its group of units which defines a unipotent algebraic group over $\zpbr$. After inverting $p$, it follows from the BCH formula that we may identify $1+\mfa$ with the exponential of the Lie  $\qpbreve$-algebra $\mfa$. Given $\mfb \subseteq \mfa$, we therefore get a subgroup $1+ \mfb$ defined by the exponential of the Lie $\qpbreve$-algebra $\mfb$. We define $U(\mathfrak{b}^+)$ to be the intersection of the subgroups $1+p^2\mfap$ and $1+\mfb$, which we can identify with $1+p^2\mfbp$ where $\mfbp=\mfap \cap \mfb$. This defines a unipotent group  $U(\mathfrak{b}^+)$ over $\zpbr$. We write $U(\mfb)$ for the base change to $\qpbreve$.\end{Cons}

\section{Local monodromy of strongly Tate-linear subvarieties}
\label{Sec:MonodromyTateLinear}
\subsection{Introduction}
In this section we will prove half of Conjecture \ref{Conj:LocalMonodromySTL}. Let $Y$ be a completely slope divisible $p$-divisible group over $\ovfp$ and let $X$ be the universal $p$-divisible group over the sustained deformation space $\defsus(Y)$. Let $\mfap=\mathbb{D}(\mathcal{H}_Y^{\circ})$ be the Dieudonn\'e--Lie algebra associated to the identity component $\mathcal{H}_Y^{\circ}$ of the internal-Hom $p$-divisible group $\mathcal{H}_Y$ of $Y$, see Example \ref{Example:DieudonneLieInternalHom}. Let $\mfb \subseteq \mfa$ be an $F$-stable Lie subalgebra with associated strongly Tate-linear subspace $ Z(\mfbp) \subseteq \defsus(Y)$. Write $\mathcal{M}=\mathbb{D}(X)[\tfrac 1p]$ for the isocrystal over $\defsus(Y)$ coming from the Dieudonn\'e module of $X$.
\begin{Thm} \label{Thm:MonodromyBoundedAbove}
There is a natural closed immersion $G(\calM_{Z(\mfbp)}) \hookrightarrow U(\mfb)$. \end{Thm}
In the proof we will make use of the Cartier--Witt stacks of Drinfeld \cite{DrinfeldPrismatization}  and Bhatt--Lurie \cite{BhattLurieII} associated to quasi-syntomic schemes of characteristic $p$. Given such a scheme $X$, there is a $p$-adic formal stack $X^{\prism}$, the \emph{prismatisation} of $X$, such that coherent crystals on $X$ are the same as coherent sheaves on $X^{\prism}$. In Section \ref{Sec:CWStacks} we will give a more detailed overview of this construction and its properties.

\subsubsection{} We will now give a sketch of the proof. Consider the $\aut(Y)$-torsor
\begin{align} \label{eq:FrameTorsor}
    P:=\textbf{Isom}(X,Y_{\defsus(Y)}) \to \defsus(Y)
\end{align}
over $\defsus(Y)$. The locally free crystal $\mathcal{M}^+=\mathbb{D}(X)$ defines a vector bundle $\mathcal{V}^+$ over the prismatisation\footnote{We pretend for now that our formal schemes are actually schemes, in the actual proof there is an additional algebraisation step.} $\defsus(Y)^\prism$. This vector bundle has an associated frame bundle 
\begin{align} \label{eq:PrismaticFrameTorsor}
    \textbf{Isom}(\mathcal{V}^+, \mathbb{D}(Y)_{\defsus(Y)^{\prism}}).
\end{align}
If we apply the prismatisation functor to the map \eqref{eq:FrameTorsor}, then we get a morphism
\begin{align}
    P^{\prism} \to \defsus(Y)^{\prism},
\end{align}
which is a torsor for the $p$-adic formal group $\aut(Y)^{\prism}$ (this will follow from Lemma \ref{Lem:PrismatizationPreservesTorsors}). Dieudonn\'e theory gives us a homomorphism of group schemes over $\spf \zp$
\begin{align} \label{eq:PrismaticHomomorphism}
    \aut(Y)^{\prism} \to \aut(\mathbb{D}(Y))
\end{align}
and a morphism
\begin{align}
    P^{\prism} \to \textbf{Isom}(\mathcal{V}^+, \mathbb{D}(Y)_{\defsus(Y)^{\prism}}),
\end{align}
which is $\aut(Y)^{\prism}$-equivariant via the homomorphism \eqref{eq:PrismaticHomomorphism}. The right hand side roughly speaking parametrises all isomorphisms between $\mathcal{V}^+$ and $\mathbb{D}(Y)_{\defsus(Y)^{\prism}}$, while the left hand side parametrises those isomorphisms that are compatible with the $F$-structures.

So how does this help us? After pulling back via $Z\coloneqq Z(p^2 \mfbp) \to Z(\mfbp) \to \defsus(Y)$, the torsor \eqref{eq:FrameTorsor} has a reduction to a $\Pi(p^2 \mfbp)$-torsor by construction. Feeding this fact into the prismatisation machinery will give us a reduction of the (pullback to $Z$ of the) torsor \eqref{eq:PrismaticFrameTorsor} to a $U(\mfbp)$-torsor. If we apply the Tannakian perspective on torsors and invert $p$, then this will exactly give us a closed immersion
\begin{align}
    G(\calM_{Z}) \hookrightarrow U(\mfb),
\end{align}
which is what we want to prove once we identify $G(\calM_{Z})=G(\calM_{Z(\mfbp)})$ using Proposition \ref{Prop:FiniteFlatMonodromy}.
\subsection{Cartier--Witt stacks}\label{Sec:CWStacks}Let us briefly recall the main properties of \textit{Cartier--Witt stacks} from \cites{DrinfeldPrismatization,BhattLurieII}. We will deal only with {quasi-syntomic schemes} $X$ over $\mathbb{F}_p$, as in Definition \ref{Def:Quassyntomic}.

\subsubsection{}
Write $\nilpo\subseteq \algo{\Zp}$ for the full subcategory of $p$-nilpotent algebras with the fpqc topology. A \textit{$p$-adic formal stack} is a groupoid valued functor $\mathcal{F}$ on $\nilpo$ whose diagonal is representable in formal algebraic spaces and which admits an fpqc cover $\mathcal{X} \to \mathcal{F}$, where $\mathcal{X}$ is a $p$-adic formal algebraic space over $\Spf \Zp$ (cf. \cite[Tag 0AIM]{stacks-project}).

Drinfeld and Bhatt--Lurie define a prismatisation functor
\begin{align}
    X \mapsto X^{\prism}
\end{align}
which goes from the category of quasi-syntomic $\Fp$-schemes to the category of $p$-adic formal stacks endowed with an endomorphism $F:X^{\prism}\to X^{\prism}$, lifting the Frobenius on the special fibre\footnote{They also define a derived version of this functor, which we will not use in this text.}. 

\subsubsection{} \label{Sec:PropertiesPrismatisation}
For every quasi-syntomic scheme $X$, it follows from \cite[Proposition 8.15]{BhattLurieII} that there is an equivalence between the category of crystals in quasi-coherent $\mathcal{O}$-modules on the absolute prismatic site of $X$ (or the absolute crystalline site by \cite[Example 4.7]{BhattScholzePrismaticCrystals}) and quasi-coherent $\mathcal{O}$-modules on the Zariski site of $X^\prism$.

There are a few important properties of this functor that we will use.
\begin{itemize}
    \item[--] If $X=\operatorname{Spec} R$ is a semiperfect quasi-syntomic scheme, then $X^{\prism}$ is simply $\Spf A_{\operatorname{cris}}(R)$, the formal spectrum of Fontaine's ring of crystalline periods (this is \cite[Lemma 6.1]{BhattLurieII}). For example $(\spec \mathbb{F}_p)^{\prism}=\Spf \Zp$.
    
    \item[--] If $f:X \to Y$ is a quasi-syntomic cover, see Definition \ref{Def:Quassyntomic}, then $f:X^{\prism} \to Y^{\prism}$ is an fpqc cover (this is \cite[Proposition 7.5]{BhattLurieII}). For example, this means that $X^{\prism} \to (\spec \mathbb{F}_p)^{\prism}=\Spf \Zp$ is automatically flat if $X$ is qrsp.
    
    \item[--] The functor commutes with products and with fibre products in the case that the structure maps are flat and quasi-syntomic, by \cite[Proposition 7.5, Remark 8.9]{BhattLurieII}.

\end{itemize}
\subsection{Proof of Theorem \ref{Thm:MonodromyBoundedAbove}} \label{Sec:ProofOfMonodromyBoundedAbove}  Let the notation be as in the statement of Theorem \ref{Thm:MonodromyBoundedAbove}. At the moment, the theory of Cartier--Witt stacks is not available for general formal schemes over $\Fp$. Let us then explain how to quickly reduce ourself to an algebraic setting in our proof.

\subsubsection{Algebraisation} We need the following lemma. Let $A$ be a complete Noetherian local ring over $\ovfp$, let $D=\varprojlim D_n$ and $E=\varprojlim E_n$ be countably indexed inverse limits of finite group schemes over $\ovfp$. 
\begin{Lem} \label{Lem:Algebraisation}
Let $D \to E$ be a morphism of group schemes, let $P$ (respectively $Q$) be a $D$ (respectively $E$)-torsor over $\spf A$ for the fpqc topology and let $P \to Q$ be a $D$-equivariant map. Then there is a unique $D$ (respectively $E$)-torsor $P^{\mathrm{alg}}$ (respectively $Q^{\mathrm{alg}}$) over $\spec A$ together with a $D$-equivariant map 
$P^{\mathrm{alg}} \to Q^{\mathrm{alg}}$ over $\spec A$, recovering the given map upon base-change along $\spf A \to \spec A$.
\end{Lem}
\begin{proof}
Note that $P=\varprojlim_n P_n$, where $P_n$ is the induced $D_n$-torsor along $D \to D_n$. Since $P_n$ algebraises uniquely by \cite[Tag 01ZC]{stacks-project}, the existence and uniqueness of $P^{\mathrm{alg}}$ follows and a similar argument works for $Q$. To show that the morphism $P \to Q$ algebraises, we note that for each $n$ the morphism $D \to E \to E_n$ factors through $D \to D_m$ for some $m$ by \cite[Tag 01ZC]{stacks-project}. The induced morphisms $P_m \to Q_{n}$ algebraises uniquely by \cite[Tag 01ZC]{stacks-project}, completing the proof. 
\end{proof}

\subsubsection{} \label{subsub:Algebraisation}
The formal scheme $Z\coloneqq Z(p^2 \mfbp)$ is equal to $\spf A$ for a complete Noetherian regular local ring $A$. It follows from \cite[Proposition 2.4.8]{DeJongDieudonne} that the category of Dieudonn\'e isocrystals over $\spf A$ is equivalent to the category of Dieudonn\'e isocrystals over $\spec A$. The universal $\aut(Y)$-torsor $P \to \defsus(Y)=\spf B$ comes (uniquely) from an algebraic torsor $P^{\mathrm{alg}} \to \spec B$. Indeed, $\aut(Y)$ is an inverse limit of finite flat group schemes, and thus Lemma \ref{Lem:Algebraisation} applies. Similarly, the natural $\Pi(p^2 \mfbp)$-torsor $Q:=\tildeu(\mathfrak{b}) \to Z$ comes (uniquely) from a $\Pi(p^2 \mfbp)$-torsor $Q^{\mathrm{alg}}$ over $\spec A$. \smallskip 

Finally, we want to algebraise the natural $\Pi(p^2 \mfbp)$-equivariant map $Q=\tildeu(\mathfrak{b}) \to \tildeu(\mathfrak{a})=P$. If we consider instead the induced $\Pi(p^2 \mfbp)$-equivariant map
\begin{align}
    Q \to P \times_{\defsus(Y)} Z=:P_Z,
\end{align}
then by Lemma \ref{Lem:Algebraisation} it algebraises (uniquely) to a $\Pi(p^2 \mfbp)$-equivariant map $Q^{\mathrm{alg}} \to P^{\mathrm{alg}}$.

\begin{Not}\label{Not:Algebraisation} In Section \ref{Sec:ProofOfMonodromyBoundedAbove} we will often treat $Z$ as the affine scheme $\spec A$ rather than the formal scheme $\Spf A$ and we will simply write $Q \to Z$ for the algebraic torsor $Q^{\mathrm{alg}}$ defined above. The same applies to $P \to \defsus(Y)$ and $Z(\mfbp)$.
\end{Not}

\subsubsection{} Let $\calM^+$ be the $F$-crystal over $\defsus(Y)$ attached to the universal $p$-divisible group $X$. We write $\calM$ for the induced $F$-isocrystal and $\calM_Z$ for the pull-back of $\calM$ to $Z$ along $$\pi:Z=Z(p^2 \mfbp) \to Z(\mfbp) \to\defsus(Y),$$ and $P_Z$ for $\pi^*P.$   The basic idea of the proof is to use descent of isocrystals along $Q \to Z$ to describe $\mathcal{M}_{Q}$ as a constant isocrystal equipped with a descent datum (or equivalently a $\Pi(p^2 \mfbp)$-equivariant structure). However, it seems quite hard to compare the group scheme $\Pi(p^2 \mfbp)$ over $\ovfp$ with the monodromy group of $\calM$, which is an algebraic group over $\qpbreve$. This is where the Cartier--Witt stacks of \cite{DrinfeldPrismatization},\cite{BhattLurieII} come in. 

The Dieudonn\'e module $\mathbb{D}(Y)$ of $Y$ is a trivial vector bundle on $\ovfp^{\prism}=\Spf(\zpbreve)$ endowed with a Frobenius. We denote by $\operatorname{GL}(\mathbb{D}(Y))$ the $p$-adic formal group over $\Spf(\zpbreve)$ of $\zpbreve$-linear automorphisms of $\mathbb{D}(Y)$ (thus forgetting the $F$-structure). Let $U(\mfbp) \subseteq \operatorname{GL}(\mathbb{D}(Y))$ denote the inclusion of the $p$-adic completion of the unipotent group $U(\mfbp)$.

By the formalism of Cartier--Witt stacks, the crystal $\calM^+_Z$ corresponds to a vector bundle $\calV^+$ on $Z^{\prism}$. In turn, this defines a $\operatorname{GL}(\mathbb{D}(Y))$-torsor 
\begin{align}
    \textbf{Isom}(\calV^+, \mathbb{D}(Y)_{Z^{\prism}}) \to Z^{\prism}
\end{align}
over $Z^{\prism}$ which we denote by $\mathfrak{Q}\to Z^{\prism}$. On the other hand, $Q \to Z$ induces a $\Pi(p^2 \mfbp)^{\prism}$-torsor $Q^\prism\to Z^{\prism}$ of $p$-adic formal stacks by the following lemma.
   
\begin{Lem}\label{Lem:PrismatizationPreservesTorsors}
Let $T$ be a torsor over a quasi-syntomic scheme $S$ over $\ovfp$ under a qrsp group scheme $G$ over $\ovfp$. The prismatisation $G^\prism$ of $G$ is a formal group scheme and $T^\prism\to S^\prism$ is a $G^\prism$-torsor of formal stacks.
\end{Lem}
\begin{proof}
We want to make use of the properties of prismatisation outlined in Section \ref{Sec:PropertiesPrismatisation}. Since $G$ is qrsp, its prismatisation is a $p$-adic formal scheme. In addition, since prismatisation of $\Fp$-schemes commutes with products, it follows that $G^{\prism}$ is a formal group scheme. The fact that prismatisation sends quasi-syntomic covers to fpqc covers and commutes with fibre products when the structure maps are quasi-syntomic covers tells us that $T^{\prism} \to S^{\prism}$ is a torsor for $G^\prism$.
\end{proof}

\subsubsection{} \label{subsub:ConstructionHomomorphism} We continue with the notation from the statement of Theorem \ref{Thm:MonodromyBoundedAbove}. We may apply Lemma \ref{Lem:PrismatizationPreservesTorsors} in our situation since the group scheme $\Pi(p^2 \mfbp)$ over $\ovfp$ is qrsp by the discussion in Section \ref{Sec:Representability}. Write $\Pi(p^2 \mfbp)=\spec R$ and consider the tautological element $g_{\mathrm{univ}} \in \Pi(p^2 \mfbp)(R)$. This element corresponds to an automorphism
\begin{align}
    g_{\mathrm{univ}}: Y_R \to Y_R,
\end{align}
and this induces an automorphism of Dieudonn\'e modules
\begin{align}
    \mathbb{D}(g_{\mathrm{univ}}): \mathbb{D}({Y})\otimes_{\zpbreve} A_{\mathrm{cris}}(R)\to \mathbb{D}({Y})\otimes_{\zpbreve} A_{\mathrm{cris}}(R).
\end{align}
This corresponds precisely to a $\spf A_{\mathrm{cris}}(R)=\Pi(p^2 \mfbp)^{\prism}$-point of $\operatorname{GL}(\mathbb{D}(Y))$, in other words, it corresponds to a map
\begin{align}
    \rho: \Pi(p^2 \mfbp)^{\prism} \to \operatorname{GL}(\mathbb{D}(Y)).
\end{align}
\begin{Lem}\label{Lem:TheImageIsInsideU}
The image of $\rho$ lands in the closed subgroup $U(\mfbp) \subseteq \operatorname{GL}(\mathbb{D}(Y))$. Moreover, the morphism $\rho$ is a group homomorphism.
\end{Lem}
\begin{proof}
The definition of $\Pi(p^2 \mfbp)$ tells us that $g_{\mathrm{univ}}$ is of the form $1+p^2f$ where $f \in T_p \mbx(\mfbp)(R)$, see Lemma \ref{Inclusion}. Therefore $\mathbb{D}(g_\mathrm{univ})$ has the form $1+p^2 \mathbb{D}(f)$, thus it lies in 
\begin{align}
    1 + p^2\mfbp \otimes_{\zpbreve} {A}_{\mathrm{cris}}(R) \subseteq \operatorname{End}(\mathbb{D}({Y})\otimes_{\zpbreve} A_{\mathrm{cris}}(R)).
\end{align}
This implies that $\rho$ factors through the unipotent group associated to $\mfbp$. The second claim of the lemma is that the following diagram commutes (where the vertical maps are the multiplication maps)
\begin{equation}
    \begin{tikzcd}
    \Pi(p^2 \mfbp)^{\prism} \times \Pi(p^2 \mfbp)^{\prism} \arrow{d} \arrow{r}{\rho} & \operatorname{GL}(\mathbb{D}(Y)) \times \operatorname{GL}(\mathbb{D}(Y)) \arrow{d} \\
    \Pi(p^2 \mfbp)^{\prism} \arrow{r} & \operatorname{GL}(\mathbb{D}(Y)).
    \end{tikzcd}
\end{equation}
This is essentially a tautological consequence of the functoriality of Dieudonn\'e theory, but we spell out the proof for the benefit of the reader. For $i=1,2$ let $p_{i, \operatorname{GL}}:\operatorname{GL}(\mathbb{D}(Y)) \times \operatorname{GL}(\mathbb{D}(Y)) \to \operatorname{GL}(\mathbb{D}(Y))$ and $p_{i,\Pi}:\Pi(p^2 \mfbp)^{\prism} \times \Pi(p^2 \mfbp)^{\prism} \to \Pi(p^2 \mfbp)^{\prism}$ be the projection maps. Using the Yoneda lemma, it suffices to show the equality
\begin{align}
    p_{1, \operatorname{GL}}^{\ast} \mathbb{D}(g_{\mathrm{univ}}) \circ p_{2, \operatorname{GL}}^{\ast} \mathbb{D}(g_{\mathrm{univ}}) = \mathbb{D}(p_{1,\Pi}^{\ast} g_{\mathrm{univ}} \circ p_{2,\Pi}^{\ast} g_{\mathrm{univ}})
\end{align}
as elements of
\begin{align}
    \operatorname{GL}(\mathbb{D}(Y))\left(\Pi(p^2 \mfbp)^{\prism} \times \Pi(p^2 \mfbp)^{\prism} \right) = \mathrm{Aut}(\mathbb{D}(Y) \otimes_{\zpbreve} {A}_{\mathrm{cris}}(R \otimes_{\ovfp} R)).
\end{align}
But functoriality of Dieudonn\'e theory tells us that
\begin{align}
    \mathbb{D}(p_{1,\Pi}^{\ast} g_{\mathrm{univ}} \circ p_{2,\Pi}^{\ast} g_{\mathrm{univ}}) &= \mathbb{D}(p_{1,\Pi}^{\ast} g_{\mathrm{univ}})  \circ \mathbb{D}(p_{2,\Pi}^{\ast} g_{\mathrm{univ}}) \\ &=p_{1, \operatorname{GL}}^{\ast} \mathbb{D}(g_{\mathrm{univ}}) \circ p_{2, \operatorname{GL}}^{\ast} \mathbb{D}(g_{\mathrm{univ}}).
\end{align}
\end{proof}
\subsubsection{} \label{subsub:ConstructionReduction}
By the Yonenda lemma, there is a canonical isomorphism
\begin{align}
    h_{\text{univ}}:X_{P_Z} \xrightarrow{\sim} Y_{P_Z}
\end{align}
Applying the Dieudonn\'e module functor we get an isomorphism
\begin{align}\mathbb{D}(h_\mathrm{univ})\colon
    \mathcal{V}^{+}_{P_Z^{\prism}} \xrightarrow{\sim} \mathbb{D}(Y)_{P_Z^{\prism}},
\end{align}
which corresponds to a morphism
\begin{align}
\sigma \colon P_Z^{\prism} \to \mathfrak{Q}
\end{align}
of $p$-adic formal stacks.

\begin{Lem} \label{Lem:Reduction}
The map $\sigma$ is $\Pi(p^2 \mfbp)^{\prism}$-equivariant, where $\Pi(p^2 \mfbp)^{\prism}$ acts on $\mathfrak{Q}$ via $\rho$.
\end{Lem}
\begin{proof}
We want to show that the following diagram commutes
\begin{equation}
    \begin{tikzcd}
    \Pi(p^2 \mfbp)^{\prism} \times P_Z^{\prism} \arrow{r}{(\rho, \sigma)} \arrow{d} & \operatorname{GL}(\mathbb{D}(Y)) \times \mathfrak{Q} \arrow{d} \\
    P_Z^{\prism} \arrow{r}{\sigma} & \mathfrak{Q},
    \end{tikzcd}
\end{equation}
where the vertical maps are given by the respective action maps. It suffices to prove that the diagram commutes for the universal point $(g_{\mathrm{univ}}^{\prism}, h_{\mathrm{univ}}^{\prism})$ induced by the identity $$\Pi(p^2 \mfbp)^{\prism} \times P_Z^{\prism}\to \Pi(p^2 \mfbp)^{\prism} \times P_Z^{\prism}.$$ We write $\Pi(p^2 \mfbp)=\spec R$ and $P_Z=\spec S$. By construction, $(\rho,\sigma)$ sends $(g_{\mathrm{univ}}^{\prism}, h_{\mathrm{univ}}^{\prism})$ to $(\mathbb{D}(g_{\mathrm{univ}}), \mathbb{D}(h_{\mathrm{univ}}))$. The map to $P_Z^{\prism}$ sends instead $(g_{\mathrm{univ}}^{\prism}, h_{\mathrm{univ}}^{\prism})$ to the composition $(g_{\mathrm{univ}} \circ  h_{\mathrm{univ}})^\prism$. The commutativity of the diagram is equivalent to the equality
\begin{align}
\mathbb{D}(g_{\mathrm{univ}} \circ  h_{\mathrm{univ}}) = \mathbb{D}(g_{\mathrm{univ}}) \circ  \mathbb{D}(h_{\mathrm{univ}}),
\end{align}
which follows from functoriality of Dieudonn\'e theory.
\end{proof}

\subsubsection{} The last result we need for the proof of Theorem \ref{Thm:MonodromyBoundedAbove}, is the following special case of \cite[Proposition in Section 6.4]{Wedhorn}.
\begin{Lem} \label{Lem:RepLattice}
Let $\mathfrak{G}$ be a smooth group scheme over $\zpbreve$ with generic fibre $\mathfrak{G}_{\eta}$. Then every representation $\rho:\mathfrak{G}_{\eta} \to \operatorname{GL}(V)$, where $V$ is a finite dimensional $\qpbreve$-vector space, extends to a representation $\mathfrak{G} \to \operatorname{GL}(\Lambda)$ for some $\zpbreve$-lattice $\Lambda \subseteq V$.
\end{Lem}
\begin{proof}[Proof of Theorem \ref{Thm:MonodromyBoundedAbove}]
We recall that the conventions of Notation \ref{Not:Algebraisation} are in force. It follows from the discussion in Section \ref{subsub:Algebraisation} that we have the following commutative diagram
\begin{equation}
\begin{tikzcd}[sep=huge]
 Q \arrow[rd,swap,"\Pi(p^2\mathfrak{b}^+)"] \arrow[r] & P_Z \arrow[r]\arrow[d,"\aut(Y)"]  & P \arrow[d, "\aut(Y)"] \\
& Z \arrow[r, "\pi"]  & \defsus(Y).
\end{tikzcd}
\end{equation}
Here $P \to \defsus(Y)$ is the universal  $\aut(Y)$-torsor, $P_Z\to Z$ is the base change to $Z$, and $Q \to Z$ is a $\Pi(p^2 \mfbp)$-torsor. In Section \ref{subsub:ConstructionReduction}, we have constructed the following commutative diagram of torsors
\begin{equation}
\begin{tikzcd}[sep=huge]
 Q^\prism \arrow[rd,swap,"\Pi(p^2\mathfrak{b}^+)^\prism"] \arrow[r] &  P_Z^\prism \arrow[r,"\sigma"]\arrow[d]  & \mathfrak{Q} \arrow[ld, "\operatorname{GL}(\mathbb{D}(Y))"] \\
& Z^\prism.  & 
\end{tikzcd}
\end{equation}
By Lemma \ref{Lem:Reduction}, the induced morphism 
    \begin{align}
        \sigma:Q^{\prism} \to \mathfrak{Q},
    \end{align}
    is $\Pi(p^2 \mfbp)^{\prism}$-equivariant via the map $\rho$ of Section \ref{subsub:ConstructionHomomorphism}. Since $\rho$ factors through $U(\mfbp)$ by Lemma \ref{Lem:TheImageIsInsideU}, we get a reduction of the $\operatorname{GL}\left(\mathbb{D}(Y)\right)$-torsor $\mathfrak{Q}\to Z^\prism$ to a $U(\mfbp)$-torsor $\mathfrak{R}\to Z^\prism$ sitting between $Q^{\prism}$ and $\mathfrak{Q}$. We can associate to this the symmetric tensor functor
\begin{align} 
        \Psi:\operatorname{Rep}_{\zpbreve}(U(\mfbp))\to \operatorname{Vect}(Z^{\prism})
\end{align}
which sends $V\in \operatorname{Rep}_{\zpbreve}(U(\mfbp))$ to $$\mathfrak{R}\times^{U(\mathfrak{b^+})}(V\otimes_{\zpbreve} Z^\prism).$$ The tautological representation $U(\mfbp) \hookrightarrow \operatorname{GL}\left(\mathbb{D}(Y)\right)$ is sent by $\Psi$ to the vector bundle $\calV^+$. Applying Lemma \ref{Lem:RepLattice} and passing to isogeny categories, we get an exact tensor functor
\begin{align}
    \Psi_{\qpbreve}:\operatorname{Rep}_{\qpbreve} (U(\mfb)) \to \operatorname{Vect}(Z^{\prism})[\tfrac 1p]
\end{align}
sending the defining representation of $U(\mfb)$ to $\calV$. We can compose this with the natural inclusion
\begin{align}
    \operatorname{Vect}(Z^{\prism})[\tfrac 1p] \hookrightarrow \operatorname{Isoc}(Z)
\end{align}
and apply Tannaka duality to get a morphism of group schemes 
\begin{align}
    G(\calM_Z) \to U(\mfb).
\end{align}
This is a closed immersion because the constructed functor 
\begin{align}
    \Psi_{\qpbreve}:\operatorname{Rep}_{\qpbreve} (U(\mfb)) \to \operatorname{Isoc}(Z)
\end{align}
between Tannakian categories commutes with the $\qpbreve$-linear fibre functor obtained by restricting the objects to the closed point of $Z$ (see \cite[Proposition 2.21.(b)]{Deligne1982}). By Proposition \ref{Prop:FiniteFlatMonodromy}, there is a natural isomorphism
\begin{align}
     G(\calM_{Z(\mfbp)}) \to G(\calM_Z)
\end{align}
and thus we get a closed immersion $G(\calM_{Z(\mfbp)}) \to U(\mfb)$ as desired.
\end{proof}

\section{Rigidity} \label{Sec:Rigidity}
\subsection{Statement}

Let $(\mfap, \varphi_{\mfap},[-,-])$ be a plain \DLz \ and let $Q$ be an algebraic group over $\qp$ together with a strongly non-trivial action on $(\mfa, \varphi_{\mfa},[-,-])$ (see Definition \ref{Def:StronglyNonTrivial}), and let $\Gamma \subseteq Q(\qp)$ be a compact open subgroup preserving $\mfap$. Recall that for any $F$-stable Lie subalgebra $\mfb \subseteq \mfa$ there is a plain Dieudonn\'e--Lie $\zpbreve$-subalgebra $\mfbp \subseteq \mfap$ defined by $\mfbp=\mfb \cap \mfap$. In particular, such a $\mfb$ defines a subspace
$ Z(\mfbp) \subseteq Z(\mfap)$. The main goal of this section is to explain how the following theorem follows from \cite[Theorem 5.1]{ChaiOortRigidity}. 
\begin{Thm}[Rigidity] \label{Thm:Rigidity}
If $Z \subseteq Z(\mfap)$ is a $\Gamma$-stable integral closed formal subscheme, then there is an $F$-stable Lie subalgebra $\mfb_Z \subseteq \mfa$ such that $Z=Z(\mfbp_Z)$.
\end{Thm}
In other words, every subspace $Z \subseteq Z(\mfap)$ that is stable under a strongly non-trivial action of a $p$-adic Lie group is strongly Tate-linear. Note that, a priori, this result is more general than considering Hecke orbits on Shimura varieties. Theorem \ref{Thm:Rigidity} is essentially equivalent to \cite[Theorem 5.1]{ChaiOortRigidity}, although it is stated in a different language. In the next section, we will translate the language of Dieudonn\'e--Lie $\zpbreve$-algebras into the language of Tate unipotent groups of [\textit{ibid.}].

\subsection{Tate unipotent groups} Equip $\Pi(\mfap)$ with the filtration by normal subgroups $\operatorname{Fil}^{\bullet} \Pi(\mfap)$ given by $\operatorname{Fil}^{\lambda} N = \Pi(\mfap_{>-\lambda})$ for $\lambda \in (0,1]$. For all $i \ge 0$ there is an induced filtration by normal subgroups $\operatorname{Fil}^{\bullet} \Pi_i(\mfap)$ of $\Pi_i(\mfap)$ for each $i$. The family $$\left\{\Pi_i(\mfap),\operatorname{Fil}^{\bullet} \Pi_i(\mfap)\right\}_{i \in \mathbb{Z}_{\ge 0}}$$ is a terraced Tate unipotent group over $\fpbar$ in the sense of \cite[Definition 3.1]{ChaiOortRigidity}. In particular this means that $\Pi(\mfap)=\varprojlim_i \Pi_i(\mfap)$ is a Tate unipotent group in the sense of \cite[Definition 3.2.4]{ChaiOortRigidity}, see [\textit{ibid.}, Remark 3.2.5]. It follows from the discussion in [\textit{ibid.}, Remark 3.2.3] that we may identify the Mal'cev completion of $\Pi(\mfap)$, denoted by $\Pi(\mfap)_{\mathbb{Q}}$, with $\tildeu(\mfa)$. 

\subsubsection{} We write $N=\Pi(\mfap)$ and we consider the Tate-linear variety $\operatorname{TL}(N)$ associated to $N$, defined as the fpqc quotient $N_{\mathbb{Q}}/N$. It follows from Lemma \ref{Lem:FpqcQuotient} that we may identify $\operatorname{TL}(N)$ with $Z(\mfap)$. Given $\mfb \subseteq \mfa$ inducing $\mfbp \subseteq \mfap$, we get a map
\begin{align}
    \{\Pi_i(\mfbp),\operatorname{Fil}^{\bullet} \Pi_i(\mfbp)\}_{i \in \mathbb{Z}_{\ge 0}} \to \{\Pi_i(\mfap),\operatorname{Fil}^{\bullet} \Pi_i(\mfap)\}_{i \in \mathbb{Z}_{\ge 0}}.
\end{align}
It is straightforward to see that the first is a terraced Tate unipotent subgroup of the second, as defined in [\textit{ibid.}, Definition 3.2.10]. If we write $N'=\Pi(\mfbp)$ and $N=\Pi(\mfap)$, then we get an inclusion $N' \to N$ of Tate unipotent groups which is co-torsion free in the sense that $N'=N \cap N'_{\mathbb{Q}}$, see [\textit{ibid.}, Lemma 3.2.11]. This induces an inclusion of formal Lie varieties
\begin{align}
    \operatorname{TL}(N') \subseteq \operatorname{TL}(N)
\end{align}
which in our notation is $Z(\mfbp) \subseteq Z(\mfap)$.

\subsubsection{} If we consider $\mbx(\mfap)$ equipped with its Lie bracket, then it is a Tate unipotent Lie $\zp$-algebra in the sense of [\textit{ibid.}, Definition 3.2.13.(a)]. It follows that $\tildex(\mfa)$ is a Tate unipotent Lie $\qp$-algebra in the sense of [\textit{ibid.}, Definition 3.2.13.(b)]. We may moreover identify $\tildeu(\mfa)$ with the $\tildex(\mfa)$ equipped with the group structure coming from the BCH formula. 

\begin{proof}[Proof of Theorem \ref{Thm:Rigidity}]
    Under the assumptions of Theorem \ref{Thm:Rigidity}, it follows from [\textit{ibid.}, Theorem 5.1] that there is an inclusion $N' \subseteq \Pi(\mfap)$ of Tate unipotent groups which is co-torsion free, such that $Z=\operatorname{TL}(N')$. This induces an inclusion $N'_{\mathbb{Q}} \subseteq \tildeu(\mfa)$ of unipotent groups which corresponds to an inclusion of Tate unipotent Lie $\qp$-algebras
    \begin{align}
        \mathfrak{Lie} N'_{\mathbb{Q}} \subseteq \tildex(\mfa),
    \end{align}
where $\mathfrak{Lie}$ denotes the $\qp$-Lie algebra sheaf associated to a unipotent group.
    By [\textit{ibid.}, Lemma 3.2.19, Lemma 3.3.4], there is an $F$-stable Lie subalgebra $\mfb_{Z} \subseteq \mfa$ such that $ \mathfrak{Lie}  N'_{\mathbb{Q}}=\tildex(\mfb)$. This implies that $N'_{\mathbb{Q}}=\tildeu(\mfb_{Z})$ and  $N'=\Pi(\mfb_{Z}^{+})$, thus $Z=\operatorname{TL}(N')=Z(\mfb_{Z}^{+})$ as desired.
\end{proof}

\section{Proof of the main theorem and some variants} \label{Sec:MainTheorems}
\subsection{Preliminaries on Hecke operators}\label{Sec:PreliminariesHeckeOperators}
Let $\gx$ be a Shimura datum of Hodge type with reflex field $\mathsf{E}$ and let $p>2$ be a prime such that $G=\g_{\qp}$ is quasi-split and split over an unramified extension. Let $U_p \subseteq G(\qp)$ be a hyperspecial subgroup and let $U^{p} \subseteq \gafp$ be a sufficiently small compact open subgroup. Let $\mathbf{Sh}_{G,U}$ be the Shimura variety of level $U=U^{p} U_p$ over $\mathsf{E}$ and for a prime $v | p$ of $\mathsf{E}$ let $E=\mathsf{E}_v$ and let $\mathscr{S}_{G,U}/\mathcal{O}_{E}$ be the canonical integral model of $\mathbf{Sh}_{G,U}$ constructed in \cite{KisinModels}. 

Let $\shgu$ be the base change to $\ovfp$ of this integral canonical model for some choice of map $\mathcal{O}_{E,v} \to \ovfp$ and let 
\begin{align}
    \shginf \coloneqq \varprojlim_{K^p \subseteq \gafp}\mathrm{Sh}_{\g,K^pU_p},
\end{align}
which is equipped with an action of $\gafp$. Note that the map
\begin{align}
    \pi:\shginf \to \shgu
\end{align}
is a pro-\'etale $U^p$-torsor. \smallskip

Let $\gsc \to \gder$ be the simply-connected cover of the derived subgroup of $\g$; we will often identify groups like $\gsc(\afp)$ and $\gsc(\ql)$ with their images in $\gafp$ and $\g(\ql)$. Note that $\gsc(\afp)$ acts on $\shginf$ via the natural map $\gsc(\afp) \to \gafp$. 

Let $Z \subseteq \shgu$ be a locally closed subvariety and let $\tilde{Z}$ be the inverse image of $Z$ under $\pi$. We say that $Z$ is \emph{stable under the prime-to-$p$
Hecke operators}, or that $Z$ is \emph{$\gafp$-stable}, if $\tilde{Z}$ is $\gafp$-stable. Similarly we say that $Z$ is \emph{stable under the reduced prime-to-$p$ Hecke operators}, or that $Z$ is \emph{$\gsc(\afp)$-stable}, if $\tilde{Z}$ is $\gsc(\afp)$-stable. For $\ell \not=p$ we say that $Z$ is \emph{$\g(\ql)$-stable} if $\tilde{Z}$ is $\g(\ql)$-stable.

The \emph{prime-to-$p$ Hecke orbit} of a point $x \in \shg(\ovfp)$ is defined to be the image in $\shg(\ovfp)$ of $\gafp \cdot \tilde{x}$, for any choice of lift of $\tilde{x} \to \shginf(\ovfp)$. This does not depend on the choice of $\tilde{x}$ since it can be identified with the image in $\shg(\ovfp)$ of the $\gafp$-orbit of $\pi^{-1}(x)$. We define the \emph{reduced prime-to-$p$ Hecke orbit} of a point $x \in \shg(\ovfp)$ to be the image in $\shg(\ovfp)$ of the $\gsc(\afp)$ orbit of $\pi^{-1}(x)$. For $\ell \not=p$ we define the $\ell$-adic Hecke orbit or $\g(\ql)$-Hecke orbit of a point $x$ to be the image in $\shg(\ovfp)$ of the $\g(\ql)$ orbit of $\pi^{-1}(x)$.

\begin{Rem}
For $g \in \gafp$ and $U^p \subseteq \gafp$ there is a finite \'etale correspondence
\begin{equation}
    \begin{tikzcd}
    & \operatorname{Sh}_{\g, U_p(U^p \cap g U^p g^{-1})} \arrow{dl}[swap]{p_1} \arrow{dr}{p_2} \\
    \operatorname{Sh}_{\g,U^pU_p} & & \operatorname{Sh}_{\g, U_pg U^p g^{-1}} \arrow[r, "g"] & \operatorname{Sh}_{\g,U^pU_p}.
    \end{tikzcd}
\end{equation}
and the Hecke operator attached to $g$ is $g \circ p_2 \circ p_1^{-1}$. A locally closed subvariety $Z \subseteq \operatorname{Sh}_{\g,U^pU_p}$ is stable under the Hecke operator attached to $g$ if and only if $\tilde{Z}$ is stable under the action of $g$ considered as an element of $\gafp$.
\end{Rem}

\subsection{Local stabiliser principle} \label{Sec:LocalStabiliser}
Choose a Hodge embedding $\gx \to \gvx$ as in Section \ref{Sec:CentralLeavesHodgeType}. In particular, there is a self-dual $\mathbb{Z}_{(p)}$-lattice $V_{(p)}$ such that $K_p$ is the stabiliser in $G(\qp)$ of $V_p\coloneqq V_{(p)} \otimes_{ \mathbb{Z}_{(p)}} \zp$. Then for every sufficiently small compact open subgroup $U^p \subseteq \gafp$, we can find $\mathcal{U}^p \subseteq \g_V(\afp)$ and a closed immersion
\begin{align}
    \shg \xhookrightarrow{} \shgv
\end{align}
Fix a point $x \in \shg(\ovfp)$ such that $Y=A_x[p^{\infty}]$ is completely slope divisible. We write $\llbracket b \rrbracket \coloneqq  \llbracket b_x \rrbracket$ for the $\mathcal{G}(\zpbreve)$-$\sigma$-conjugacy class of elements of $G(\qpbreve)$ associated to $x$, and let $\cb \subseteq \shgb$ be the associated central leaf. Then we have seen that the profinite group $\aut_G(Y)(\ovfp)$ acts on $\defsus(Y)$.

For $x \in \shg(\ovfp)$ we let $I_x$ be the algebraic group over $\mathbb{Q}$ consisting of tensor-preserving self-quasi-isogenies of the abelian variety $A_x$ introduced in \cite[Section 2.1.2]{KisinPoints}. By definition it is a closed subgroup of the algebraic group $\operatorname{Aut}_{x}$ over $\mathbb{Q}$, whose $R$-points are given by
\begin{align}
    \operatorname{Aut}_{x}(R)=\left(\operatorname{End}_{\ovfp}(A_x) \otimes_{\mathbb{Z}} R \right)^{\times}.
\end{align}
        We let $I_x(\mathbb{Z}_{(p)}) \subseteq I_x(\mathbb{Q})$ be the intersection of $I_x(\mathbb{Q})$ with $\left(\operatorname{End}_{\ovfp}(A_x) \otimes_{\mathbb{Z}} \mathbb{Z}_{(p)} \right)^{\times}$. Then for a lift $\tilde{x} \in \shginf(\ovfp)$ of $x$, the stabiliser of $\tilde{x}$ in $\gafp$ is equal to $I_x(\mathbb{Z}_{(p)}) \subseteq I_x(\mathbb{Q}) \subseteq \gafp$, by \cite[Lemma 6.1.3]{OrdinaryHO}. The following result is \cite[Proposition 6.1.1]{OrdinaryHO}, see also \cite[Theorem 9.5]{ChaiOortNotes} for the Siegel case. Recall that $\cb^{/x}$ admits a closed immersion into $\defsus(Y)$.
\begin{Prop}[Local stabiliser principle] \label{Prop:LocalStabiliser}
If $Z \subseteq \cb$ is a $\gafp$-stable reduced closed subset containing $x$, then $Z^{/x} \subseteq \cb^{/x}$ is stable under the action of $I_x(\mathbb{Z}_{(p)}) \subseteq \aut_G(Y)(\ovfp)$ on $\defsus(Y)$.
\end{Prop}
\begin{Rem} \label{Rem:ReducedLocalStabiliser}
The same proof shows that for any $\gsc(\afp)$-stable reduced closed subset $Z \subseteq \cb$ containing $x$, the subscheme $Z^{/x} \subseteq \cb^{/x}$ is stable under the action of $I_x(\mathbb{Z}_{(p)}) \cap \gsc(\afp)$ on $\defsus(Y)$.
\end{Rem}

\subsection{Proof of Theorem \ref{Thm:MainThmHO}}
We keep the notation as in Section \ref{Sec:PreliminariesHeckeOperators} and Section \ref{Sec:LocalStabiliser}. Let $\gad=\g_1 \times \cdots \times \g_n$ be the decomposition of $\gad$ into a product of $\mathbb{Q}$-simple groups and write $G_i=\g_{i} \otimes \qp$ for $i=1, \cdots, n$. Recall that for a reductive group $\mathbb{G}$ over $\qp$ we denote by $B(\mathbb{G})$ the set of $\sigma$-conjugacy classes in $\mathbb{G}(\qpbreve)$.
\begin{Def}[Definition 5.3.2 of \cite{KretShin}] \label{Def:QNonBasic}
An element $[b] \in B(G_{\qp})$ is called $\mathbb{Q}$\emph{-non-basic} if the image $[b_i]$ of $[b]$ in $B(G_{i, \qp})$ is non-basic for all $i$.
\end{Def}
Let $\cb \subseteq \shgb$ be a central leaf defined as in Section \ref{Sec:CentralLeavesHodgeType} and let $\nu_b$ be the Newton cocharacter of $b$ for some $b \in [b]$, see \cite[Section 1.1.2]{KMPS} for the definition of the Newton cocharacter. Let $P_{\nu_b} \subseteq G_{\qpbreve}$ be the associated parabolic subgroup with unipotent radical $U_{\nu_b}$.
\begin{Thm} \label{Thm:MainTheorem2}
Let $Z \subseteq \cb$ be a $\gafp$-stable closed subvariety. If $[b]$ is $\mathbb{Q}$-non-basic, then $Z=\cb$.
\end{Thm}
Theorem \ref{Thm:MainTheorem2} clearly implies Theorem \ref{Thm:MainThmHO}. The discussion in \cite[Section 1.6]{KretShin} implies that the theorem is also true when $[b]$ is $\mathbb{Q}$-basic, that is, when $[b_i]$ is basic for all $i$. In this case the central leaves are finite and the claim is that the prime-to-$p$ Hecke operators act transitively on them. We expect Theorem \ref{Thm:MainTheorem2} to be true for arbitrary $[b]$, but we do not know how to prove the discrete part.
\begin{proof}
We reduce immediately to the case that $Z$ is the Zariski closure of the $\gafp$-orbit of a point. It follows from \cite[Lemma 3.1.2]{OrdinaryHO} that such a $Z$ is itself $\gafp$-stable.

By \cite[Theorem C]{KretShin}, which states that under our assumptions a $\gafp$-stable subvariety $Z \subseteq \cb$ intersects each connected component of $\cb$ non-trivially, it suffices to show that $Z$ is equidimensional of the same dimension as $\cb$. 

It follows from \cite[Proposition 2.4.5]{KimCentralLeaves} that there is a central leaf $\cbp \subseteq \shgb$ such that the universal $p$-divisible group over $\cb$ is completely slope divisible. Since $\cbp$ and $\cb$ share a $\gafp$-equivariant finite \'etale cover, it suffices to prove this equidimensionality for $\cbp$ and therefore we will assume without loss of generality that the universal $p$-divisible group over $\cb$ is completely slope divisible.

By \cite[Lemma 3.1.1]{OrdinaryHO}, the smooth locus $Z^{\mathrm{sm}}$ of $Z$ is also $\gafp$-stable. It is moreover explained in \cite[Section 3.3]{OrdinaryHO} that the abelian variety up to prime-to-$p$ isogeny over $\shgv$ induces an (overconvergent) $F$-isocrystal $\mathcal{M}$ over $\shg$. 

The assumption that $[b]$ is $\mathbb{Q}$-non-basic allows us to invoke \cite[Corollary 3.3.5]{OrdinaryHO}\footnote{The statement of 
\cite[Corollary 3.3.5]{OrdinaryHO} contains the assumption that \cite[Hypothesis 2.3.1]{OrdinaryHO} holds. This is true for us because $K_p$ is hyperspecial, see \cite[Lemma 2.3.2]{OrdinaryHO}.}, which tells us that the unipotent radical of the monodromy group of $\mathcal{M}$ over $Z^{\mathrm{sm}}$ is isomorphic the unipotent radical of $P_{\nu_b}$. Theorem \ref{Thm:MainThmMonodromy} then tells us that for $x \in Z^{\mathrm{sm}}(\ovfp)$ the monodromy of the isocrystal $\mathcal{M}$ over $\spec \widehat{\mathcal{O}}_{Z,x}$ is equal to $U_{\nu_b}$. 

The assumption that the universal $p$-divisible group over $\cb$ is completely slope divisible tells us that $\cb^{/x} \subseteq \defsus(Y)$ for $Y=A_x[p^{\infty}]$. Theorem \ref{Thm:MonodromyBoundedAbove} tells us that $Z^{/x}$ is not contained in $Z(\mfbp)$ for any $F$-stable Lie algebra $\mfb \subsetneq \mfa=\mathbb{D}(\mathcal{H}_Y^{G,\circ})[\tfrac 1p] = \operatorname{Lie} U_{\nu_b}$.

Proposition \ref{Prop:LocalStabiliser} tells us that $Z^{/x}$ is stable under the action of
\begin{align}
    I_x'(\mathbb{Z}_{(p)}) \subseteq \aut_G(Y)(\ovfp).
\end{align}
By continuity it is also stable under its closure $\Gamma \subseteq \aut_G(Y)(\ovfp)$. It follows as in the proof of \cite[Corollary 6.1.6]{OrdinaryHO} that $\Gamma$ acts strongly non-trivially on $\cb^{/x}=Z(\mfap)$. Therefore Theorem \ref{Thm:Rigidity} tells us that $Z^{/x}=Z(\mfbp)$ for some $F$-stable Lie algebra $\mfb \subseteq \mfa$\footnote{To be precise, Theorem \ref{Thm:Rigidity} states that the inverse image of $Z^{/x}$ in $Z(p^2 \mfap)$ is of the form $Z(p^2 \mfbp)$ for some $F$-stable Lie algebra $\mfb \subseteq \mfa$; this implies that $Z=Z(\mfbp)$.}, and the previous paragraph tells us that $\mfa=\mfb$. In other words, $Z^{/x}=\cb^{/x}$ for all points $x \in Z^{\mathrm{sm}}(\ovfp)$. Since $Z^{\mathrm{sm}} \subseteq Z$ is dense because $Z$ is reduced, it follows that $Z$ is equidimensional of the same dimension as $\cb$, and therefore we are done. 
\end{proof}
\subsection{Isogeny classes are dense in Newton strata}
Let $\gx$ be as above, let $x \in \shgb(\ovfp)$ and let $\mathscr{I}_x \subseteq \shgb(\ovfp)$ be the isogeny class of $x$ in the sense of \cite{KisinPoints}. Then $\mathscr{I}_x \subseteq \shgb$ for some $[b]$.
\begin{Thm} \label{Thm:IsogenyClassesDenseNewtonStrata}
If $[b]$ is $\mathbb{Q}$-non-basic, then $\mathscr{I}_x$ is dense in $\shgb$.
\end{Thm}
\begin{proof}
Write $W$ for the closure of $\mathscr{I}_x$ in $\shgb$, note that it is $\gafp$-stable since $\mathscr{I}_x$ is, see \cite[Lemma 3.1.2]{OrdinaryHO}. The isogeny class $\mathscr{I}_x \subseteq \shgb(\ovfp)$ intersects every central leaf $\cbp \subseteq \shgb$ nontrivially, by the Rapoport--Zink uniformisation of isogeny classes (which follows from the main result of \cite{KisinPoints}, see [\textit{ibid}, Section 1.4]). Thus $W$ intersects every central leaf $\cbp \subseteq \shgb$ in a $\gafp$-stable \emph{non-empty} closed subset $W_C$. Theorem \ref{Thm:MainTheorem2} now tells us that $W_C=C$ and since $\shgb$ is the (set-theoretic) union of all the central leaves it contains, it follows that $W=\shgb$.
\end{proof}
\subsection{Orthogonal Shimura varieties} \label{Sec:BraggYang} A conjecture of Bragg--Yang, see \cite[Conjecture 8.2]{BraggYang}, predicts that prime-to-$p$ Hecke orbits are Zariski dense in certain Newton strata of certain \textit{orthogonal} Shimura varieties. These are the Shimura varieties for the group $\operatorname{SO}(M)$ where $M$ is a quadratic space over $\mathbb{Q}$ with signature $(2,m-2)$; they are of abelian type by \cite[Appendix B]{Milne}. 

In general one does \emph{not} expect that prime-to-$p$ Hecke orbits are Zariski dense in Newton strata, but when the Shimura datum is fully Hodge--Newton decomposable at $p$, then $\mathbb{Q}$-non-basic Newton strata are equal to central leaves by \cite[Theorem E.(2)]{ShenYuZhang}. By \cite[Theorem D]{GoertzHeNie}, the orthogonal Shimura varieties in question are indeed fully Hodge--Newton decomposable at $p$.

It follows from the results of \cite{KretShin} and the proof of \cite[Theorem 6.0.7]{vHXiao} that
\begin{align}
    \pi_0(\shgb) \to \pi_0(\shg)
\end{align}
is a bijection for $\mathbb{Q}$-non-basic $[b]$ for Shimura varieties of Hodge type. Since Newton strata behave well with respect to the dévissage from Hodge type to abelian type, this result also holds for Shimura varieties of abelian type (see \cite[Section 5.5]{ShenZhang}). Thus \cite[Conjecture 8.2]{BraggYang} comes down to showing that the Zariski closure of prime-to-$p$ Hecke orbits have the correct dimensions, and this can be reduced to the Hodge type case and then to Theorem \ref{Thm:MainTheorem2} as in the proof of \cite[Corollary 6.4.1]{OrdinaryHO}. 
\begin{Rem}
This line of reasoning shows more generally that the Hecke orbit conjecture holds for fully Hodge--Newton decomposable Shimura varieties of abelian type, at primes $p>2$ of hyperspecial good reduction, for central leaves in Newton strata corresponding to $\mathbb{Q}$-non-basic $[b]$.
\end{Rem}

\subsection{\texorpdfstring{$\ell$}{l}-power Hecke orbits}
In this section we study the Zariski closures of $\ell$-adic Hecke orbits of points for primes $\ell \not=p$. Since the $\ell$-adic Hecke operators do not, generally, act transitively on $\pi_0(\shg)$, all we can hope to prove is that $\ell$-adic Hecke orbits are dense in a union of connected components of a central leaf. We work with primes $\ell$ such that $\g_{\ql}$ is split reductive. 
\begin{Thm} \label{Thm:LadicHeckeOrbits} 
If $\shg$ is proper and $\g_{\ql}$ is split reductive, then any $\g(\ql)$-stable reduced closed subscheme $Z \subseteq \cb$ is a union of connected components of $C$.
\end{Thm}
We start by proving a lemma, cf. \cite[Lemma 3.3.2]{ZhouMotivic}. 
\begin{Lem} \label{Lem:FiniteOrbitBasic}
Let $\ell$ be a prime such that $\g_{\ql}$ is split reductive. If $Z \subseteq \shg$ is a finite scheme that is $\g(\ql)$-stable, then $Z$ is contained in the basic locus of $\shg$.
\end{Lem}
\begin{proof}
Let $\tilde{x} \in \shginf(\ovfp)$ with image $x \in Z(\ovfp)$. Let $I_x(\mathbb{Z}_{(p)}) \subseteq \gafp$ be the group of tensor-preserving automorphisms as in Section \ref{Sec:LocalStabiliser}. Let $U_{\ell}$ be the image of $U^{p}$ under $\gafp \to \g(\ql)$ and identify $I_x(\mathbb{Z}_{(p)})$ with its image under $\gafp \to \g(\ql)$. Then the $\ell$-adic Hecke orbit of $x$ can be written as 
\begin{align}
    I_x(\mathbb{Z}_{(p)}) \backslash \g(\ql)/U_{\ell},
\end{align}
which is finite by assumption. Since the closure of $I_x(\mathbb{Z}_{(p)})$ has finite index in $I_{x}(\ql)$, it follows that
\begin{align}
    I_x(\ql) \backslash \g(\ql)
\end{align}
is compact. Since $I_{x,\ql}$ is connected it follows from of \cite[Propositions 8.4, Proposition 9.3]{BorelTits} that it contains a maximal trigonalizable subgroup of $\g_{\ql}$. Since $\g_{\ql}$ is split, $I_{x,\ql}$ must contain a Borel subgroup of $\g_{\ql}$, and because it is reductive it follows that $I_{x,\ql}=\g_{\ql}$. It is well known that this only happens when $x$ is contained in the basic locus.
\end{proof}
\begin{Lem} \label{Lem:LAdicPrimeToPAdic}
Assume that $\shg$ is proper. If a reduced closed subscheme $Z \subseteq \shgb$ is stable under the action of $\g(\ql)$ for some $\ell \not=p$ such that $\gsc \otimes \ql$ is split reductive, then $Z$ is stable under the action of $\gsc(\afp)$.
\end{Lem}
\begin{proof}
The proof is almost exactly the same as the proof of \cite[Proposition 4.6]{ChaiElladicmonodromy}. Nevertheless, we will give a complete proof for the benefit of the reader.

\textbf{Step 1.} A standard argument (see e.g. \cite[The proof of Proposition 3.3.1]{ZhouMotivic}) using the quasi-affineness of the Ekedahl--Oort stratification (\cite[Corollary I.2.6]{GoldringKoskivirta}) and the properness of $\shg$ shows that the Zariski closure $\overline{Z}$ of $Z$ in $\shg$ contains a point $x \in \shg(\ovfp)$ with finite $\ell$-power Hecke orbit. In fact, this argument shows that for any irreducible component $V$ of $Z$ the closure $\overline{V}$ contains a point with a finite $\ell$-power orbit. 

It follows from Lemma \ref{Lem:FiniteOrbitBasic} that $x$ is contained in the basic locus of $\shg$, and moreover that $I_x$ is an inner form of $\g$ (see \cite[Corollary 5.2.11]{HeZhouZhu}). Thus the isomorphism $I_x \otimes \afp \to \g \otimes \afp$ induces an isomorphism $I_x^{\mathrm{sc}} \otimes \afp \to \gsc \otimes \afp$. Strong approximation away from $\infty, \ell$ (see \cite[Theorem 7.8]{PR}), using the fact that $\g_{\ql}$ is split, tells us that the image of
\begin{align}
    I_x^{\mathrm{sc}}(\mathbb{Z}_{(p)}) \to \gsc(\mathbb{A}_f^{p,\ell})
\end{align}
is dense. Since the inclusion of the $\ell$-adic Hecke orbit of $x$ inside the prime-to-$p$ Hecke orbit of $x$ can be identified with
\begin{align}
    I_x(\mathbb{Z}_{(p)}) \backslash \g(\ql)U^p/U^p \subseteq I_x(\mathbb{Z}_{(p)}) \backslash \gafp/U^p,
\end{align}
we see that the $\ell$-adic Hecke orbit of $x$ is $\gsc(\afp)$-stable. Moreover the local stabiliser principle, see Remark \ref{Rem:ReducedLocalStabiliser}, tells us that $$\overline{Z}^{/x} \subseteq \cb^{/x}$$ is kept stable under the action of
\begin{align}
    I_x(\mathbb{Z}_{(p)}) \cap \gsc(\afp) \subseteq \aut_G(Y)(\ovfp)
\end{align}
and by continuity  it is also kept stable by its closure. 

\textbf{Step 2.} Take a $\lambda$-adic Hecke operator $g_{\lambda} \in \gsc(\mathbb{Q}_{\lambda})$ for $\lambda \not=p,\ell$ and let $W$ be the image of $\overline{Z}$ under $g_{\lambda}$. Since every irreducible component of $\overline{Z}$ contains a point $x$ with finite $\ell$-power Hecke orbit, it follows that every irreducible component of $W$ contains an element in the $g_{\lambda}$-orbit of such an $x$. Since $\overline{Z}$ contains the $\gsc(\afp)$ orbit of $x$, it follows that every irreducible component $W_i$ of $W$ intersects $\overline{Z}$ in a point $y_i$ with finite $\ell$-power Hecke orbit. 

The reduced prime-to-$p$ Hecke orbit of $y_i$ has the form
\begin{align}
    I_{y_i}^{\mathrm{sc}}(\mathbb{Z}_{(p)}) \backslash \gsc(\mathbb{A}_f^{p}) U^{p} / U^{p}.
\end{align}
By strong approximation (\cite[Theorem 7.8]{PR}) away from $\ell$ for $I_{y_i}^{\mathrm{sc}}$, using the fact that $\g_{\ql}$ is split, we can choose $\delta \in I_{y_i}^{\mathrm{sc}}(\mathbb{Q})$ which lands in $U^{p,\ell,\lambda}$ and such that there is an element $g_{\ell} \in \gsc(\ql)$ such that $\delta \cdot g_{\ell} = g_{\lambda}^{-1}$ in $g_{\lambda}$ in $I_{y_i}(\mathbb{Z}_{(p)}) \backslash \gsc(\mathbb{A}_f^{p}) U^{p} / U^{p}$. 

We see that $g_{\lambda} \circ g_{\ell}$ fixes $y_i$, and since $\overline{Z}$ is $g_{\ell}$-stable, the image of $\overline{Z}$ under $g_{\lambda} \circ g_{\ell}$ is equal to $W$. Now we consider the closed subschemes
\begin{align}
    W_i^{/y_i}, \overline{Z}^{/y_i} \subseteq \shg^{/y_i}.
\end{align}
The subscheme $W_i^{/y_i}$ is the image of $\overline{Z}^{/y_i}$ under the Hecke operator $g_{\lambda} \circ g_{\ell}$. Since $\overline{Z}^{/y_i}$ is stable under the action of the closure of $I_x(\mathbb{Z}_{(p)})$ and since $g_{\lambda} \circ g_{\ell}$ is contained in that closure by construction, it follows that
\begin{align}
    W_i^{/y_i} \subseteq \overline{Z}^{/y_i}.
\end{align}
From this we deduce that $W_i \subseteq \overline{Z}$. Thus every irreducible component $W_i$ of $W$ is contained in $\overline{Z}$, and we conclude that $\overline{Z}$ is stable under the action of $g_{\lambda}$. Since $\lambda$ and $g_{\lambda}$ were arbitrary, it follows that $\overline{Z}$ is stable under the action of $\gsc(\afp)$.

We know that $Z$ is the intersection of $\overline{Z}$ with $\shgb$ and as the intersection of two $\gsc(\afp)$-stable subschemes it must itself be $\gsc(\afp)$-stable.
\end{proof}
\begin{proof}[Proof of Theorem \ref{Thm:LadicHeckeOrbits}]
If we let $\mathcal{M}$ be the isocrystal attached to the universal abelian variety up to prime-to-$p$ isogeny over $Z$ and if we let $x \in Z$ be a smooth point, then arguing as in the proof of Theorem \ref{Thm:MainTheorem2}, we can combine \cite[Corollary 3.3.5]{OrdinaryHO} with Theorem \ref{Thm:MainThmMonodromy} to deduce that the monodromy of the isocrystal $\mathcal{M}$ over $\spec \widehat{\mathcal{O}}_{Z,x}$ is isomorphic to $U_{\nu_b}$. \smallskip

Proposition \ref{Prop:LocalStabiliser} (see Remark \ref{Rem:ReducedLocalStabiliser}) tells us that $Z^{/x}$ is stable under the action of
\begin{align}
    \left(I_x(\mathbb{Z}_{(p)}) \cap \gsc(\afp) \right) \subseteq \aut_G(Y)(\ovfp).
\end{align}
By continuity, it is also stable under the closure $\Gamma \subseteq \aut_G(Y)(\ovfp)$. As in the proof of \cite[Corollary 6.1.6]{OrdinaryHO}, it follows that $\Gamma$ acts strongly non-trivially on $\cb^{/x}=Z(\mfap)$. The same argument as in the proof of Theorem \ref{Thm:MainTheorem2} allows us to conclude that $Z$ is a union of connected components of $C$. \end{proof}

\subsection{Further questions of Chai--Oort} Let $Z \subseteq \cgsp$ be an irreducible smooth closed subvariety and let $x \in Z(\ovfp)$. We call $Z$ \emph{strongly Tate-linear at $x$} if $Z^{/x} \subseteq \cgsp^{/x}$ is a strongly Tate-linear subvariety. 
\begin{Question}
Suppose that $Z$ is strongly Tate-linear at some closed point $z_0 \in Z(\ovfp)$. Is $Z$ then strongly Tate-linear at all closed points $z \in Z(\ovfp)$?
\end{Question}
It follows from Theorem \ref{Thm:MainThmMonodromy} that the monodromy group of $\mathcal{M}$ over $\spec \widehat{\mathcal{O}}_{Z,x}$ does not depend on $x$. Now the validity of Conjecture \ref{Conj:LocalMonodromySTL} would imply that the question above has an affirmative answer. 
\begin{Question} \label{Q:Lifting}
Suppose that $Z$ is strongly Tate-linear at some closed point $z_0 \in Z(\ovfp)$. Is $Z$ an irreducible component of a central leaf in the mod $p$ reduction of a Shimura variety of Hodge type?
\end{Question}
The stronger assertion that $Z$ must itself be an irreducible component of a Shimura variety of Hodge type is false in general, because only finitely many central leaves in a given Newton stratum contain the mod $p$ reductions of special points by \cite[Theorem 1.3]{KisinLamShankarSrinivasan}. If $\cgsp$ is the ordinary locus and $Z$ is proper, then results towards this stronger assertion are proved in work of Moonen \cite{MoonenLifting}.
\subsection{Results at ramified primes and parahoric level} The statements of our theorems make sense for central leaves in special fibres of the Kisin--Pappas \cite{KisinPappas} integral models of Shimura varieties of parahoric level at tamely ramified primes $p>2$. 

\subsubsection{} If $G$ is unramified over $\qp$ and the level is parahoric, then the Hecke orbit conjecture for central leaves at parahoric level follows immediately from Theorem \ref{Thm:MainTheorem2}. The main observation is that the forgetful map
\begin{align}
    C' \to C, 
\end{align}
where $C'$ is a central leaf at Iwahori level and $C$ is a central leaf at hyperspecial level, is equivariant for the prime-to-$p$ Hecke operators and induces a bijection on $\pi_0$. This last statement can be proven using the surjectivity\footnote{\label{FN:RZ} This surjectivity is axiom 4c of the He--Rapoport axioms, see \cite{ShenYuZhang}, and follows from Rapoport--Zink uniformisation at parahoric level, which is \cite[Theorem 2]{vH}.} of $C' \to C$, and the explicit description of connected components of Igusa varieties in \cites{vHXiao, KretShin}. 

Rapoport--Zink uniformisation of isogeny classes at parahoric level\footref{FN:RZ} implies as before that isogeny classes are dense in the Newton strata containing them.

\subsubsection{} If $G$ is a ramified group over $\qp$, then it is not always true that $\gafp$ acts transitively on $\pi_0(\shg)$ (see \cite{OkiComponents} for explicit counterexamples), and therefore it is not necessarily true that $\gafp$ acts transitively on $\pi_0(C)$ either because $\pi_0(C) \to \pi_0(\shg)$ is surjective. Nevertheless, we expect that the continuous part of the Hecke orbit conjecture is true for ramified groups. In fact, we suspect that the strategy adopted in this paper can be made to work for ramified groups.



\DeclareRobustCommand{\VAN}[3]{#3}
\bibliographystyle{amsalpha}
\bibliography{references}
\end{document}